\documentclass[english, sn-mathphys-num]{sn-jnl}%

\usepackage[yyyymmdd]{datetime} 

\usepackage{enumerate} 
\usepackage{enumitem} 

\usepackage{pgfplots} 
\pgfplotsset{compat=1.18}

\usepackage{tikz} 
\usetikzlibrary{calc}
\tikzset{myspherestyle/.style={ball color=blue!10!white, opacity=0.5}}

\usepackage{subcaption} 

\def\r{1.8cm}

\newcommand\irregularcircle[2]{
  \pgfextra {\pgfmathsetmacro\len{(#1)+rand*(#2)}}
  +(0:\len pt)
  \foreach \a in {40,80,...,320}{
    \pgfextra {\pgfmathsetmacro\len{(#1)+rand*(#2)}}
    -- +(\a:\len pt)
  } -- cycle
}

\usepackage{siunitx}

\usepackage[notrig]{physics} 
\usepackage{braket} 

\usepackage{amsmath}
\usepackage{amsthm} 
\usepackage{amsfonts}
\usepackage{amssymb}
\usepackage{mathrsfs} 

\usepackage{cleveref} 

\newtheorem{theorem}{Theorem}
\newtheorem{proposition}[theorem]{Proposition}%
\newtheorem{corollary}[theorem]{Corollary}%
\newtheorem{lemma}[theorem]{Lemma}%
\newtheorem{remark}{Remark}%

\numberwithin{equation}{section}

\usepackage{algorithm}
\usepackage{algorithmic}

\usepackage{mleftright} 
\mleftright 

\newcommand*\mathhat{\widehat}

\let\oldexp\exp
\newcommand*{\Exp}[1]{\oldexp#1}
\renewcommand{\exp}[1]{\mathrm{e}^{#1}}

\renewcommand{\innerproduct}[2]{\left<{#1},{#2}\right>}

\newcommand*{\reals}{\mathbb{R}}

\newcommand*\from{\colon}

\newcommand*{\transpose}[1]{{#1}^{\!\mathsf{T}}}
\newcommand*{\inverse}[1]{{#1}^{-1}}

\newcommand*{\new}[1]{{#1}}

\usepackage{csquotes} 

\usepackage{orcidlink}

\begin{document}

\linespread{1.0}
\title[Approximating maps into manifolds]{%
	Approximating maps into manifolds with lower curvature bounds%
}
\author*[1]{\fnm{Simon} \sur{Jacobsson}\orcidlink{0000-0002-1181-972X}}
\email{simon.jacobsson@kuleuven.be}
\author[1]{\fnm{Raf} \sur{Vandebril}\orcidlink{0000-0003-2119-8696}}
\author[2]{\fnm{Joeri} \sur{van der Veken}\orcidlink{0000-0003-0521-625X}}
\author[1,3]{\fnm{Nick} \sur{Vannieuwenhoven}\orcidlink{0000-0001-5692-4163}}
\affil[1]{\orgdiv{Department of Computer Science}, \orgname{KU Leuven}, \orgaddress{\street{Celestijnenlaan 200A - box 2402}, \city{Leuven}, \postcode{3000}, \country{Belgium}}}
\affil[2]{\orgdiv{Department of Mathematics}, \orgname{KU Leuven}, \orgaddress{\street{Celestijnenlaan 200A - box 2400}, \city{Leuven}, \postcode{3000}, \country{Belgium}}}
\affil[3]{\orgdiv{KU Leuven Institute for AI}, \city{Leuven}, \postcode{3000}, \country{Belgium}}

\abstract{%
	Many interesting functions arising in applications map into Riemannian manifolds.
	We present an algorithm, using the manifold exponential and logarithm, for approximating such functions.
	Our approach extends approximation techniques for functions into linear spaces in such a way that we can upper bound the forward error in terms of a lower bound on the manifold's sectional curvature.
	Furthermore, when the sectional curvature is nonnegative, such as for compact Lie groups, the error is guaranteed to not be worse than in the linear case.

	We implement the algorithm in a Julia package \texttt{ManiFactor.jl} and apply it to two example problems.
}

\keywords{%
	Riemannian manifold,
	manifold exponential,
	manifold logarithm,
	function approximation,
	constructive approximation.
}

\maketitle


\pacs[MSC Classification]{15A69, 53Z99, 65D15}

\section{Introduction}%
\label{sec:Introduction}

Approximation theory is concerned with approximating functions with simpler functions.
Classic approximation theory studies schemes for approximating real-valued functions.
For example, univariate schemes like interpolation or regression~\cite{Christensen04,Trefethen13}, but also multivariate schemes like hyperbolic crosses, sparse grids, greedy approximation, and moving least squares~\cite{Temlyakov18,Wendland04}.
More recently, multivariate function approximation methods based on tensor decompositions have been extensively investigated~\cite{BM2002,BM2005,BEM2016,HT2017,HN2018,GKM2019,Dolgov2021,Strossner2022,Strossner2023}.
All of these schemes are naturally extended to approximate vector-valued functions by applying them component-wise.

Many interesting functions arising in applications have some further structure.
However, naively applying aforementioned classic approximation schemes fails to preserve such structures.
To remedy this, we introduce a scheme to approximate maps whose values are points on a \emph{Riemannian manifold}.
Our scheme works by pulling back the approximation problem to \new{a tangent space}.
In this way, the problem is reduced again to the classic function approximation problem between vector spaces.
Consequently, our scheme can leverage all of the previously mentioned classic techniques.

\new{%
Some examples from applications of structured data that can be described as points on a Riemannian manifold are 
orthogonal matrices from motion tracking, matrix decompositions, differential equations, statistics, and optimization~\cite{Zimmermann20,Zimmermann22,Chakraborty19,Absil08,Benner15};
symmetric positive definite matrices and tensors from diffusion tensor imaging, materials modeling, computer vision, and statistics~\cite{Cheng16,Thanwerdas21,Thanwerdas23,Pennec20a,Moakher06,Pennec20b};
fixed-rank matrices and tensors from compression, tensor completion, and dynamic low-rank approximation~\cite{Oseledets2010,Oseledets2011,Swijsen2021,Vandereycken12,Koch07};
linear subspaces from Krylov subspace methods and subspace tracking~\cite{Soodhalter20,Zimmermann22b,Benner15,Yang95};
and sequences of nested linear subspaces from principal component analysis~\cite{Manikovich24,Buet23}.
}%
For instance, the present work was conceived to address difficulties encountered in the context of reduced order models of a microstructure materials engineering problem in~\cite{DVLM2022}.

\subsection{Contribution}%
\label{sub:Contribution}

We propose to approximate a map $f \from R \to M$, where $R$ is a \new{set} and $M$ is a Riemannian manifold, by a natural, generally applicable three-step template:
\begin{enumerate}
	\item Choose a point $p \in M$ from which to linearize the approximation problem.%
	\label{item:step1}
	\item Pull $f$ back to the tangent space $T_p M$ using a \emph{normal coordinate chart}~\cite{Lee18}, and approximate the resulting map $g := {\log_p} \circ f \from R \to T_p M$ by $\mathhat{g} \from R \to T_p M$ using an arbitrary approximation scheme for maps into vector spaces.
	\label{item:step2}
	\item Push $\mathhat{g}$ forward with $\operatorname{exp}_p$ to yield $\mathhat{f} := {\operatorname{exp}_p} \circ \mathhat{g} \from R \to M$.%
	\label{item:step3}
\end{enumerate}
While the proposed template is \new{a classic approach}, the main challenge is establishing a bound on the approximation error, as any discrepancy between $g$ and $\mathhat{g}$ will be propagated through the exponential map.
This requires understanding how the endpoint of a geodesic, which is a solution of a particular initial value problem on a manifold, changes as the initial conditions of this differential equation are varied.
Our main contribution, \cref{thm:manifold_error_bound}, consists of showing that, in this framework, an exact error analysis in principle only requires knowledge of a lower bound on the sectional curvature.

Alternatively, if $\operatorname{exp}_p$ and $\operatorname{log}_p$ are not known explicitly or are expensive to compute, a \emph{retraction}~\cite{Absil08,Boumal2023} may be used instead.
\new{%
In \cref{cor:retraction}, we show how the error analysis can be extended to this case as well.
}

\Cref{alg:manifold_approximation_scheme} is a detailed description of a concrete algorithm that results from applying this template to maps with domain $R = [-1,1]^m$. 
We implement it in a Julia package \texttt{ManiFactor.jl}.
\Cref{fig:map_into_sphere} is a small example of approximating the map $f \from [-1, 1]^{2} \to S^{2}$, defined by stereographic projection of $(x, y) \mapsto (x^{2} - y^{2}, 2 x y)$.

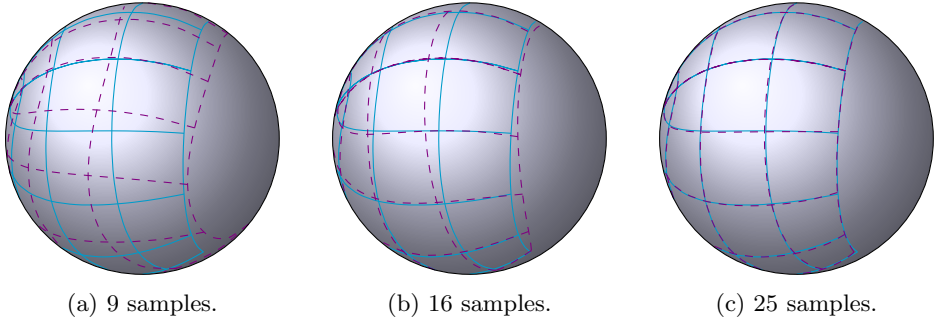
\begin{figure}[t]
	\tikzset{f/.style={cyan!80!black}} 
	\tikzset{fhat/.style={violet, dashed}} 
	\centering
	\begin{subfigure}{0.32\textwidth}
		\centering
		\begin{tikzpicture}
		    \node (O) at (0, 0) {}; 
		
			\shade [myspherestyle] (O) circle (\r);
			
			\draw [f] plot [smooth] coordinates {
			(0.35461990849812774 * \r, 0.49768891227433143 * \r)
			(0.15480834419220604 * \r, 0.5512774312461516 * \r)
			(-0.06541160934793494 * \r, 0.5819005074323589 * \r)
			(-0.28657125486494506 * \r, 0.5834507603275669 * \r)
			(-0.48671233758953414 * \r, 0.5563270843617839 * \r)
			(-0.6492930574882867 * \r, 0.5080066837973016 * \r)
			(-0.7681725941637781 * \r, 0.4501008376053993 * \r)
			(-0.8468128471549871 * \r, 0.39418444669223673 * \r)
			(-0.8937165501261771 * \r, 0.3490669367256721 * \r)
			(-0.9177194132994653 * \r, 0.3201975495656527 * \r)
			(-0.9249960830190598 * \r, 0.31030682175729835 * \r)
			(-0.9177194132994653 * \r, 0.3201975495656527 * \r)
			(-0.8937165501261771 * \r, 0.3490669367256721 * \r)
			(-0.8468128471549871 * \r, 0.39418444669223673 * \r)
			(-0.7681725941637781 * \r, 0.4501008376053993 * \r)
			(-0.6492930574882867 * \r, 0.5080066837973016 * \r)
			(-0.48671233758953414 * \r, 0.5563270843617839 * \r)
			(-0.28657125486494506 * \r, 0.5834507603275669 * \r)
			(-0.06541160934793494 * \r, 0.5819005074323589 * \r)
			(0.15480834419220604 * \r, 0.5512774312461516 * \r)
			(0.35461990849812774 * \r, 0.49768891227433143 * \r)
			}; \draw [f] plot [smooth] coordinates {
			(0.3042263876867834 * \r, 0.03908975864860598 * \r)
			(0.1058830296102599 * \r, 0.05063188668897822 * \r)
			(-0.11255611269639901 * \r, 0.05627089090515669 * \r)
			(-0.33324619461168503 * \r, 0.05680439515294622 * \r)
			(-0.5352267273744729 * \r, 0.055215102418578854 * \r)
			(-0.701423867353378 * \r, 0.05617684786182475 * \r)
			(-0.8238582401645557 * \r, 0.06480658542670553 * \r)
			(-0.9039028511927052 * \r, 0.08549061080900322 * \r)
			(-0.9485659966221153 * \r, 0.12138579212550502 * \r)
			(-0.9659400052572956 * \r, 0.17450640740528664 * \r)
			(-0.9619526779836589 * \r, 0.24592293523475672 * \r)
			(-0.9387866010184561 * \r, 0.3356489599184813 * \r)
			(-0.8945413345878729 * \r, 0.441996492068909 * \r)
			(-0.8238755970619736 * \r, 0.5604143399966105 * \r)
			(-0.7198240647801982 * \r, 0.6821999856273893 * \r)
			(-0.5770856270351098 * \r, 0.7940652261899315 * \r)
			(-0.39625692902234794 * \r, 0.8799348256521664 * \r)
			(-0.18685891989908457 * \r, 0.9255433162528415 * \r)
			(0.033656221238566736 * \r, 0.9239716216109956 * \r)
			(0.24536594919043359 * \r, 0.878396740607072 * \r)
			(0.43244306852874065 * \r, 0.7999948386860298 * \r)
			}; \draw [f] plot [smooth] coordinates {
			(0.3041088258523764 * \r, -0.4089404095123532 * \r)
			(0.1251830144566918 * \r, -0.44210226215649395 * \r)
			(-0.07481289637301547 * \r, -0.47051298734524666 * \r)
			(-0.28387411289597453 * \r, -0.48762821292919006 * \r)
			(-0.48579945818620973 * \r, -0.48655647128291857 * \r)
			(-0.6641821670909152 * \r, -0.46177413916498977 * \r)
			(-0.8068832653648643 * \r, -0.41040834809421456 * \r)
			(-0.9084199484591851 * \r, -0.3323716143153083 * \r)
			(-0.9692247326317776 * \r, -0.22949768583838578 * \r)
			}; \draw [f] plot [smooth] coordinates {
			(0.35861666010137755 * \r, -0.7094901291822779 * \r)
			(0.216561361065778 * \r, -0.7680047907343848 * \r)
			(0.05722249002053528 * \r, -0.8197573890814775 * \r)
			(-0.11308477221691995 * \r, -0.8570690995780968 * \r)
			(-0.2847666570358103 * \r, -0.8715455471533294 * \r)
			(-0.44617038323884395 * \r, -0.8556552563207178 * \r)
			(-0.5856665494756039 * \r, -0.8043673942387144 * \r)
			}; \draw [f] plot [smooth] coordinates {
			(0.4458131325319566 * \r, -0.8341423252737478 * \r)
			(0.3459455325115729 * \r, -0.8912848597210525 * \r)
			(0.23601717764322283 * \r, -0.9416910645479524 * \r)
			(0.11947753291676869 * \r, -0.9795043186589512 * \r)
			(0.0015474583860723146 * \r, -0.9982289374418344 * \r)
			};
			
			\draw [f] plot [smooth] coordinates {
			(-0.9734502512412145 * \r, 0.21645603838080008 * \r)
			(-0.9503650978095842 * \r, 0.2695639155955872 * \r)
			(-0.9319027913015903 * \r, 0.30029198363127285 * \r)
			(-0.9249960830190598 * \r, 0.31030682175729835 * \r)
			(-0.9319027913015903 * \r, 0.30029198363127285 * \r)
			(-0.9503650978095842 * \r, 0.2695639155955872 * \r)
			(-0.9734502512412145 * \r, 0.21645603838080008 * \r)
			}; \draw [f] plot [smooth] coordinates {
			(-0.8933051772026108 * \r, -0.4473276762951144 * \r)
			(-0.9437264382844831 * \r, -0.2838810566103164 * \r)
			(-0.9556361731823627 * \r, -0.12013626344817188 * \r)
			(-0.9420244751112437 * \r, 0.03198907553846231 * \r)
			(-0.9169820338659903 * \r, 0.16571209105209736 * \r)
			(-0.8918592116441181 * \r, 0.27847459619059345 * \r)
			(-0.8736426618596116 * \r, 0.36986103012019333 * \r)
			(-0.8647058074052786 * \r, 0.4396171487037881 * \r)
			(-0.8629573718317479 * \r, 0.48632279099550135 * \r)
			}; \draw [f] plot [smooth] coordinates {
			(-0.24887452699923499 * \r, -0.9670707889480283 * \r)
			(-0.3835049655974334 * \r, -0.903744784202541 * \r)
			(-0.5038502773817606 * \r, -0.7963522732699487 * \r)
			(-0.5995176709543579 * \r, -0.6461535279119586 * \r)
			(-0.6641821670909152 * \r, -0.46177413916498977 * \r)
			(-0.6975919657678042 * \r, -0.25702421281576293 * \r)
			(-0.7049807754233284 * \r, -0.04685628278467713 * \r)
			(-0.6944523655637698 * \r, 0.15612973442767675 * \r)
			(-0.6739836278079241 * \r, 0.34304158902523585 * \r)
			(-0.6492930574882867 * \r, 0.5080066837973016 * \r)
			(-0.6228481288947172 * \r, 0.6465064785996874 * \r)
			(-0.5937379321355996 * \r, 0.7538220434576267 * \r)
			(-0.5581747733444302 * \r, 0.8243675862804476 * \r)
			}; \draw [f] plot [smooth] coordinates {
			(0.1782878076528404 * \r, -0.9625714375805758 * \r)
			(0.09481954969973533 * \r, -0.9388625125517499 * \r)
			(0.01227895333316184 * \r, -0.8828330916905951 * \r)
			(-0.06429365917220373 * \r, -0.7885935419515656 * \r)
			(-0.12958504925707726 * \r, -0.6537569692151357 * \r)
			(-0.17916245213437132 * \r, -0.48092478119651527 * \r)
			(-0.21054453359558659 * \r, -0.2779049380157696 * \r)
			(-0.2236637488749284 * \r, -0.056422853605624324 * \r)
			(-0.22048857606976363 * \r, 0.1700503906008594 * \r)
			(-0.2040374734996862 * \r, 0.38829697279872016 * \r)
			(-0.17728124890467267 * \r, 0.5864632063576696 * \r)
			(-0.14236321065241325 * \r, 0.7543044189429402 * \r)
			(-0.10033008382571201 * \r, 0.88313395007027 * \r)
			(-0.05135858529677241 * \r, 0.9661264233087006 * \r)
			(0.004667165987406435 * \r, 0.9992742435139426 * \r)
			}; \draw [f] plot [smooth] coordinates {
			(0.4458131325319566 * \r, -0.8341423252737478 * \r)
			(0.4089728542436598 * \r, -0.8028295912922291 * \r)
			(0.37442043768152067 * \r, -0.7474499452734952 * \r)
			(0.3441365849139677 * \r, -0.6642080915108389 * \r)
			(0.32012794889533924 * \r, -0.5511193026080723 * \r)
			(0.3041088258523764 * \r, -0.4089404095123532 * \r)
			(0.2971824087693057 * \r, -0.24174672979927664 * \r)
			(0.29963128776713077 * \r, -0.056904273644859726 * \r)
			(0.3108980531376256 * \r, 0.13560887907953117 * \r)
			(0.3297597056854791 * \r, 0.3243442221059205 * \r)
			(0.35461990849812774 * \r, 0.49768891227433143 * \r)
			(0.3838087712166306 * \r, 0.6450997452084247 * \r)
			(0.41580247020092237 * \r, 0.7581667501588631 * \r)
			(0.44932899437946006 * \r, 0.8314804878642604 * \r)
			(0.48337701531684873 * \r, 0.863229791720793 * \r)
			};
			
			\draw [fhat] plot [smooth] coordinates {
			(0.47927771104756633 * \r, 0.4204564748630433 * \r)
			(0.1813945447059405 * \r, 0.5397220407416308 * \r)
			(-0.10369701702203826 * \r, 0.592957929941559 * \r)
			(-0.34974973421218547 * \r, 0.593694346636251 * \r)
			(-0.545604168970977 * \r, 0.5589063381056432 * \r)
			(-0.6910302317159831 * \r, 0.5049559587422567 * \r)
			(-0.792260437465486 * \r, 0.4454510256537357 * \r)
			(-0.8582395036634926 * \r, 0.39056106750105496 * \r)
			(-0.8979331257915091 * \r, 0.34721859016543777 * \r)
			(-0.918629840921787 * \r, 0.31971051248242643 * \r)
			(-0.9249960830190599 * \r, 0.3103068217572973 * \r)
			}; \draw [fhat] plot [smooth] coordinates {
			(0.38161379619900404 * \r, 0.07862520723323531 * \r)
			(0.07803088243191615 * \r, 0.15933883982124475 * \r)
			(-0.205746406346256 * \r, 0.20209999857778246 * \r)
			(-0.44584813628460906 * \r, 0.2154491126448478 * \r)
			(-0.6331650062464116 * \r, 0.21037626807872323 * \r)
			(-0.7689694926814542 * \r, 0.197573017807344 * \r)
			(-0.8604282667627332 * \r, 0.1860131097715584 * \r)
			(-0.9169343193474782 * \r, 0.18252224829940727 * \r)
			(-0.9475325121477751 * \r, 0.19192126737363008 * \r)
			(-0.9593338194160526 * \r, 0.21738445719528668 * \r)
			(-0.956664180284561 * \r, 0.2607556946536388 * \r)
			(-0.9406954503015034 * \r, 0.32266066403818433 * \r)
			(-0.9093899592039757 * \r, 0.4023287744784985 * \r)
			(-0.8577101771541621 * \r, 0.4971003338744247 * \r)
			(-0.7781719715411466 * \r, 0.6016615565418277 * \r)
			(-0.6619302922365717 * \r, 0.7071433827930639 * \r)
			(-0.5006479070247067 * \r, 0.8003550880565748 * \r)
			(-0.2893602549153718 * \r, 0.8635969692747467 * \r)
			(-0.030342273610653225 * \r, 0.8756690724015043 * \r)
			(0.26246495181407614 * \r, 0.814771838072429 * \r)
			(0.5596212591737388 * \r, 0.6638243424887111 * \r)
			}; \draw [fhat] plot [smooth] coordinates {
			(0.3302549204901353 * \r, -0.3380046768602053 * \r)
			(0.020071851005824336 * \r, -0.30723328657104443 * \r)
			(-0.26859218131988094 * \r, -0.2826258562314842 * \r)
			(-0.511463800640301 * \r, -0.26179135836465517 * \r)
			(-0.6992839981174169 * \r, -0.2406648760082472 * \r)
			(-0.8333795041823973 * \r, -0.21493966328875047 * \r)
			(-0.9210929375757995 * \r, -0.18081163553892604 * \r)
			(-0.9720332525212909 * \r, -0.13521108166480023 * \r)
			(-0.995439847626628 * \r, -0.07575928282174366 * \r)
			}; \draw [fhat] plot [smooth] coordinates { 
			(0.03939560202091008 * \r, 0.9990391067676854 * \r)
			(0.3327278843634063 * \r, 0.9390634695131548 * \r)
			(0.6222896183760183 * \r, 0.7730807026536496 * \r)
			}; \draw [fhat] plot [smooth] coordinates {
			(0.44883624351767926 * \r, -0.689559278659502 * \r)
			(0.15578864293569164 * \r, -0.7398386590467241 * \r)
			(-0.12666962290867434 * \r, -0.7616251509956314 * \r)
			(-0.3694368400632184 * \r, -0.7580041469002933 * \r)
			(-0.5592404258043677 * \r, -0.7319550356132061 * \r)
			(-0.6946366493268553 * \r, -0.6861743650204732 * \r)
			(-0.7813985971311831 * \r, -0.622914791523637 * \r)
			}; \draw [fhat] plot [smooth] coordinates {
			(0.7758621700749558 * \r, -0.6293292406430886 * \r)
			(0.5788847209327349 * \r, -0.8008721947169262 * \r)
			(0.35964236096350216 * \r, -0.9175234381684894 * \r)
			(0.1554080084147471 * \r, -0.981693060290067 * \r)
			};
			
			\draw [fhat] plot [smooth] coordinates {
			(-0.9671827927254403 * \r, 0.23768456591443227 * \r)
			(-0.9464953384884737 * \r, 0.27907352340786573 * \r)
			(-0.9307774335523005 * \r, 0.30266446999466107 * \r)
			(-0.9249960830190599 * \r, 0.3103068217572973 * \r)
			(-0.9307774335523005 * \r, 0.302664469994661 * \r)
			(-0.9464953384884737 * \r, 0.27907352340786573 * \r)
			(-0.9671827927254403 * \r, 0.2376845659144322 * \r)
			}; \draw [fhat] plot [smooth] coordinates {
			(-0.9665783852352087 * \r, -0.2563678543658543 * \r)
			(-0.9866654028824842 * \r, -0.1073414832129454 * \r)
			(-0.9759599123740386 * \r, 0.024116734577775634 * \r)
			(-0.9498810809275359 * \r, 0.13633372938821992 * \r)
			(-0.9204185105391839 * \r, 0.22994694710128133 * \r)
			(-0.8958444102639365 * \r, 0.3064528448836388 * \r)
			(-0.8808630700072878 * \r, 0.3671242994338634 * \r)
			(-0.8768493397193343 * \r, 0.4122722699286095 * \r)
			(-0.8819773043152976 * \r, 0.4408099911531401 * \r)
			(-0.8911822267775997 * \r, 0.4501083777985409 * \r)
			}; \draw [fhat] plot [smooth] coordinates {
			(-0.620563184575695 * \r, -0.7713426628214844 * \r)
			(-0.7508094416627282 * \r, -0.5947915333583111 * \r)
			(-0.816377437203666 * \r, -0.40387414904199537 * \r)
			(-0.8333795041823968 * \r, -0.21493966328875097 * \r)
			(-0.8186900886536974 * \r, -0.03763759260215116 * \r)
			(-0.7872155307944114 * \r, 0.12351219092205121 * \r)
			(-0.7506225954425807 * \r, 0.26731066247488533 * \r)
			(-0.7170070223614338 * \r, 0.39421257555878486 * \r)
			(-0.6910302317159831 * \r, 0.5049559587422567 * \r)
			(-0.674179906887686 * \r, 0.5995154046371126 * \r)
			(-0.6649552025123273 * \r, 0.6763629180216009 * \r)
			(-0.6589152550995118 * \r, 0.7320374344087037 * \r)
			(-0.6486594664262864 * \r, 0.7610783030011569 * \r)
			}; \draw [fhat] plot [smooth] coordinates {
			(0.25393863505393194 * \r, -0.9558255344998174 * \r)
			(0.0071585176173880605 * \r, -0.9435513070741341 * \r)
			(-0.1831702754946848 * \r, -0.8393512099818102 * \r)
			(-0.3093280432554883 * \r, -0.6750098510811064 * \r)
			(-0.3764718758348421 * \r, -0.4787225214885763 * \r)
			(-0.39663491378815713 * \r, -0.2719721839876701 * \r)
			(-0.384017647606862 * \r, -0.06903175947420748 * \r)
			(-0.3519447962895268 * \r, 0.12199634431850762 * \r)
			(-0.3112832655058834 * \r, 0.2974738471111025 * \r)
			(-0.26988540160761165 * \r, 0.45641374932812473 * \r)
			(-0.2326107441326451 * \r, 0.5988507367624822 * \r)
			(-0.20157545668762633 * \r, 0.7245039576527059 * \r)
			(-0.17640797909949663 * \r, 0.8317542358496433 * \r)
			(-0.1544172797180734 * \r, 0.9169469824957727 * \r)
			(-0.13069776733715888 * \r, 0.9740740944001123 * \r)
			}; \draw [fhat] plot [smooth] coordinates {
			(0.7758621700749588 * \r, -0.6293292406430855 * \r)
			(0.6309819382819835 * \r, -0.7299586377500763 * \r)
			(0.5013936935719756 * \r, -0.7242446267061193 * \r)
			(0.4060338082522868 * \r, -0.6385199578519238 * \r)
			(0.35020669848505803 * \r, -0.5014359830967872 * \r)
			(0.3302549204901369 * \r, -0.338004676860206 * \r)
			(0.3380353334567161 * \r, -0.16697245103544384 * \r)
			(0.36426630851724046 * \r, -0.000560279503670047 * \r)
			(0.400563534519467 * \r, 0.15445043250329504 * \r)
			(0.4403891924698695 * \r, 0.2950796197714636 * \r)
			(0.47927771104756667 * \r, 0.420456474863043 * \r)
			(0.514680391387928 * \r, 0.5303252725491908 * \r)
			(0.5456803576127625 * \r, 0.6238438554348996 * \r)
			(0.5727226506815938 * \r, 0.6987016929048178 * \r)
			(0.5974052125458165 * \r, 0.7505870085518442 * \r)
			(0.6222896183760183 * \r, 0.7730807026536496 * \r)
			};
		
		    \draw (O) circle (\r);
		\end{tikzpicture}
		\caption{$9$ samples.}
	\end{subfigure}
	\begin{subfigure}{0.32\textwidth}
		\centering
		\begin{tikzpicture}
		    \node (O) at (0, 0) {}; 
		
			\shade [myspherestyle] (O) circle (\r);
		
			\draw [f] plot [smooth] coordinates {
			(0.35461990849812774 * \r, 0.49768891227433143 * \r)
			(0.15480834419220604 * \r, 0.5512774312461516 * \r)
			(-0.06541160934793494 * \r, 0.5819005074323589 * \r)
			(-0.28657125486494506 * \r, 0.5834507603275669 * \r)
			(-0.48671233758953414 * \r, 0.5563270843617839 * \r)
			(-0.6492930574882867 * \r, 0.5080066837973016 * \r)
			(-0.7681725941637781 * \r, 0.4501008376053993 * \r)
			(-0.8468128471549871 * \r, 0.39418444669223673 * \r)
			(-0.8937165501261771 * \r, 0.3490669367256721 * \r)
			(-0.9177194132994653 * \r, 0.3201975495656527 * \r)
			(-0.9249960830190598 * \r, 0.31030682175729835 * \r)
			(-0.9177194132994653 * \r, 0.3201975495656527 * \r)
			(-0.8937165501261771 * \r, 0.3490669367256721 * \r)
			(-0.8468128471549871 * \r, 0.39418444669223673 * \r)
			(-0.7681725941637781 * \r, 0.4501008376053993 * \r)
			(-0.6492930574882867 * \r, 0.5080066837973016 * \r)
			(-0.48671233758953414 * \r, 0.5563270843617839 * \r)
			(-0.28657125486494506 * \r, 0.5834507603275669 * \r)
			(-0.06541160934793494 * \r, 0.5819005074323589 * \r)
			(0.15480834419220604 * \r, 0.5512774312461516 * \r)
			(0.35461990849812774 * \r, 0.49768891227433143 * \r)
			}; \draw [f] plot [smooth] coordinates {
			(0.3042263876867834 * \r, 0.03908975864860598 * \r)
			(0.1058830296102599 * \r, 0.05063188668897822 * \r)
			(-0.11255611269639901 * \r, 0.05627089090515669 * \r)
			(-0.33324619461168503 * \r, 0.05680439515294622 * \r)
			(-0.5352267273744729 * \r, 0.055215102418578854 * \r)
			(-0.701423867353378 * \r, 0.05617684786182475 * \r)
			(-0.8238582401645557 * \r, 0.06480658542670553 * \r)
			(-0.9039028511927052 * \r, 0.08549061080900322 * \r)
			(-0.9485659966221153 * \r, 0.12138579212550502 * \r)
			(-0.9659400052572956 * \r, 0.17450640740528664 * \r)
			(-0.9619526779836589 * \r, 0.24592293523475672 * \r)
			(-0.9387866010184561 * \r, 0.3356489599184813 * \r)
			(-0.8945413345878729 * \r, 0.441996492068909 * \r)
			(-0.8238755970619736 * \r, 0.5604143399966105 * \r)
			(-0.7198240647801982 * \r, 0.6821999856273893 * \r)
			(-0.5770856270351098 * \r, 0.7940652261899315 * \r)
			(-0.39625692902234794 * \r, 0.8799348256521664 * \r)
			(-0.18685891989908457 * \r, 0.9255433162528415 * \r)
			(0.033656221238566736 * \r, 0.9239716216109956 * \r)
			(0.24536594919043359 * \r, 0.878396740607072 * \r)
			(0.43244306852874065 * \r, 0.7999948386860298 * \r)
			}; \draw [f] plot [smooth] coordinates {
			(0.3041088258523764 * \r, -0.4089404095123532 * \r)
			(0.1251830144566918 * \r, -0.44210226215649395 * \r)
			(-0.07481289637301547 * \r, -0.47051298734524666 * \r)
			(-0.28387411289597453 * \r, -0.48762821292919006 * \r)
			(-0.48579945818620973 * \r, -0.48655647128291857 * \r)
			(-0.6641821670909152 * \r, -0.46177413916498977 * \r)
			(-0.8068832653648643 * \r, -0.41040834809421456 * \r)
			(-0.9084199484591851 * \r, -0.3323716143153083 * \r)
			(-0.9692247326317776 * \r, -0.22949768583838578 * \r)
			}; \draw [f] plot [smooth] coordinates {
			(0.35861666010137755 * \r, -0.7094901291822779 * \r)
			(0.216561361065778 * \r, -0.7680047907343848 * \r)
			(0.05722249002053528 * \r, -0.8197573890814775 * \r)
			(-0.11308477221691995 * \r, -0.8570690995780968 * \r)
			(-0.2847666570358103 * \r, -0.8715455471533294 * \r)
			(-0.44617038323884395 * \r, -0.8556552563207178 * \r)
			(-0.5856665494756039 * \r, -0.8043673942387144 * \r)
			}; \draw [f] plot [smooth] coordinates {
			(0.4458131325319566 * \r, -0.8341423252737478 * \r)
			(0.3459455325115729 * \r, -0.8912848597210525 * \r)
			(0.23601717764322283 * \r, -0.9416910645479524 * \r)
			(0.11947753291676869 * \r, -0.9795043186589512 * \r)
			(0.0015474583860723146 * \r, -0.9982289374418344 * \r)
			};
			
			\draw [f] plot [smooth] coordinates {
			(-0.9734502512412145 * \r, 0.21645603838080008 * \r)
			(-0.9503650978095842 * \r, 0.2695639155955872 * \r)
			(-0.9319027913015903 * \r, 0.30029198363127285 * \r)
			(-0.9249960830190598 * \r, 0.31030682175729835 * \r)
			(-0.9319027913015903 * \r, 0.30029198363127285 * \r)
			(-0.9503650978095842 * \r, 0.2695639155955872 * \r)
			(-0.9734502512412145 * \r, 0.21645603838080008 * \r)
			}; \draw [f] plot [smooth] coordinates {
			(-0.8933051772026108 * \r, -0.4473276762951144 * \r)
			(-0.9437264382844831 * \r, -0.2838810566103164 * \r)
			(-0.9556361731823627 * \r, -0.12013626344817188 * \r)
			(-0.9420244751112437 * \r, 0.03198907553846231 * \r)
			(-0.9169820338659903 * \r, 0.16571209105209736 * \r)
			(-0.8918592116441181 * \r, 0.27847459619059345 * \r)
			(-0.8736426618596116 * \r, 0.36986103012019333 * \r)
			(-0.8647058074052786 * \r, 0.4396171487037881 * \r)
			(-0.8629573718317479 * \r, 0.48632279099550135 * \r)
			}; \draw [f] plot [smooth] coordinates {
			(-0.24887452699923499 * \r, -0.9670707889480283 * \r)
			(-0.3835049655974334 * \r, -0.903744784202541 * \r)
			(-0.5038502773817606 * \r, -0.7963522732699487 * \r)
			(-0.5995176709543579 * \r, -0.6461535279119586 * \r)
			(-0.6641821670909152 * \r, -0.46177413916498977 * \r)
			(-0.6975919657678042 * \r, -0.25702421281576293 * \r)
			(-0.7049807754233284 * \r, -0.04685628278467713 * \r)
			(-0.6944523655637698 * \r, 0.15612973442767675 * \r)
			(-0.6739836278079241 * \r, 0.34304158902523585 * \r)
			(-0.6492930574882867 * \r, 0.5080066837973016 * \r)
			(-0.6228481288947172 * \r, 0.6465064785996874 * \r)
			(-0.5937379321355996 * \r, 0.7538220434576267 * \r)
			(-0.5581747733444302 * \r, 0.8243675862804476 * \r)
			}; \draw [f] plot [smooth] coordinates {
			(0.1782878076528404 * \r, -0.9625714375805758 * \r)
			(0.09481954969973533 * \r, -0.9388625125517499 * \r)
			(0.01227895333316184 * \r, -0.8828330916905951 * \r)
			(-0.06429365917220373 * \r, -0.7885935419515656 * \r)
			(-0.12958504925707726 * \r, -0.6537569692151357 * \r)
			(-0.17916245213437132 * \r, -0.48092478119651527 * \r)
			(-0.21054453359558659 * \r, -0.2779049380157696 * \r)
			(-0.2236637488749284 * \r, -0.056422853605624324 * \r)
			(-0.22048857606976363 * \r, 0.1700503906008594 * \r)
			(-0.2040374734996862 * \r, 0.38829697279872016 * \r)
			(-0.17728124890467267 * \r, 0.5864632063576696 * \r)
			(-0.14236321065241325 * \r, 0.7543044189429402 * \r)
			(-0.10033008382571201 * \r, 0.88313395007027 * \r)
			(-0.05135858529677241 * \r, 0.9661264233087006 * \r)
			(0.004667165987406435 * \r, 0.9992742435139426 * \r)
			}; \draw [f] plot [smooth] coordinates {
			(0.4458131325319566 * \r, -0.8341423252737478 * \r)
			(0.4089728542436598 * \r, -0.8028295912922291 * \r)
			(0.37442043768152067 * \r, -0.7474499452734952 * \r)
			(0.3441365849139677 * \r, -0.6642080915108389 * \r)
			(0.32012794889533924 * \r, -0.5511193026080723 * \r)
			(0.3041088258523764 * \r, -0.4089404095123532 * \r)
			(0.2971824087693057 * \r, -0.24174672979927664 * \r)
			(0.29963128776713077 * \r, -0.056904273644859726 * \r)
			(0.3108980531376256 * \r, 0.13560887907953117 * \r)
			(0.3297597056854791 * \r, 0.3243442221059205 * \r)
			(0.35461990849812774 * \r, 0.49768891227433143 * \r)
			(0.3838087712166306 * \r, 0.6450997452084247 * \r)
			(0.41580247020092237 * \r, 0.7581667501588631 * \r)
			(0.44932899437946006 * \r, 0.8314804878642604 * \r)
			(0.48337701531684873 * \r, 0.863229791720793 * \r)
			};
			
			\draw [fhat] plot [smooth] coordinates {
			(0.3740913799521763 * \r, 0.4710689588060385 * \r)
			(0.08658818749438835 * \r, 0.5624100401542647 * \r)
			(-0.17984463025902286 * \r, 0.5930313215060905 * \r)
			(-0.4057523258483986 * \r, 0.5775968837133927 * \r)
			(-0.5839685616337768 * \r, 0.5324748156745527 * \r)
			(-0.7159146191813496 * \r, 0.47266530005926166 * \r)
			(-0.8079125131941934 * \r, 0.4103096169368865 * \r)
			(-0.8681808081956575 * \r, 0.35433890218344116 * \r)
			(-0.9047041373996477 * \r, 0.3107844524791296 * \r)
			(-0.9238907984174164 * \r, 0.28335559207124167 * \r)
			(-0.9298214686145898 * \r, 0.2740107240727113 * \r)
			(-0.9238907984174164 * \r, 0.28335559207124167 * \r)
			(-0.9047041373996477 * \r, 0.31078445247912956 * \r)
			(-0.8681808081956573 * \r, 0.3543389021834411 * \r)
			(-0.8079125131941934 * \r, 0.41030961693688633 * \r)
			(-0.7159146191813494 * \r, 0.47266530005926144 * \r)
			(-0.5839685616337768 * \r, 0.5324748156745522 * \r)
			(-0.40575232584839827 * \r, 0.5775968837133925 * \r)
			(-0.17984463025902292 * \r, 0.5930313215060901 * \r)
			(0.08658818749438868 * \r, 0.562410040154264 * \r)
			(0.37409137995217634 * \r, 0.471068958806038 * \r)
			}; \draw [fhat] plot [smooth] coordinates {
			(0.3538012164124909 * \r, 0.014067852568600936 * \r)
			(0.06353069631759359 * \r, 0.0512592446807267 * \r)
			(-0.20370735390954262 * \r, 0.05822310757111568 * \r)
			(-0.4296686735523381 * \r, 0.04938703275050649 * \r)
			(-0.6088129556285553 * \r, 0.03721326068758507 * \r)
			(-0.7433977903706314 * \r, 0.031020897219370736 * \r)
			(-0.8394238533460471 * \r, 0.03704786183068681 * \r)
			(-0.9037944774911112 * \r, 0.05903794021928466 * \r)
			(-0.9425806770989338 * \r, 0.09889934829844765 * \r)
			(-0.9601161948974054 * \r, 0.15720144356870785 * \r)
			(-0.9586476915163188 * \r, 0.2334228386148942 * \r)
			(-0.9383353172964104 * \r, 0.3259481523810402 * \r)
			(-0.8974835074838028 * \r, 0.4318555599427611 * \r)
			(-0.8329553793078007 * \r, 0.5465648449109536 * \r)
			(-0.740777083220288 * \r, 0.6634407674700824 * \r)
			(-0.6169676692146004 * \r, 0.7734767702516068 * \r)
			(-0.4586328988753771 * \r, 0.8652197069964304 * \r)
			(-0.2653335474380561 * \r, 0.9251303435184939 * \r)
			(-0.040673828097347875 * \r, 0.9385915614733612 * \r)
			(0.20605221403641638 * \r, 0.8917522693932138 * \r)
			(0.45846966794387034 * \r, 0.7742978243258004 * \r)
			}; \draw [fhat] plot [smooth] coordinates {
			(0.38806253016670966 * \r, -0.4090039804150277 * \r)
			(0.12478227573722034 * \r, -0.43518504281502135 * \r)
			(-0.12334570890076182 * \r, -0.4616303199507321 * \r)
			(-0.3434048112170245 * \r, -0.47606105559939443 * \r)
			(-0.5314431401980602 * \r, -0.47037927594855694 * \r)
			(-0.6874733146197283 * \r, -0.4399410889125639 * \r)
			(-0.8123883829818266 * \r, -0.3826598598852806 * \r)
			(-0.9064460185583224 * \r, -0.2983995231918955 * \r)
			(-0.9688379474034818 * \r, -0.18871435170568707 * \r)
			(-0.9979326678333909 * \r, -0.056812681542963706 * \r)
			}; \draw [fhat] plot [smooth] coordinates { 
			(0.5629564061711463 * \r, 0.8264700003192975 * \r)
			}; \draw [fhat] plot [smooth] coordinates {
			(0.43282830212611867 * \r, -0.6841244927283157 * \r)
			(0.22361076404842417 * \r, -0.7596074139432417 * \r)
			(0.02605839184977063 * \r, -0.8186232588082516 * \r)
			(-0.15754315704490046 * \r, -0.8531145873278406 * \r)
			(-0.32844154646426127 * \r, -0.8581339950973536 * \r)
			(-0.4864877320053647 * \r, -0.8300588515137411 * \r)
			(-0.6281260285917338 * \r, -0.7661989793302796 * \r)
			}; \draw [fhat] plot [smooth] coordinates {
			(0.4526084585115908 * \r, -0.8249866570583231 * \r)
			(0.3352964545157068 * \r, -0.9014128298368487 * \r)
			(0.23968917641131854 * \r, -0.9532259154206235 * \r)
			(0.15174484671561767 * \r, -0.9845992316433761 * \r)
			};
			
			\draw [fhat] plot [smooth] coordinates {
			(-0.9860515540228791 * \r, 0.16328216478214164 * \r)
			(-0.9685699829525785 * \r, 0.21443196716263163 * \r)
			(-0.949266488226039 * \r, 0.2484585545550785 * \r)
			(-0.9350117719157399 * \r, 0.2677682415606809 * \r)
			(-0.9298214686145898 * \r, 0.2740107240727113 * \r)
			(-0.93501177191574 * \r, 0.2677682415606808 * \r)
			(-0.9492664882260391 * \r, 0.2484585545550782 * \r)
			(-0.9685699829525786 * \r, 0.21443196716263124 * \r)
			(-0.9860515540228791 * \r, 0.16328216478214103 * \r)
			}; \draw [fhat] plot [smooth] coordinates {
			(-0.9135049559999202 * \r, -0.38897542419049086 * \r)
			(-0.9416888315865963 * \r, -0.2465736456028486 * \r)
			(-0.9446933338190164 * \r, -0.11032072141919719 * \r)
			(-0.9338881564179919 * \r, 0.016929897170139263 * \r)
			(-0.9177378164039888 * \r, 0.1336363751897507 * \r)
			(-0.90180026316402 * \r, 0.23868524545588116 * \r)
			(-0.8889619645661127 * \r, 0.33073062297702405 * \r)
			(-0.879717296320549 * \r, 0.4077519352218826 * \r)
			(-0.8724066666488728 * \r, 0.4668276490274501 * \r)
			(-0.8634072843891247 * \r, 0.5041377294178065 * \r)
			}; \draw [fhat] plot [smooth] coordinates {
			(-0.2521374694235695 * \r, -0.967164146375153 * \r)
			(-0.41959589238317335 * \r, -0.8867030658564091 * \r)
			(-0.5432097402704762 * \r, -0.7645727746063432 * \r)
			(-0.6299643950051431 * \r, -0.6124799548446647 * \r)
			(-0.6874733146197283 * \r, -0.43994108891256356 * \r)
			(-0.7224351356710133 * \r, -0.2548652963063601 * \r)
			(-0.7399943933701603 * \r, -0.06425828266516198 * \r)
			(-0.7436337408621903 * \r, 0.12515783741190933 * \r)
			(-0.7353432210544046 * \r, 0.3065221344817861 * \r)
			(-0.7159146191813494 * \r, 0.47266530005926144 * \r)
			(-0.685279497225061 * \r, 0.6162081704133522 * \r)
			(-0.6428503638538879 * \r, 0.7299097674900334 * \r)
			(-0.5878438959103854 * \r, 0.8072140832274499 * \r)
			}; \draw [fhat] plot [smooth] coordinates {
			(0.19556045280636414 * \r, -0.9712224017781473 * \r)
			(0.07659312884008201 * \r, -0.950931083248545 * \r)
			(-0.023437146197580805 * \r, -0.8865738769914475 * \r)
			(-0.10783200404033952 * \r, -0.7825641522532498 * \r)
			(-0.17884171323065853 * \r, -0.6426357433101583 * \r)
			(-0.23727269311154553 * \r, -0.4709458519673967 * \r)
			(-0.28257850479830515 * \r, -0.2731100380558884 * \r)
			(-0.31322622309247883 * \r, -0.056929792002328755 * \r)
			(-0.32722154801689407 * \r, 0.1673693677983813 * \r)
			(-0.32270864905065727 * \r, 0.3874957497046644 * \r)
			(-0.2985590462200922 * \r, 0.590043418539118 * \r)
			(-0.25485407870839744 * \r, 0.7619050294079233 * \r)
			(-0.1931687051793975 * \r, 0.8918891920161442 * \r)
			(-0.1165917928634409 * \r, 0.9722110794571308 * \r)
			}; \draw [fhat] plot [smooth] coordinates {
			(0.4526084585115897 * \r, -0.8249866570583235 * \r)
			(0.4486528040667601 * \r, -0.7806452176293388 * \r)
			(0.4393987887228389 * \r, -0.7210519572902017 * \r)
			(0.42513673557218795 * \r, -0.6416847430139384 * \r)
			(0.4072722900586033 * \r, -0.5384244323601216 * \r)
			(0.3880625301667095 * \r, -0.40900398041502783 * \r)
			(0.3703976714189348 * \r, -0.2542441872737781 * \r)
			(0.3574895724484635 * \r, -0.07883005382403196 * \r)
			(0.3524155696231388 * \r, 0.10862262838765326 * \r)
			(0.3575622773185057 * \r, 0.296357208597251 * \r)
			(0.37409137995217634 * \r, 0.471068958806038 * \r)
			(0.40158128049447606 * \r, 0.6200104620594616 * \r)
			(0.43796909911426696 * \r, 0.7331373363669472 * \r)
			(0.4798333986842207 * \r, 0.8047758502796658 * \r)
			(0.5229492747907767 * \r, 0.8344145147426859 * \r)
			(0.5629564061711463 * \r, 0.8264700003192975 * \r)
			};
		
		    \draw (O) circle (\r);
		\end{tikzpicture}
		\caption{$16$ samples.}
	\end{subfigure}
	\begin{subfigure}{0.32\textwidth}
		\centering
		\begin{tikzpicture}
		    \node (O) at (0, 0) {}; 
		
			\shade [myspherestyle] (O) circle (\r);
			
			\draw [f] plot [smooth] coordinates {
			(0.35461990849812774 * \r, 0.49768891227433143 * \r)
			(0.15480834419220604 * \r, 0.5512774312461516 * \r)
			(-0.06541160934793494 * \r, 0.5819005074323589 * \r)
			(-0.28657125486494506 * \r, 0.5834507603275669 * \r)
			(-0.48671233758953414 * \r, 0.5563270843617839 * \r)
			(-0.6492930574882867 * \r, 0.5080066837973016 * \r)
			(-0.7681725941637781 * \r, 0.4501008376053993 * \r)
			(-0.8468128471549871 * \r, 0.39418444669223673 * \r)
			(-0.8937165501261771 * \r, 0.3490669367256721 * \r)
			(-0.9177194132994653 * \r, 0.3201975495656527 * \r)
			(-0.9249960830190598 * \r, 0.31030682175729835 * \r)
			(-0.9177194132994653 * \r, 0.3201975495656527 * \r)
			(-0.8937165501261771 * \r, 0.3490669367256721 * \r)
			(-0.8468128471549871 * \r, 0.39418444669223673 * \r)
			(-0.7681725941637781 * \r, 0.4501008376053993 * \r)
			(-0.6492930574882867 * \r, 0.5080066837973016 * \r)
			(-0.48671233758953414 * \r, 0.5563270843617839 * \r)
			(-0.28657125486494506 * \r, 0.5834507603275669 * \r)
			(-0.06541160934793494 * \r, 0.5819005074323589 * \r)
			(0.15480834419220604 * \r, 0.5512774312461516 * \r)
			(0.35461990849812774 * \r, 0.49768891227433143 * \r)
			}; \draw [f] plot [smooth] coordinates {
			(0.3042263876867834 * \r, 0.03908975864860598 * \r)
			(0.1058830296102599 * \r, 0.05063188668897822 * \r)
			(-0.11255611269639901 * \r, 0.05627089090515669 * \r)
			(-0.33324619461168503 * \r, 0.05680439515294622 * \r)
			(-0.5352267273744729 * \r, 0.055215102418578854 * \r)
			(-0.701423867353378 * \r, 0.05617684786182475 * \r)
			(-0.8238582401645557 * \r, 0.06480658542670553 * \r)
			(-0.9039028511927052 * \r, 0.08549061080900322 * \r)
			(-0.9485659966221153 * \r, 0.12138579212550502 * \r)
			(-0.9659400052572956 * \r, 0.17450640740528664 * \r)
			(-0.9619526779836589 * \r, 0.24592293523475672 * \r)
			(-0.9387866010184561 * \r, 0.3356489599184813 * \r)
			(-0.8945413345878729 * \r, 0.441996492068909 * \r)
			(-0.8238755970619736 * \r, 0.5604143399966105 * \r)
			(-0.7198240647801982 * \r, 0.6821999856273893 * \r)
			(-0.5770856270351098 * \r, 0.7940652261899315 * \r)
			(-0.39625692902234794 * \r, 0.8799348256521664 * \r)
			(-0.18685891989908457 * \r, 0.9255433162528415 * \r)
			(0.033656221238566736 * \r, 0.9239716216109956 * \r)
			(0.24536594919043359 * \r, 0.878396740607072 * \r)
			(0.43244306852874065 * \r, 0.7999948386860298 * \r)
			}; \draw [f] plot [smooth] coordinates {
			(0.3041088258523764 * \r, -0.4089404095123532 * \r)
			(0.1251830144566918 * \r, -0.44210226215649395 * \r)
			(-0.07481289637301547 * \r, -0.47051298734524666 * \r)
			(-0.28387411289597453 * \r, -0.48762821292919006 * \r)
			(-0.48579945818620973 * \r, -0.48655647128291857 * \r)
			(-0.6641821670909152 * \r, -0.46177413916498977 * \r)
			(-0.8068832653648643 * \r, -0.41040834809421456 * \r)
			(-0.9084199484591851 * \r, -0.3323716143153083 * \r)
			(-0.9692247326317776 * \r, -0.22949768583838578 * \r)
			}; \draw [f] plot [smooth] coordinates {
			(0.35861666010137755 * \r, -0.7094901291822779 * \r)
			(0.216561361065778 * \r, -0.7680047907343848 * \r)
			(0.05722249002053528 * \r, -0.8197573890814775 * \r)
			(-0.11308477221691995 * \r, -0.8570690995780968 * \r)
			(-0.2847666570358103 * \r, -0.8715455471533294 * \r)
			(-0.44617038323884395 * \r, -0.8556552563207178 * \r)
			(-0.5856665494756039 * \r, -0.8043673942387144 * \r)
			}; \draw [f] plot [smooth] coordinates {
			(0.4458131325319566 * \r, -0.8341423252737478 * \r)
			(0.3459455325115729 * \r, -0.8912848597210525 * \r)
			(0.23601717764322283 * \r, -0.9416910645479524 * \r)
			(0.11947753291676869 * \r, -0.9795043186589512 * \r)
			(0.0015474583860723146 * \r, -0.9982289374418344 * \r)
			};
			
			\draw [f] plot [smooth] coordinates {
			(-0.9734502512412145 * \r, 0.21645603838080008 * \r)
			(-0.9503650978095842 * \r, 0.2695639155955872 * \r)
			(-0.9319027913015903 * \r, 0.30029198363127285 * \r)
			(-0.9249960830190598 * \r, 0.31030682175729835 * \r)
			(-0.9319027913015903 * \r, 0.30029198363127285 * \r)
			(-0.9503650978095842 * \r, 0.2695639155955872 * \r)
			(-0.9734502512412145 * \r, 0.21645603838080008 * \r)
			}; \draw [f] plot [smooth] coordinates {
			(-0.8933051772026108 * \r, -0.4473276762951144 * \r)
			(-0.9437264382844831 * \r, -0.2838810566103164 * \r)
			(-0.9556361731823627 * \r, -0.12013626344817188 * \r)
			(-0.9420244751112437 * \r, 0.03198907553846231 * \r)
			(-0.9169820338659903 * \r, 0.16571209105209736 * \r)
			(-0.8918592116441181 * \r, 0.27847459619059345 * \r)
			(-0.8736426618596116 * \r, 0.36986103012019333 * \r)
			(-0.8647058074052786 * \r, 0.4396171487037881 * \r)
			(-0.8629573718317479 * \r, 0.48632279099550135 * \r)
			}; \draw [f] plot [smooth] coordinates {
			(-0.24887452699923499 * \r, -0.9670707889480283 * \r)
			(-0.3835049655974334 * \r, -0.903744784202541 * \r)
			(-0.5038502773817606 * \r, -0.7963522732699487 * \r)
			(-0.5995176709543579 * \r, -0.6461535279119586 * \r)
			(-0.6641821670909152 * \r, -0.46177413916498977 * \r)
			(-0.6975919657678042 * \r, -0.25702421281576293 * \r)
			(-0.7049807754233284 * \r, -0.04685628278467713 * \r)
			(-0.6944523655637698 * \r, 0.15612973442767675 * \r)
			(-0.6739836278079241 * \r, 0.34304158902523585 * \r)
			(-0.6492930574882867 * \r, 0.5080066837973016 * \r)
			(-0.6228481288947172 * \r, 0.6465064785996874 * \r)
			(-0.5937379321355996 * \r, 0.7538220434576267 * \r)
			(-0.5581747733444302 * \r, 0.8243675862804476 * \r)
			}; \draw [f] plot [smooth] coordinates {
			(0.1782878076528404 * \r, -0.9625714375805758 * \r)
			(0.09481954969973533 * \r, -0.9388625125517499 * \r)
			(0.01227895333316184 * \r, -0.8828330916905951 * \r)
			(-0.06429365917220373 * \r, -0.7885935419515656 * \r)
			(-0.12958504925707726 * \r, -0.6537569692151357 * \r)
			(-0.17916245213437132 * \r, -0.48092478119651527 * \r)
			(-0.21054453359558659 * \r, -0.2779049380157696 * \r)
			(-0.2236637488749284 * \r, -0.056422853605624324 * \r)
			(-0.22048857606976363 * \r, 0.1700503906008594 * \r)
			(-0.2040374734996862 * \r, 0.38829697279872016 * \r)
			(-0.17728124890467267 * \r, 0.5864632063576696 * \r)
			(-0.14236321065241325 * \r, 0.7543044189429402 * \r)
			(-0.10033008382571201 * \r, 0.88313395007027 * \r)
			(-0.05135858529677241 * \r, 0.9661264233087006 * \r)
			(0.004667165987406435 * \r, 0.9992742435139426 * \r)
			}; \draw [f] plot [smooth] coordinates {
			(0.4458131325319566 * \r, -0.8341423252737478 * \r)
			(0.4089728542436598 * \r, -0.8028295912922291 * \r)
			(0.37442043768152067 * \r, -0.7474499452734952 * \r)
			(0.3441365849139677 * \r, -0.6642080915108389 * \r)
			(0.32012794889533924 * \r, -0.5511193026080723 * \r)
			(0.3041088258523764 * \r, -0.4089404095123532 * \r)
			(0.2971824087693057 * \r, -0.24174672979927664 * \r)
			(0.29963128776713077 * \r, -0.056904273644859726 * \r)
			(0.3108980531376256 * \r, 0.13560887907953117 * \r)
			(0.3297597056854791 * \r, 0.3243442221059205 * \r)
			(0.35461990849812774 * \r, 0.49768891227433143 * \r)
			(0.3838087712166306 * \r, 0.6450997452084247 * \r)
			(0.41580247020092237 * \r, 0.7581667501588631 * \r)
			(0.44932899437946006 * \r, 0.8314804878642604 * \r)
			(0.48337701531684873 * \r, 0.863229791720793 * \r)
			};
		
			\draw [fhat] plot [smooth] coordinates {
			(0.359208706112651 * \r, 0.4987310980385744 * \r)
			(0.15065782257807764 * \r, 0.550502846593558 * \r)
			(-0.07348030935058453 * \r, 0.5800016883314649 * \r)
			(-0.2924121790259052 * \r, 0.581499434809406 * \r)
			(-0.4873372373645832 * \r, 0.5560312839617805 * \r)
			(-0.6458387401127383 * \r, 0.5102688537297579 * \r)
			(-0.763605703329838 * \r, 0.4541353225672979 * \r)
			(-0.8433625986302536 * \r, 0.39817678729087236 * \r)
			(-0.8920066859744301 * \r, 0.351546738514948 * \r)
			(-0.9172814319002328 * \r, 0.3209396860760428 * \r)
			(-0.9249960830190589 * \r, 0.3103068217572983 * \r)
			(-0.9172814319002328 * \r, 0.32093968607604284 * \r)
			(-0.8920066859744302 * \r, 0.35154673851494816 * \r)
			(-0.8433625986302539 * \r, 0.3981767872908726 * \r)
			(-0.7636057033298382 * \r, 0.45413532256729816 * \r)
			(-0.6458387401127386 * \r, 0.5102688537297583 * \r)
			(-0.48733723736458334 * \r, 0.5560312839617809 * \r)
			(-0.2924121790259054 * \r, 0.5814994348094066 * \r)
			(-0.07348030935058458 * \r, 0.5800016883314655 * \r)
			(0.15065782257807758 * \r, 0.5505028465935584 * \r)
			(0.3592087061126516 * \r, 0.4987310980385747 * \r)
			}; \draw [fhat] plot [smooth] coordinates {
			(0.30433808799790063 * \r, 0.05097914699711112 * \r)
			(0.09833448833852315 * \r, 0.05118557294450462 * \r)
			(-0.12323383904366983 * \r, 0.051008144410225165 * \r)
			(-0.3412909118691884 * \r, 0.04914287598663128 * \r)
			(-0.5376121647607792 * \r, 0.047005009927134356 * \r)
			(-0.6991320822908296 * \r, 0.04800447574542066 * \r)
			(-0.8200364065596023 * \r, 0.05655906068574468 * \r)
			(-0.9012467295262934 * \r, 0.07714845687011468 * \r)
			(-0.9480263018793345 * \r, 0.11361983323110303 * \r)
			(-0.9669140795491624 * \r, 0.1687506125153267 * \r)
			(-0.9630359328703623 * \r, 0.24390371452913118 * \r)
			(-0.9383995624193004 * \r, 0.33857233606388665 * \r)
			(-0.8913960981746264 * \r, 0.44972475407241874 * \r)
			(-0.8175305051737833 * \r, 0.5710865996822861 * \r)
			(-0.7112927352050245 * \r, 0.6927577649489853 * \r)
			(-0.5688879153111417 * \r, 0.8017234475424924 * \r)
			(-0.39117977202231896 * \r, 0.8837184672991005 * \r)
			(-0.18579681051416252 * \r, 0.9264305108928617 * \r)
			(0.032749970094639536 * \r, 0.923280138200389 * \r)
			(0.24554136594232268 * \r, 0.8763859908661004 * \r)
			(0.4337361953077159 * \r, 0.7973345902882021 * \r)
			}; \draw [fhat] plot [smooth] coordinates {
			(0.2992996555226626 * \r, -0.3972740228935312 * \r)
			(0.12339351699151814 * \r, -0.4459138558521465 * \r)
			(-0.07627616519410094 * \r, -0.4790441139177672 * \r)
			(-0.2851414118347961 * \r, -0.4945976257041489 * \r)
			(-0.48596254406431894 * \r, -0.4896131062328854 * \r)
			(-0.6629619262870747 * \r, -0.46154780989400634 * \r)
			(-0.8050436227707097 * \r, -0.4089512430501964 * \r)
			(-0.9069995103532917 * \r, -0.33157413862035345 * \r)
			(-0.9686510166990254 * \r, -0.2301630299286167 * \r)
			}; \draw [fhat] plot [smooth] coordinates {
			(0.3539266219188689 * \r, -0.7054436079138677 * \r)
			(0.23012595070210748 * \r, -0.773670663218174 * \r)
			(0.07529361554301595 * \r, -0.8274864202946295 * \r)
			(-0.10073958951370396 * \r, -0.862279385126902 * \r)
			(-0.2829196066426315 * \r, -0.8713790726561192 * \r)
			(-0.45429571503925625 * \r, -0.8491323237112242 * \r)
			(-0.5998205856550505 * \r, -0.792611293161608 * \r)
			}; \draw [fhat] plot [smooth] coordinates {
			(0.4204835469896887 * \r, -0.8363029769797129 * \r)
			(0.35183449847522225 * \r, -0.8908413379740356 * \r)
			(0.25299687047991715 * \r, -0.9402198805367845 * \r)
			(0.13162540253695726 * \r, -0.9790590672250401 * \r)
			(0.0002948148195304201 * \r, -0.9983069001770652 * \r)
			};
			
			\draw [fhat] plot [smooth] coordinates {
			(-0.9744395695783507 * \r, 0.21390513016785856 * \r)
			(-0.9513118530500335 * \r, 0.26814442653564685 * \r)
			(-0.9322413746349252 * \r, 0.2998967218353967 * \r)
			(-0.9249960830190589 * \r, 0.3103068217572983 * \r)
			(-0.9322413746349253 * \r, 0.2998967218353966 * \r)
			(-0.9513118530500337 * \r, 0.26814442653564685 * \r)
			(-0.9744395695783508 * \r, 0.21390513016785845 * \r)
			}; \draw [fhat] plot [smooth] coordinates {
			(-0.8966741044089725 * \r, -0.43970537525259396 * \r)
			(-0.9427325576912367 * \r, -0.28379401820956646 * \r)
			(-0.9534401357782512 * \r, -0.1262717232399519 * \r)
			(-0.9401902570651584 * \r, 0.023407601369679065 * \r)
			(-0.915474306275765 * \r, 0.1588534819022458 * \r)
			(-0.8900428798927428 * \r, 0.2762573983550519 * \r)
			(-0.8710726344209778 * \r, 0.37321342645956024 * \r)
			(-0.8612656109336192 * \r, 0.4474642794529591 * \r)
			(-0.858632395444207 * \r, 0.4960869675252606 * \r)
			}; \draw [fhat] plot [smooth] coordinates {
			(-0.2603178927014831 * \r, -0.9637841965958185 * \r)
			(-0.393041092891077 * \r, -0.8981229892755758 * \r)
			(-0.510161762273099 * \r, -0.7895143707055029 * \r)
			(-0.6016586872864597 * \r, -0.6411651772506053 * \r)
			(-0.6629619262870746 * \r, -0.4615478098940058 * \r)
			(-0.6949109622960667 * \r, -0.2621118513252738 * \r)
			(-0.7024079682433995 * \r, -0.05494749716620467 * \r)
			(-0.6923846254838077 * \r, 0.1488592591986816 * \r)
			(-0.6717581997140545 * \r, 0.33989525976317797 * \r)
			(-0.6458387401127386 * \r, 0.5102688537297583 * \r)
			(-0.6173735318054683 * \r, 0.6530644598540674 * \r)
			(-0.5862325366396838 * \r, 0.7618963198313375 * \r)
			(-0.5496857735778129 * \r, 0.8309581657933198 * \r)
			}; \draw [fhat] plot [smooth] coordinates {
			(0.1944362840297894 * \r, -0.9614998147453855 * \r)
			(0.11707187909006128 * \r, -0.9407657290408861 * \r)
			(0.03162728631781969 * \r, -0.8884414710813127 * \r)
			(-0.051797547527100196 * \r, -0.7963795385526358 * \r)
			(-0.1245897609052493 * \r, -0.6621512647143863 * \r)
			(-0.18062314778854216 * \r, -0.4892004785538161 * \r)
			(-0.21679018499621075 * \r, -0.2858667977349504 * \r)
			(-0.23286514978676087 * \r, -0.06375501326254773 * \r)
			(-0.23079598567876944 * \r, 0.1639987165902979 * \r)
			(-0.2137017353701599 * \r, 0.3841547873650828 * \r)
			(-0.1848808874147076 * \r, 0.5843396379369263 * \r)
			(-0.14706034994244 * \r, 0.7536450982435845 * \r)
			(-0.10200491838901027 * \r, 0.8830790015845673 * \r)
			(-0.05051079413451692 * \r, 0.9660825033362163 * \r)
			(0.007264977048576815 * \r, 0.9992043918719713 * \r)
			}; \draw [fhat] plot [smooth] coordinates {
			(0.42048354698968576 * \r, -0.8363029769797137 * \r)
			(0.39805146609300723 * \r, -0.8011407410251035 * \r)
			(0.36892750112315975 * \r, -0.7442599684467621 * \r)
			(0.3395851359579846 * \r, -0.6590007887203981 * \r)
			(0.315189972647401 * \r, -0.542756215676506 * \r)
			(0.2992996555226616 * \r, -0.3972740228935302 * \r)
			(0.29368491151444687 * \r, -0.2281506254284591 * \r)
			(0.2984288436834292 * \r, -0.04380223654336019 * \r)
			(0.3123019393415583 * \r, 0.14574507890578015 * \r)
			(0.3333014048937661 * \r, 0.3299688293412569 * \r)
			(0.3592087061126516 * \r, 0.4987310980385747 * \r)
			(0.3880428191440367 * \r, 0.6429377154462164 * \r)
			(0.4183367203796562 * \r, 0.7550766739166581 * \r)
			(0.4492150856850901 * \r, 0.8297438894066522 * \r)
			(0.48029168834319413 * \r, 0.8641956473724512 * \r)
			};
			
		    \draw (O) circle (\r);
		\end{tikzpicture}
		\caption{$25$ samples.}
	\end{subfigure}
	\caption{Using different number of samples, we approximate $f \from [-1, 1]^{2} \to S^{2}$.
	Let $\Gamma$ be a $10 \times 10$ grid on $[-1, 1]^{2}$.
	Solid lines show $f(\Gamma)$ and dashed lines show $\mathhat{f}(\Gamma)$.}
	\label{fig:map_into_sphere}
\end{figure}

\subsection{Recent work}%
\label{sub:Recent work}

\new{%
Lots of research is being done on approximating maps into manifolds.
Dyn and Sharon~\cite{Dyn17a,Dyn17b} do manifold curve fitting by adapting subdivision schemes.
Hardering and Wirth~\cite{Hardering21} and Heeren, Rumpf, and Wirth~\cite{Heeren19} and Zhang and Noakes~\cite{Zhang19} consider manifold versions of cubic splines.
Bergmann and Gousenbourger~\cite{Bergmann18}, and Gousenbourger, Massart, and Absil~\cite{Gousenbourger19} define manifold B\'ezier curves.
Sharon, Cohen, and Wendland~\cite{Sharon23}, Petersen and M\"uller~\cite{Petersen19}, and Grohs, Sprecher, and Yu~\cite{Grohs16} generalize moving least squares to manifolds.
Cornea, Zhu, Kim, and Ibrahim \cite{Cornea16}, Hinkle, Fletcher, and Joshi \cite{Hinkle14}, and Fletcher \cite{Fletcher13} do regression of data on manifolds.
}

\new{%
Also similar problems receive attention.
The \emph{Hermite interpolation problem} is to find a map that has given values and given derivatives in a set of sample points.
Such a map can also be used for approximation.
Séguin and Kressner~\cite{Seguin24}, Vardi, Dyn, and Sharon~\cite{Vardi24}, Zimmermann~\cite{Zimmermann20,Zimmermann22b}, Zimmermann and Bergmann~\cite{Zimmermann22}, and Moosm\"uller~\cite{Moosmuller16,Moosmuller17}, solve Hermite problems for maps into manifolds.
}

\new{%
There are also efforts to adapt statistics tools to manifolds.
We only mention a few here.
Gebhardt, Schubert, and Steinbach~\cite{Gebhardt23}, Curry, Marsland, and McLachlan~\cite{Curry19}, Lazar and Lin~\cite{Lazar17}, and Chakraborty, Seo, and Vemuri~\cite{Chakraborty16} study how principal component analysis can be extended to different manifolds, a tool known as \emph{principal geodesic analysis}.
Pennec \cite{Pennec20a} discusses these techniques specifically in the context of \emph{diffusion tensor imaging}.
Diepeveen, Chew, and Needell~\cite{Diepeveen23} develop curvature corrected versions of tensor decompositions to compress manifold data.
Diepeveen~\cite{Diepeveen24} extends several other common data analysis tools to manifolds.
}

\subsection{Outline}

In \cref{sec:Preliminaries}, we recall some geometric preliminaries including Toponogov's theorem, a well-known Riemannian comparison theorem.

\Cref{sec:Error_bounds} describes the error analysis of our algorithm template, where our main original contribution is \cref{thm:manifold_error_bound}.
Our analysis assumes a lower bound on the \emph{sectional curvature} of the manifold. 
For completeness and ease of reference, we lists several concrete manifolds along with explicit bounds on their curvature in \cref{sec:Manifolds with lower bounded sectional curvature}.
We are especially interested in manifolds with nonnegative curvature, as their error bound is guaranteed not to be worse than in the familiar linear case. 

Inspired by Dolgov, Kressner, and Str\"ossner~\cite{Strossner2022} and by Str\"ossner, Sun, and Kressner~\cite{Strossner2023}, we use \emph{tensorized Chebyshev interpolation} in step \ref{item:step2} of our template to implement \cref{alg:manifold_approximation_scheme}.
We discuss the rest of the implementation details and recall error bounds for tensorized Chebyshev interpolation in \cref{sec:Algorithm}.

We present two example approximation problems inspired by linear algebra applications in \cref{sec:Numerical experiments}.
In these examples, the approximation error is verified to be bounded by \cref{thm:manifold_error_bound}.
Our implementation is available as a Julia package, \texttt{ManiFactor.jl}~\cite{manifactor}. 

Finally, the paper is concluded with a summary of the main results and a future outlook in \cref{sec:conclusions}.

\section*{Acknowledgements}
We thank Arne Bouillon for spotting an error in \cref{eq:manifold_error_bound2_simplified} in an earlier draft, and Astrid Herremans for her expertise in approximation theory and her suggestions about the figures.
We thank Ronny Bergmann for reading the preprint and giving valuable comments.
\new{We also thank three anonymous reviewers for valuable suggestions for improvements.}

\section*{Funding}

This research was funded by BOF project C16/21/002 by the Internal Funds of KU Leuven and FWO project G080822N.\
J. Van der Veken is additionally supported by the Research Foundation--Flanders (FWO) and the Fonds de la Recherche Scientifique (FNRS) under EOS Project G0I2222N.\ 
R. Vandebril is additionally supported by the Research Foundation--Flanders (Belgium), projects G0A9923N and G0B0123N.

\section{Geometric preliminaries}%
\label{sec:Preliminaries}

Morally, a \emph{manifold} is a space $M$ that locally can be identified with open subsets of $\reals^{n}$.
Such an identification is called a \emph{chart}.
A detailed introduction to manifolds can be found in Lee~\cite[chapters 1--3]{Lee13}.

\new{%
The \emph{tangent space} $T_p M$ at a point $p \in M$ is the space of velocity vectors of curves in $M$ through $p$.
It is an $n$-dimensional vector space.
A \emph{Riemannian manifold} is a manifold whose tangent space is equipped with an inner product $\innerproduct{\cdot}{\cdot}_p$.
Such an inner product induces a norm $\norm{\cdot}_{p}$ on $T_p M$ and a notion of distance $d_M(\cdot, \cdot)$ between points on $M$.}
A \emph{geodesic} is a curve in $M$ that is locally distance-minimizing.
Each tangent vector is associated with a geodesic through $p$, and vice versa.
This gives a canonical chart around each point $p \in M$, called a \emph{normal coordinate chart}.
It consists of the \emph{manifold exponential} and its inverse, the \emph{manifold logarithm}, which 
identify a neighbourhood $S \subset T_p M$ of $0$ with a neighbourhood $S' \subset M$ of $p$.
We denote them by
	\begin{align}
		\operatorname{exp}_p \from S \to S',
\qquad\text{and}\qquad
		\operatorname{log}_p \from S' \to S.
	\end{align}%
A detailed introduction to Riemannian geometry is given in Lee~\cite[chapters 1--8]{Lee18}.

A Riemannian manifold $M$ is \emph{geodesically convex} if for any two points on $M$, there is a unique distance-minimizing geodesic connecting them.
Such a geodesic is called \emph{minimal}.
A \emph{geodesic triangle} is a set of three minimal geodesics connecting three points on $M$.

\subsection{Sectional curvature}%
\label{sub:Sectional curvature}

If $M$ is a two-dimensional Riemannian manifold, then we can define the \emph{Gaussian curvature} $K$ at a point $p \in M$ in several ways.
One way is to define it by the angular defect of small triangles:
\begin{align}
	\alpha + \beta + \gamma - \pi = K A + \textrm{higher order terms}, 
\end{align}
where $\alpha$, $\beta$, $\gamma$ are the angles of a geodesic triangle in $M$ at $p$ with area $A$.

If $M$ is not two-dimensional, consider a two-dimensional linear subspace $\Pi$ of $T_p M$.
Then $\Sigma = \operatorname{exp}_p(\Pi)$ is a two-dimensional submanifold of $M$.
The Gaussian curvature of $\Sigma$ at $p$ is called the \emph{sectional curvature} of $M$ along $\Pi$.

A manifold with constant sectional curvature $H$ is called a \emph{model manifold}.
It is uniquely determined by $H$ and its dimension.

\subsection{Riemannian comparison theory}%
\label{sub:Riemannian_comparison_theory}

Riemannian comparison theory studies inequalities involving geometric quantities.
A good reference is Cheeger and Ebin~\cite{Cheeger08}. Many comparison theorems relate properties of geodesics to curvature bounds.
\emph{Toponogov's theorem} is such a theorem, bounding side lengths of geodesic triangles defined on manifolds with lower bounded sectional curvature.
It is illustrated in \cref{fig:geodesic_triangle}.

\begin{figure}[t]
	\centering
	\begin{tikzpicture}
	    \node (O) at (0, 0) {\huge{$\leadsto$}}; 
	    \node (O1) at (-\r, 0) {}; 
	    \node (O2) at (1.33 * \r, 0) {}; 
		
		\pgfmathsetseed{38}

		\shade[rounded corners={0.2 * \r}, myspherestyle] (O1) \irregularcircle{0.9 * \r}{0.25 * \r};
		\pgfmathsetseed{38}
		\draw[rounded corners={0.2 * \r}] (O1) \irregularcircle{0.9 * \r}{0.25 * \r};
	
	    \node (M) at (-0.85 * \r, 0.57 * \r) {$M$};
	    \node (A) at (-\r, -0.67 * \r) {};
	    \node (B) at (-1.33 * \r, 0.2 * \r) {}; 
		\draw (B) node[anchor=south] {};
	    \node (C) at (-0.687 * \r, 0.067 * \r) {}; 
		\draw (C) node[anchor=south] {};
		\node (AB_control1) at (-1.2 * \r, -0.4 * \r) {};
		\node (AB_control2) at (-1.26 * \r, 0) {};
		\node (AC_control) at (-0.74 * \r, -0.23 * \r) {};
		\node (BC_control) at (-0.92 * \r, 0.18 * \r) {};
	
	    \draw plot[smooth, tension=0.8] coordinates { (A) (AB_control1) (AB_control2) (B) };
	    \draw plot[smooth, tension=0.8] coordinates { (A) (AC_control) (C) };
	    \draw[dashed] plot[smooth, tension=0.8] coordinates { (B) (BC_control) (C) };
		
		\node at ($(AB_control1)!0.35!(AB_control2)$) {\footnotesize\rotatebox[origin=c]{-80}{$|$}};
		\node at ($(A)!0.85!(AC_control)$) {\footnotesize\rotatebox[origin=c]{63}{$||$}};
	
		\begin{scope}
			\node (L) at ($(O1) + (0.5 * \r, -0.5 * \r)$) {};
			\node (R) at ($(O1) + (-0.5 * \r, -0.5 * \r)$) {};
	
			\clip plot[smooth, tension=0.8] coordinates { (A) (AB_control1) (AB_control2) (B) (L) };
	    	\clip plot[smooth, tension=0.8] coordinates { (A) (AC_control) (C) (R) };
	    	\draw (A) circle (0.2 * \r);
			\node () at ($(A) + (0, 0.2 * \r)$) {\footnotesize$|||$};
		\end{scope}

	    \shade [myspherestyle] (O2) circle (\r);
	    \draw (O2) circle (\r);
	
	    \node (N) at (1.55 * \r, 0.7 * \r) {$N$};
	    \node (A') at (1.33 * \r, -0.67 * \r) {};
	    \node (B') at (0.88 * \r, 0.2 * \r) {}; 
		\draw (B') node[anchor=south] {};
	    \node (C') at (1.82 * \r, 0.067 * \r) {}; 
		\draw (C') node[anchor=south] {};
		\node (A'B'_control) at (1.074 * \r, -0.273 * \r) {};
		\node (A'C'_control) at (1.605 * \r, -0.333 * \r) {};
		\node (B'C'_control) at (1.35 * \r, 0.152 * \r) {};
	
	    \draw plot[smooth, tension=1.0] coordinates { (A') (A'B'_control) (B') };
	    \draw plot[smooth, tension=1.0] coordinates { (A') (A'C'_control) (C') };
	    \draw[dashed] plot[smooth, tension=1.0] coordinates { (B') (B'C'_control) (C') };
	
		\node at (A'B'_control) {\footnotesize\rotatebox[origin=c]{-60}{$|$}};
		\node at (A'C'_control) {\footnotesize\rotatebox[origin=c]{57}{$||$}};
	
		\begin{scope}
			\node (L') at ($(O2) + (0.5 * \r, -0.5 * \r)$) {};
			\node (R') at ($(O2) + (-0.5 * \r, -0.5 * \r)$) {};
	
			\clip plot[smooth, tension=1.0] coordinates { (A') (A'B'_control) (B') (L') };
	    	\clip plot[smooth, tension=1.0] coordinates { (A') (A'C'_control) (C') (R') };
	    	\draw (A') circle (0.2 * \r);
			\node () at ($(A') + (0, 0.2 * \r)$) {\footnotesize$|||$};
		\end{scope}
	
	\end{tikzpicture}
	\caption{%
 		\Cref{prop:toponogov} gives conditions for which the length of the opposite (dashed) side in $M$ is bounded by the length of the opposite (dashed) side in $N$.
		Equal length geodesics and equal angles have the same number of notches.%
		}
	\label{fig:geodesic_triangle}
\end{figure}
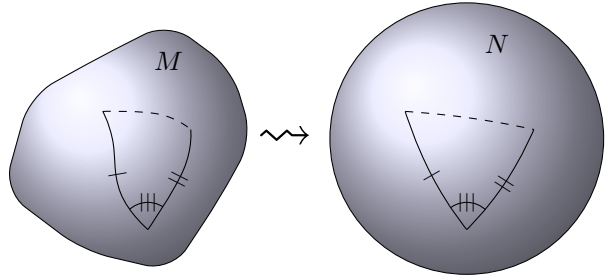

\begin{proposition}[Toponogov's theorem~\protect{\cite[Theorem 2.2]{Cheeger08}}]\label{prop:toponogov}
	Let $M$ be a geodesically convex Riemannian manifold with sectional curvature bounded from below by some constant $H$.
	Let $\gamma_1$, $\gamma_2 \from [0, 1] \to M$ be minimal geodesics such that $\gamma_1(0) = \gamma_2(0)$.
	If $H > 0$, also assume that $d_M(\gamma_1(0), \gamma_1(1))$, $d_M(\gamma_2(0), \gamma_2(1)) \leq \frac{\pi}{\sqrt{H}}$.
	Consider a model manifold $N$ of constant sectional curvature $H$ and let $\lambda_1$, $\lambda_2 \from [0, 1] \to N$ be geodesics such that $\lambda_1(0) = \lambda_2(0)$ and the angle where they meet satisfies $\angle(\lambda_1, \lambda_2) = \angle(\gamma_1, \gamma_2)$.
	Then
	\begin{align}
		d_{M}(\gamma_1(1), \gamma_2(1)) \leq d_{N}(\lambda_1(1), \lambda_2(1)).
	\end{align}
\end{proposition}

The advantage of looking at geodesic triangles on the model manifold is that there are spherical and hyperbolic versions of the trigonometric identities.
These identities are used in \cref{sec:Error_bounds} to find explicit error bounds when approximating maps into manifolds.

\begin{lemma}[Reid \protect{\cite[sections 3.2 and 3.10]{Reid05}}]\label{lemma:law_of_cosines}
	Let $N$ be a manifold of constant sectional curvature $H$ and consider a geodesic triangle on $N$ with side lengths $A$, $B$, $C$ and opposite angles $a$, $b$, $c$ respectively.
	Then
	\begin{align}
		\cos( C \sqrt{H} ) ={}& \cos( A \sqrt{H} ) \cos( B \sqrt{H} ) + \sin( A \sqrt{H} ) \sin( B \sqrt{H} ) \cos c \quad & \text{if } H > 0,\\
		\cosh( C \sqrt{\abs{H}} ) ={}& \cosh( A \sqrt{\abs{H}} ) \cosh( B \sqrt{\abs{H}} ) - \sinh( A \sqrt{\abs{H}} ) \sinh( B \sqrt{\abs{H}} ) \cos c \quad & \text{if } H < 0.
	\end{align}
\end{lemma}

\section{Error analysis}%
\label{sec:Error_bounds}

The distance between two points on a Riemannian manifold depends on the geometry of that manifold.
For example, the distance between two orthogonal matrices $A$ and $B$ in the space of orthogonal matrices will always be greater than or equal to the distance between $A$ and $B$ in the ambient space of matrices.
When approximating maps into Riemannian manifolds it is natural to measure the approximation error \emph{intrinsically} on the manifold, rather than in some \emph{extrinsic} ambient space in which the manifold could be embedded.
Moreover, for embedded manifolds with large codimension, such as low-rank matrices or tensors, it can be inefficient or even infeasible to measure the error in the ambient space.
The next result thus bounds the intrinsic error of the approximation template from \cref{sub:Contribution}.
 
\begin{theorem}\label{thm:manifold_error_bound}
	Let $M$ be a Riemannian manifold with sectional curvature bounded from below by some constant $H$ and let $f \from R \to M$, where $R$ is a set whose image fits in a single geodesically convex normal coordinate chart $S \subset T_p M$ around $p \in M$.
	Assume that the elements in $S$ have norm bounded by some constant $\sigma$.
	Let
	\begin{align}
	 	g = {\operatorname{log}_p} \circ f \from R \to S \quad\text{and}\quad \mathhat{f} = {\operatorname{exp}_p} \circ \mathhat{g} \from R \to M,
	\end{align}
	where $\mathhat{g} \from R \to S$ is an approximation of $g$ such that
	\begin{align}
		\norm{g(x) - \mathhat{g}(x)}_p \leq \epsilon
	\end{align}
	for all $x \in R$.
	Then, the distance between $f(x)$ and $\mathhat{f}(x)$ on $M$ obeys
	\begin{align}
		d_M(f(x), \mathhat{f}(x)) \leq{}& \epsilon\quad & \text{if } H \geq 0,%
		\label{eq:manifold_error_bound1}\\
		d_M(f(x), \mathhat{f}(x)) \leq{}& \epsilon + \frac{2}{\sqrt{\abs{H}}} \operatorname{arcsinh} \mleft( \frac{\epsilon \sinh(\sigma \sqrt{\abs{H}})}{2 \sigma} \mright)\quad & \text{if } H < 0,%
		\label{eq:manifold_error_bound2}
	\end{align}
	for all $x \in R$.
\end{theorem}

\begin{remark}
We can in turn upper bound \cref{eq:manifold_error_bound2} by
\begin{align}
	d_M(f(x), \mathhat{f}(x)) \leq{}& \epsilon + \frac{2}{\sqrt{\abs{H}}} \log \mleft(
		\frac{\epsilon \Exp(\sigma \sqrt{\abs{H}})}{2 \sigma} + 1
		\mright).
 	\label{eq:manifold_error_bound2_simplified}
\end{align}
\end{remark}

The assumptions in \cref{thm:manifold_error_bound} about geodesic convexity, bounded curvature, and bounded normal coordinate chart may seem very specific, but we note that they are always satisfied in some neighbourhood of $p$ \new{since curvature is continuous}.
To formulate explicit bounds, we however need explicit expressions for $H$.
Therefore, we list several standard manifolds along with explicit lower bounds for their sectional curvature in \cref{sec:Manifolds with lower bounded sectional curvature}.

\new{%
Our ability to approximate $g$ typically depends on both its smoothness and the specific features of $R$.
As an example, a scheme for approximating real-analytic functions from $[-1, 1]^{m}$ is presented in \cref{sub:Tensorized Chebyshev interpolation}.
}

\begin{proof}[Proof of \cref{thm:manifold_error_bound}]
The proof can be summarized as applying Toponogov's theorem to the geodesic triangle $(f(x), \mathhat{f}(x), p)$, as visualized in \cref{fig:geodesic_triangle2}.

\begin{figure}[t]
	\centering
	\begin{tikzpicture}
	    \node (O) at (0, 0) {\huge{$\leadsto$}}; 
	    \node (O1) at (-\r, 0) {}; 
	    \node (O2) at (1.33 * \r, 0) {}; 
		
		\pgfmathsetseed{38}
		\shade[rounded corners={0.2 * \r}, myspherestyle] (O1) \irregularcircle{0.9 * \r}{0.25 * \r};
		\pgfmathsetseed{38}
		\draw[rounded corners={0.2 * \r}] (O1) \irregularcircle{0.9 * \r}{0.25 * \r};
	
	    \node (M) at (-0.85 * \r, 0.57 * \r) {$M$};
	    \node (A) at (-\r, -0.67 * \r) {};
		\draw (A) node[anchor=north] {$p$};
	    \node (B) at (-1.33 * \r, 0.2 * \r) {}; 
		\draw (B) node[anchor=south] {$f(x)$};
	    \node (C) at (-0.687 * \r, 0.067 * \r) {}; 
		\draw (C) node[anchor=south] {~~~$\mathhat{f}(x)$};
		\node (AB_control1) at (-1.2 * \r, -0.4 * \r) {};
		\node (AB_control2) at (-1.26 * \r, 0) {};
		\node (AC_control) at (-0.74 * \r, -0.23 * \r) {};
		\node (BC_control) at (-0.92 * \r, 0.18 * \r) {};
	
	    \draw plot[smooth, tension=0.8] coordinates { (A) (AB_control1) (AB_control2) (B) };
	    \draw plot[smooth, tension=0.8] coordinates { (A) (AC_control) (C) };
	    \draw[dashed] plot[smooth, tension=0.8] coordinates { (B) (BC_control) (C) };
		
		\node at ($(AB_control1)!0.35!(AB_control2)$) {\footnotesize\rotatebox[origin=c]{-80}{$|$}};
		\node at ($(A)!0.85!(AC_control)$) {\footnotesize\rotatebox[origin=c]{63}{$||$}};
	
		\begin{scope}
			\node (L) at ($(O1) + (0.5 * \r, -0.5 * \r)$) {};
			\node (R) at ($(O1) + (-0.5 * \r, -0.5 * \r)$) {};
	
			\clip plot[smooth, tension=0.8] coordinates { (A) (AB_control1) (AB_control2) (B) (L) };
	    	\clip plot[smooth, tension=0.8] coordinates { (A) (AC_control) (C) (R) };
	    	\draw (A) circle (0.2 * \r);
			\node () at ($(A) + (0, 0.2 * \r)$) {\footnotesize$|||$};
		\end{scope}

	    \shade [myspherestyle] (O2) circle (\r);
	    \draw (O2) circle (\r);
	
	    \node (N) at (1.55 * \r, 0.7 * \r) {$N$};
	    \node (A') at (1.33 * \r, -0.67 * \r) {};
		\draw (A') node[anchor=north] {$q$};
	    \node (B') at (0.88 * \r, 0.2 * \r) {}; 
		\draw (B') node[anchor=south] {$h(x)$};
	    \node (C') at (1.82 * \r, 0.067 * \r) {}; 
		\draw (C') node[anchor=south] {$\mathhat{h}(x)$};
		\node (A'B'_control) at (1.074 * \r, -0.273 * \r) {};
		\node (A'C'_control) at (1.605 * \r, -0.333 * \r) {};
		\node (B'C'_control) at (1.35 * \r, 0.152 * \r) {};
	
	    \draw plot[smooth, tension=1.0] coordinates { (A') (A'B'_control) (B') };
	    \draw plot[smooth, tension=1.0] coordinates { (A') (A'C'_control) (C') };
	    \draw[dashed] plot[smooth, tension=1.0] coordinates { (B') (B'C'_control) (C') };
	
		\node at (A'B'_control) {\footnotesize\rotatebox[origin=c]{-60}{$|$}};
		\node at (A'C'_control) {\footnotesize\rotatebox[origin=c]{57}{$||$}};
	
		\begin{scope}
			\node (L') at ($(O2) + (0.5 * \r, -0.5 * \r)$) {};
			\node (R') at ($(O2) + (-0.5 * \r, -0.5 * \r)$) {};
	
			\clip plot[smooth, tension=1.0] coordinates { (A') (A'B'_control) (B') (L') };
	    	\clip plot[smooth, tension=1.0] coordinates { (A') (A'C'_control) (C') (R') };
	    	\draw (A') circle (0.2 * \r);
			\node () at ($(A') + (0, 0.2 * \r)$) {\footnotesize$|||$};
		\end{scope}
	
	\end{tikzpicture}
	\caption{Geodesic triangles defined by $p$, $f$, and $\mathhat{f}$.
	The map from $M$ to $N$ is $\operatorname{exp}_{N, q} \circ \operatorname{log}_{M, p}$.
	}
	\label{fig:geodesic_triangle2}
\end{figure}
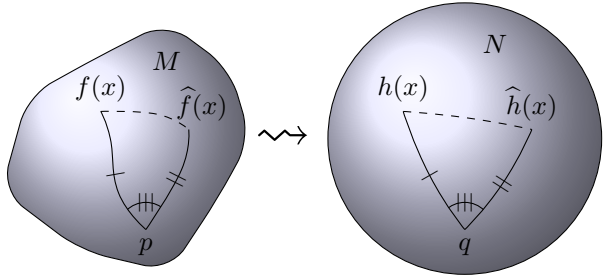

	Let $N$ be the manifold of constant curvature $H$ and let $q \in N$.
    Fix some isometry $T_p M \sim T_q N$ and let $\operatorname{exp}_{N, q} \from S \to N$ be the manifold exponential on $N$.
	Define $h = {\operatorname{exp}_{N, q}} \circ g$ and $\mathhat{h} = {\operatorname{exp}_{N, q}} \circ \mathhat{g}$.
	By \cref{prop:toponogov},
	\begin{align}
		d_M(f(x), \mathhat{f}(x)) \leq d_N(h(x), \mathhat{h}(x)).
	\end{align}
	Now our task is to bound the right-hand side of this inequality.

	For brevity, from now on we drop the argument $x$ of the functions $f$, $g$, and $h$.
	
	{\bf Case 1: $H \geq 0$.}
	Without loss of generality, let $H = 0$.
	Then $N = T_p M$ and so
	\begin{align}
		d_N(h, \mathhat{h}) 
		={}& \norm{g - \mathhat{g}}_p
		\leq \epsilon.
	\end{align}

	{\bf Case 2: $H < 0$.}
	Consider the geodesic triangle $(h(x), \mathhat{h}(x), q)$.
	Define $c(t) = \cosh( t \sqrt{\abs{H}} )$ and $s(t) = \sinh( t \sqrt{\abs{H}} )$.
	Then \cref{lemma:law_of_cosines} says that
	\begin{align}
		c(d(h, \mathhat{h})) ={}&
		c(\norm{g}_p) c(\norm{\mathhat{g}}_p) - s(\norm{g}_p) s(\norm{\mathhat{g}}_p) \cos{\left( \angle(g, \mathhat{g}) \right)}.%
		\label{eq:exact_expression_for_cosh_of_rhs}
	\end{align}
	Since ${\sinh}$ is positive on $\reals^{+}$,
	\begin{align}
		s(\norm{g}_p) s(\norm{\mathhat{g}}_p) \geq 0.
	\end{align}
	Hence using
	\begin{align}
		\cos{\left( \angle(g, \mathhat{g}) \right)} = \frac{\innerproduct{g}{\mathhat{g}}_p}{\norm{g}_p \norm{\mathhat{g}}_p}
		= \frac{\norm{g}_p^2 + \norm{\mathhat{g}}_p^2}{2 \norm{g}_p \norm{\mathhat{g}}_p} - \frac{\norm{g - \mathhat{g}}_p^2}{2 \norm{g}_p \norm{\mathhat{g}}_p}
		\geq 1 - \frac{\norm{g - \mathhat{g}}_p^2}{2 \norm{g}_p \norm{\mathhat{g}}_p}%
		\label{eq:bound_for_cos}
	\end{align}
	in \cref{eq:exact_expression_for_cosh_of_rhs} gives us
	\begin{align}
		c(d(h, \mathhat{h})) \leq{}&
		c(\norm{g}_p) c(\norm{\mathhat{g}}_p) - s(\norm{g}_p) s(\norm{\mathhat{g}}_p) \left( 1 - \frac{\norm{g - \mathhat{g}}_p^2}{2 \norm{g}_p \norm{\mathhat{g}}_p} \right) \\
		={}& c( \norm{g}_p - \norm{\mathhat{g}}_p ) + \frac{\norm{g - \mathhat{g}}_p^2 s(\norm{g}_p) s(\norm{\mathhat{g}}_p)}{2 \norm{g}_p \norm{\mathhat{g}}_p}.
	\end{align}
	Since ${\cosh}$ is even and increasing on $\reals^{+}$, we may use the reverse triangle inequality to see that
	\begin{align}
		c(d(h, \mathhat{h})) \leq{}& c(\norm{g - \mathhat{g}}_p) + \frac{\norm{g - \mathhat{g}}_p^2 s(\norm{g}_p) s(\norm{\mathhat{g}}_p)}{2 \norm{g}_p \norm{\mathhat{g}}_p}.%
		\label{eq:where_we_used_reverse_triangle_inequality}
	\end{align}
	Furthermore, noting that $\frac{\sinh{x}}{x}$ is increasing and positive on $\reals^+$, we have that
	\begin{align}
		c(d(h, \mathhat{h})) \leq{}& c(\norm{g - \mathhat{g}}_p) + \frac{\norm{g - \mathhat{g}}_p^2\, s(\sigma)^2}{2 \sigma^2}.
	\end{align}
	This is equivalent to
	\begin{align}
		1 + 2 s\mleft( \frac{d(h, \mathhat{h})}{2} \mright)^2 \leq{}& 1 + 2 s\mleft( \frac{\norm{g - \mathhat{g}}_p}{2} \mright)^2 + \frac{\norm{g - \mathhat{g}}_p^2 s(\sigma)^2 }{2 \sigma^2},%
		\label{eq:where_we_used_cosh_double_angle_formula}
	\end{align}
	so that
	\begin{align}
		d(h, \mathhat{h}) \leq{}& 2 \inverse{s} \mleft( \sqrt{
			s\mleft( \frac{\norm{g - \mathhat{g}}_p}{2} \mright)^2 + \frac{\abs{g - \mathhat{g}}^2\, s(\sigma)^2 }{4 \sigma^2}
			} \mright),
	\end{align}
	and since $\operatorname{sinh}$, $\operatorname{arcsinh}$, $\cdot^{2}$, and $\sqrt{\cdot}$ are all increasing on $\reals^{+}$, we find
	\begin{align}
		d(h, \mathhat{h}) \leq{}& 2 \inverse{s} \mleft( \sqrt{
			s\mleft( \frac{\epsilon}{2} \mright)^2 + \frac{\epsilon^2 s(\sigma)^2}{4 \sigma^2}
			} \mright).
	\end{align}
	Lastly we will use that concave functions are \emph{subadditive}, $\sqrt{a^2 + b^2} \leq a + b$ and $\operatorname{arcsinh}(\sinh{a} + \sinh{b}) \leq a + b$ for nonnegative $a$ and $b$.
	Thus
	\begin{align}
		d(h, \mathhat{h}) \leq{}& \epsilon + 2 \inverse{s}\mleft(
			\frac{\epsilon s(\sigma)}{2 \sigma}
			\mright).%
		\label{eq:where_we_used_subadditivity}
	\end{align}
	This concludes the proof.
\end{proof}

\subsection{Retractions}%
\label{sub:Retractions}

\new{%
For many manifolds that appear in numerical analysis, we know the exponential and logarithmic maps.
They can also often be computed efficiently.
For concrete examples of such manifolds, see the various references in \cref{sec:Manifolds with lower bounded sectional curvature}.
}

However, even if the exponential or logarithmic map are not known for a manifold, we might still be able to proceed by using a \emph{retraction}.
This is an injective map $r_p \from S \to M$, \new{for some open subset $S \subset T_p M$}, that approximates the exponential map at $p$.
For a precise definition, see Absil, Mahony, and Sepulchre~\cite[definition 4.1.1]{Absil08}.

\new{%
\begin{corollary}\label{cor:retraction}
	Assume the same conditions as in \cref{thm:manifold_error_bound}.
	Moreover, assume the maximum (Riemannian) error between $r_p$ and $\operatorname{exp}_p$ is bounded by $\zeta$, assume the maximum (Euclidean) error between $\inverse{r}_p$ and $\operatorname{log}_p$ is bounded by $\eta$, and assume that the vector-valued approximation scheme $\mathcal{F} \from g \mapsto \mathhat{g}$ is $\reals$-linear with operator max norm $\Lambda$.
	Now, instead of using $\mathhat{f} = ({\operatorname{exp}_p} \circ \mathcal{F} \circ {\operatorname{log}_p})(f)$, define $\mathhat{f} = (r_p \circ \mathcal{F} \circ \inverse{r_p})(f)$.
	Then
	\begin{align}
		d_M(f(x), \mathhat{f}(x)) \leq{}& \epsilon + \Lambda \eta + \zeta &\text{if $H \geq 0$,}\\
		d_M(f(x), \mathhat{f}(x)) \leq{}& \epsilon + \Lambda \eta + \zeta + \frac{2}{\sqrt{\abs{H}}} \operatorname{arcsinh}\mleft( \frac{(\epsilon + \Lambda \eta) \sinh(\sigma \sqrt{\abs{H}})}{2 \sigma} \mright) &\text{if $H < 0$.}
	\end{align}
\end{corollary}
}

\new{%
\begin{proof}
	We can reduce to \cref{thm:manifold_error_bound} by letting $\mathhat{g}_1 = (\mathcal{F} \circ \inverse{r_p})(f)$, $\mathhat{g}_2 = (\mathcal{F} \circ \operatorname{log}_p)(f)$, and noticing that
	\begin{align}
		\norm{g - \mathhat{g}_1} \leq{}& \norm{g - \mathhat{g}_2} + \norm{\mathhat{g}_2 - \mathhat{g_1}}
		\leq \norm{g - \mathcal{F}(g)} + \Lambda \norm{\inverse{r_p}(f) - \operatorname{log}_p(f)}
		\leq \epsilon + \Lambda \eta.
	\end{align}
\end{proof}
}

\new{%
If $M$ is an embedded submanifold, one can construct a retraction around a point $p \in M$ by embedding the tangent space and then projecting from there back to $M$.
Alternatively, one can correct any given smooth chart around $p$ to a retraction around $p$.
For Riemannian homogeneous manifolds, defining a retraction amounts to approximating the matrix exponential.
See the references in \cite[section 4.10]{Absil08} on how to implement retractions efficiently for different manifolds.
See also Absil and Oseledet's~\cite{Absil15} survey of retractions on the manifold of fixed-rank matrices, for which the exponential or logarithmic maps are not known.
The bottom row of their tables 1 and 2 list what we in \cref{cor:retraction} call $\zeta$.
}

\subsection{Condition numbers}%
\label{sub:Condition numbers}

When $H \to 0^-$, \cref{eq:manifold_error_bound2} tends to
\begin{align}
	d_M(f, \mathhat{f}) \leq 2 \epsilon,
\end{align}
which shows that it is not a tight bound.
To explain why, first note that \cref{eq:bound_for_cos} is an equality when $\norm{g(x)}_p = \norm{\mathhat{g}(x)}_p$, but we also used the reverse triangle inequality in \cref{eq:where_we_used_reverse_triangle_inequality} which is an equality when $g$ and $\mathhat{g}$ are collinear.
Furthermore, in \cref{eq:where_we_used_subadditivity} we used subadditivity for the square root, which is an equality only when one of the terms is $0$.
But for small curvatures and small $\epsilon$, the two terms in \cref{eq:where_we_used_subadditivity} are approximately equal.
One could say that we lost a factor $\sqrt{2}$ in each of these steps.

Even though the bound \cref{eq:manifold_error_bound2} is not tight, there is, in the following sense, no better bound.
\begin{proposition}\label{prop:lower_bound_for_error_bound}
Recall the notation of \cref{thm:manifold_error_bound}.
Assume that $M$ has constant negative sectional curvature $H$ and let $\mathhat{g}$ be an approximation to $g$ such that the maximum error $\epsilon$ is attained at a point $x \in R$ with $\norm{\mathhat{g}(x)}_p = \norm{g(x)}_p = \sigma$. Then
	\begin{align}
		d_M(f(x), \mathhat{f}(x)) =
			\frac{2}{\sqrt{\abs{H}}}
			\operatorname{arcsinh}\mleft(
				\frac{\epsilon \sinh{(\sigma \sqrt{\abs{H}})}}{2 \sigma}
				\mright).
	\end{align}
\end{proposition}

\begin{proof}
	Consider again the proof of \cref{thm:manifold_error_bound}.
	If $\norm{g}_p = \norm{\mathhat{g}}_p$, then \cref{eq:bound_for_cos} is an equality, and we do not need to use the reverse triangle inequality in \cref{eq:where_we_used_reverse_triangle_inequality}, hence \cref{eq:where_we_used_cosh_double_angle_formula} reduces to
	\begin{align}
		1 + 2 s\mleft( \frac{d_M(f, \mathhat{f})}{2} \mright)^2 ={}& 1 + \frac{\epsilon^{2} s(\sigma)^2}{2 \sigma^2}.
	\end{align}
 	This is equivalent to what we wanted to prove.
\end{proof}

As an immediate consequence of \cref{thm:manifold_error_bound} we can obtain bounds on the propagation of a small perturbation of a tangent vector through the exponential map $\mathrm{exp}_p$.
The size of this output perturbation relative to the input perturbation is quantified by the \emph{condition number}.
Recall the definition of Rice's condition number~\cite{Rice66} of a map $\phi \from A \to B$ between metric spaces $(A, d_A)$ and $(B, d_B)$:
\begin{align}
	\kappa[\phi](x) = \limsup_{\substack{x' \in A,\\ d_A(x, x') \to 0}} \frac{d_B(\phi(x), \phi(x'))}{d_A(x, x')}.
\end{align}
Let $\phi = \mathrm{exp}_p \from (R \to S) \to (R \to M)$, with the metrics
\begin{align}
	d_{R \to S}(g, \mathhat{g}) ={} \sup_{x \in R} \norm{g(x) - \mathhat{g}(x)}_p,\qq{and}
	d_{R \to M}(f, \mathhat{f}) ={} \sup_{x \in R} d_M(f(x), \mathhat{f}(x)).
\end{align}
The condition number implies an asymptotically sharp error bound
\begin{align}
	d_{R \to M} \mleft( \mathrm{exp}_p \circ g, \mathrm{exp}_p \circ \mathhat{g} \mright) \leq{}& \kappa[\mathrm{exp}_p](g) d_{R \to S}(g, \mathhat{g}) + o( d_{R \to S}(g, \mathhat{g}) )\\
	={}& \kappa[\mathrm{exp}_p](g) \epsilon + o( \epsilon ),
\end{align}

\begin{proposition}\label{cor:condition_bounds}
	Recall the notation of \cref{thm:manifold_error_bound}.
	The condition number $\kappa$ of the map $g \mapsto f$ satisfies
	\begin{align}
		1 \leq{} \kappa \leq{}& 1 + \frac{\sinh(\sigma \sqrt{\abs{H}})}{\sigma \sqrt{\abs{H}}} \quad & \text{if } H < 0, \\ 
		\kappa ={}& 1 \quad & \text{if } H \geq 0.
	\end{align}
	Furthermore, if $M$ is a model manifold with constant negative sectional curvature $H$, then it is also lower bounded by
	\begin{align}
		\frac{\sinh(\sigma \sqrt{\abs{H}})}{\sigma \sqrt{\abs{H}}} \leq \kappa.
	\end{align}
\end{proposition}

\new{%
\begin{proof}
	The upper bound is a corollary of \cref{thm:manifold_error_bound}.
	Moreover, $1 \leq \kappa$ since $\operatorname{exp}_p$ is a radial isometry.
	When $M$ is a model manifold, $\kappa$ is lower bounded by \cref{prop:lower_bound_for_error_bound}.
\end{proof}
}

The condition number grows exponentially in both $\sigma$ and $\sqrt{\abs{H}}$.
For example, such innocent looking parameters as $H = -10$ and $\sigma = 10$ gives a condition number of almost $9\cdot 10^{11}$, potentially losing about 12 digits of accuracy.
In general, if we want to be guaranteed a reasonable condition number, we should make sure that the inputs have small norm: $\sigma \leq 1 / \sqrt{\abs{H}}$.

\new{%
Zimmermann~\cite[Theorem 3.1]{Zimmermann20} derives a special case of \cref{cor:condition_bounds} for the Stiefel manifold using the Jacobi equation.
Diepeveen et al.~\cite[Theorem 3.4]{Diepeveen23} arrive at a similar expression when they estimate tensor compression errors for data on locally symmetric manifolds, also using the Jacobi equation.
The Jacobi equation only encodes first-order information, and so an advantage of our approach is that we can derive \emph{exact} error bounds with Toponogov's theorem.
}

\section{A concrete implementation}
\label{sec:Algorithm}

We now want to approximate a function $f \from [-1, 1]^m \to M$, where $M$ is a Riemannian manifold.
In the notation of \cref{sec:Error_bounds}, $R = [-1, 1]^m$.
\Cref{alg:manifold_approximation_scheme} details a concrete realization of the approximation template discussed in \cref{sub:Contribution} that uses tensorized Chebyshev interpolation.
We describe steps \ref{item:step1} and \ref{item:step2} of the template in detail.
The final step consists of applying the exponential map to the constructed approximation, which requires no further discussion.

\begin{algorithm}[t]
\caption{Approximating a map into a manifold.}\label{alg:manifold_approximation_scheme}
\begin{algorithmic}[1]
	\REQUIRE $f \from {[-1, 1]}^m \to M$, number $N_k + 1$ of Chebyshev nodes to use in each direction $k = 1$, \dots, $m$.
	\ENSURE Approximant $\mathhat{f} \from [-1, 1]^m \to M$ of $f$ satisfies the error bound of \cref{thm:manifold_error_bound}.
	\STATE Choose $p \in M$ as (an approximation of) the Karcher mean of a large number of $f(x)$ for $x$ chosen from a uniform distribution on ${[-1, 1]}^{m}$.
	\STATE Let $g = {\operatorname{log}_p} \circ f$.
	\FOR{$k = 1$, \dots, $m$}
		\FOR{$i_k = 1$, \dots, $N_k + 1$}
			\STATE Let $t_{i_k} = \cos\mleft( (2 i_k - 1) \pi / (2 (N_k + 1)) \mright)$ be the $i_k$th Chebyshev node.
			\FOR{$j = 1$, \dots, $n$}
				\STATE Let $G_{i_1 \dots i_m j} = g_{j}(t_{i_1}, \dots, t_{i_m})$.
			\ENDFOR
		\ENDFOR
	\ENDFOR
	\STATE Let $\mathhat{G}_{i_1 \dots i_m j} = \sum_{a_1, \dots, a_m, b} C_{a_1 \dots a_m b} U^{(1)}_{i_1 a_1} \dots U^{(m)}_{i_m a_m} V_{b j}$ be an ST-HOSVD of $G$.
	\FOR{$k = 1$, \dots, $m$}
		\STATE Let $h^{(k)}_{a_k}$ be the degree-$N_k$ interpolating polynomial satisfying $h^{(k)}_{a_k}(t_{i_k}) = U^{(k)}_{i_k a_k}$.
	\ENDFOR
		\FOR{$j = 1$, \dots, $n$}
			\STATE Let $\mathhat{g}_{j}(x) = \sum_{a_1, \dots, a_m, b} C_{a_1 \dots a_m b} h^{(1)}_{a_1}(x_1) \dots h^{(m)}_{a_m}(x_m) V_{b j}$.
	\ENDFOR
	\STATE Let $\mathhat{f} = {\operatorname{exp}_p} \circ \mathhat{g}$.
	\STATE Return $\mathhat{f}$.
\end{algorithmic}
\end{algorithm}

\subsection{Choosing from where to linearize}

\new{
In practice, the choice of the point $p \in M$ in \cref{thm:manifold_error_bound} affects the error a lot.
\Cref{prop:lower_bound_for_error_bound} says that the further away from $p$ we get, the bigger the error can be.
It is thus natural to minimize the distance to $p$.
Selecting the optimal $p$ is a non-trivial problem, and is out of scope for this article.
However, a related problem that has been studied in detail is the generalization of means and medians to manifolds.
We will leverage this to define a heuristic for choosing $p$.
}

Consider a sequence $(\xi_i)_{i = 1}^{N}$ of points in $R$, sampled independently from the same distribution on $R$.
Then, consider the total squared distance to $p$,
\begin{align}
    \sum_{i = 1}^{N} d_M(f(\xi_i), p)^{2}.\label{eq:distance_sum}
\end{align}
A $p$ that minimizes~\cref{eq:distance_sum} is called a \emph{Karcher mean}, or \emph{Riemannian center of mass}, of $(f(\xi_i))_{i = 1}^{N}$~\cite{Karcher1977}.
\new{%
There are many situations where it is guaranteed to be unique, given the data $\xi_i$.
For example, it is unique if the data is contained in a geodesic ball whose radius is less than half the injectivity radius of the exponential map.%
\footnote{This can be shown by combining Klingenberg's lemma~\cite[lemma 6.4.7]{Petersen16} with Kendall's uniqueness result~\cite[(8.5)]{Arnaudon13}.}
In practice, the Karcher mean can be found via gradient descent~\cite[section 7.3.4]{Nielsen13}, but also approximated efficiently using successive geodesic interpolation~\cite{Chakraborty20,Chakraborty15,Chakraborty19,Cheng16}, namely
}
\begin{align}
	m(p_1) ={}& p_1,\\
	m(p_1, \dots, p_N) ={}& \operatorname{exp}_{p_N} \mleft(\frac{N - 1}{N} \operatorname{log}_{p_N} m(p_1, \dots, p_{N - 1}) \mright).
\end{align}
\new{%
For a general Riemannian manifold $M$, $m(p_1, \dots, p_N)$ depends on the order of $p_1$, \dots, $p_N$, but note that in the Euclidean case it is just the arithmetic mean of $(p_i)_{i = 1}^{N}$, and thus exactly the Karcher mean.
}

In our implementation, we choose $p$ as $m(f(\xi_1), \dots, f(\xi_N))$ with $(\xi_i)_{i = 1}^{N}$ sampled from a uniform distribution on $R$.

\subsection{Tensorized Chebyshev interpolation}%
\label{sub:Tensorized Chebyshev interpolation}

For the second step in our implementation, we choose \emph{tensorized Chebyshev interpolation}.
Tensorizing a univariate approximation scheme is a standard technique for multivariate approximation.
\new{%
Here, we recall Schultz's~\cite{Schultz1969} approach.
Let $C(A \to B)$ denote the space of continuous functions from $A$ to $B$ and say, for example, that we have an approximation scheme,
\begin{align}
	\mathcal{F} \from C([-1, 1] \to \reals) \to C([-1, 1] \to \reals),
\end{align}
for continuous univariate functions.
We identify an approximation scheme,
\begin{align}
	\underbrace{\mathcal{F} \otimes \dots \otimes \mathcal{F}}_{\times m} \from C([-1, 1]^{m} \to \reals) \to C([-1, 1]^{m} \to \reals),
\end{align}
for continuous multivariate functions via $C([-1, 1] \to \reals)^{\otimes m} = C([-1, 1]^{m} \to \reals)$.
}

\new{%
Recall that the Chebyshev interpolation scheme $\mathcal{I}_{n}$ interpolates a function on $[-1, 1]$ in the \emph{Chebyshev nodes} $t_k = \cos( (2 k - 1) \pi / (2 n) )$, $k = 1$, \dots, $n$.
Chebyshev interpolation has many nice properties.
Specifically, the approximation error and operator norm are bounded.
See for example~\cite[Sections 7, 8, and 15]{Trefethen13}.
Combining \cite[Theorem 2.1]{Schultz1969} with \cite[Theorem 8.2]{Trefethen13} gives a concrete error bound for tensorized Chebyshev interpolation.
}

\new{%
\begin{corollary}\label{prop:multivariate_chebyshev_error_bound_for_analytic_functions}
	Let $g \from [-1, 1]^{m} \to \reals$.
	Assume $g$ is analytic in its $k$th argument with an analytic continuation inside the Bernstein ellipse $(\rho_k \exp{i \theta} + \inverse{\rho_k} \exp{-i \theta}) / 2$ bounded by $C_k$.
	Define $\mathhat{g} = (\mathcal{I}_{N_1} \otimes \dots \otimes \mathcal{I}_{N_m}) (g + e)$, where $e$ represents some evaluation error.
	Then $\mathhat{g} \from [-1, 1]^{m} \to \reals$ satsifies
	\begin{align}
		\norm{g - \mathhat{g}}_\infty \leq{}&
			\frac{4 C_1}{(\rho_1 - 1) \rho_1^{N_1}}
			+ \frac{4 \Lambda_{N_1} C_2}{(\rho_2 - 1) \rho_2^{N_2}}
			+ \dots
			+ \frac{4 \Lambda_{N_1} \cdots \Lambda_{N_{m - 1}} C_m}{(\rho_m - 1) \rho_m^{N_m}} \nonumber\\
		& + \Lambda_{N_1} \cdots \Lambda_{N_m} \norm{e}_{\infty},%
		\label{eq:multivariate_chebyshev_error_bound_for_analytic_functions}
	\end{align}
	where $\Lambda_{N} \leq \frac{2}{\pi} \log(N + 1) + 1$ and $\norm{e}_{\infty}$ is the max norm of $e$.
	Especially, if $N_1 = \dots = N_m = N$, $\rho_1 = \dots = \rho = \rho$, and $C_1 = \dots = C_m = C$, then
	\begin{align}\label{eq:multivariate_chebyshev_error_bound_for_analytic_functions_simplified}
		\norm{g - \mathhat{g}}_\infty \leq{}&
			\frac{4 (\Lambda_N^{m} - 1) C}{(\rho - 1) \rho^{N} (\Lambda_N - 1)}
			+ \Lambda_N^m \norm{e}_{\infty}.
	\end{align}
\end{corollary}
}

Note that these bounds for tensorized Chebyshev approximation are slightly better than those obtained by Mason~\cite[(21) and (22)]{Mason1980}.
In particular, tensorized Chebyshev approximation yields quasi-optimal approximations of continuous functions by bounded-degree polynomials~\cite{Mason1980}.

The approximant $\mathhat{g}$ in \cref{prop:multivariate_chebyshev_error_bound_for_analytic_functions} must in practice be constructed via some tensor decomposition of the evaluation tensor $G_{i_1 \dots i_m} = g(t_{i_1}, \dots, t_{i_m})$.
Unfortunately, this suffers from the curse of dimensionality because it requires access to all $N_1 \cdots N_m$ entries of $G$.
To mitigate this, a common assumption is that $G$ can be well approximated by a tensor $\mathhat{G} \in \reals^{N_1 \times \dots \times N_m}$ that admits a data-sparse tensor decomposition.
Consider, for example, a \emph{Tucker decomposition} $\mathhat{G} = (U^{(1)} \otimes \dots \otimes U^{(m)}) \cdot C$, where $C \in \reals^{r_1 \times \dots \times r_m}$ and each $U^{(k)} \in \reals^{N_k \times r_k}$, $k = 1$, \dots, $m$, is a matrix.
Discrete tensors containing smooth function evaluations can be well approximated by a Tucker decomposition with small $r_k$'s~\cite{ST2021}.
\new{%
Given such a Tucker decomposition of $G$, we can construct $\mathhat{g}$ as
\begin{align}
	\mathhat{g}(x) &= (\mathcal{I}_{N_1} \otimes \dots \otimes \mathcal{I}_{N_m})((U^{(1)} \otimes \dots \otimes U^{(m)}) \cdot C)(x)\\
 	&= \sum_{a_1 = 1}^{r_1} \cdots \sum_{a_m = 1}^{r_m} C_{a_1 \dots a_m} h^{a_1}_1(x_1) \cdots h^{a_m}_m(x_m), \label{eq:construct_hatg}
\end{align}
where $h^{a}_{k} \in \reals^{N_k}$ is the univariate Chebyshev interpolant constructed from the $a$th column of $U^{(k)} \in \reals^{N_k \times r_k}$.
Each factor is thus approximated independently.
}

Evidently, the curse of dimensionality is only truly circumvented if a data-sparse model can also be constructed from a \emph{sparse} sample of the tensor $G$, i.e., from viewing only a few entries of $G$.
In theory, this is possible for any polynomial-based decomposition model with a number of samples that equals the dimensionality of the model; see \cite{BGMV2023,RWX2021} for precise statements.
In practice, compressed sensing approximation schemes have been proposed for several tensor decomposition models; see, among others,~\cite{Oseledets2010,BGK2013,GKK2015,Steinlechner2016,KSV2014,STDdLS2013}.

The foregoing multivariate approximation scheme is naturally extended to approximate vector-valued $g$ by applying it component-wise.
We have chosen to present \cref{alg:manifold_approximation_scheme} with a basic orthogonal Tucker decomposition~\cite{Lathauwer2000} as data-sparse model, computed with the sequentially truncated higher-order singular value decomposition (ST-HOSVD) approximation algorithm~\cite{Vannieuwenhoven2012} from the full tensor.

\section{Numerical experiments}%
\label{sec:Numerical experiments}

To illustrate \cref{alg:manifold_approximation_scheme} and to verify that the approximant it produces satisfies the error bounds predicted by \cref{thm:manifold_error_bound}, we present some examples inspired by applications in linear algebra.
All experiments were run on an HP EliteBook x360 830 G7 Notebook with 32 GiB memory and an Intel Core i7-10610U CPU using Julia 1.9.3 and \texttt{ManiFactor.jl} 0.1.0 on Ubuntu 22.04.3 LTS.

\subsection{Implementation details and Julia package}%
\label{sub:Julia package}

\Cref{alg:manifold_approximation_scheme} has been implemented in our Julia package, \texttt{ManiFactor.jl}, which is available online along with the experiments from this section~\cite{manifactor}.
Our package depends on the package \texttt{Manifolds.jl} by Axen, Baran, Bergmann, and Rzecki~\cite{Axen2023}.
Many of the exponential and logarithmic maps for the manifolds in \cref{sec:Manifolds with lower bounded sectional curvature} are implemented there.

Since any vector approximation scheme can be used in our template from \cref{sub:Contribution}, \texttt{ManiFactor.jl} has been implemented to allow for any such choice.
There are many other interesting alternatives to tensorized Chebyshev interpolation.
We just mention a few here that are also decomposition-based: \cite{BM2002,BM2005,GKM2019,BEM2016}.

By default, \texttt{ManiFactor.jl} uses a Chebyshev interpolation from \texttt{ApproxFun.jl} by Olver and Townsend~\cite{Olver14}, combined with a ST-HOSVD from \texttt{TensorToolbox.jl} by Lana Periša~\cite{Perisa23}.

\subsection{An example from Krylov subspaces}%
\label{sub:An example from Krylov subspaces}

In this example we approximate a map into the \emph{Grassmannian manifold} $\mathrm{Gr}(n, k)$ of $k$-dimensional linear subspaces of $\reals^{n}$.
See \cref{sub:Homogeneous manifolds} for a discussion of its geometry.
The example is inspired by Zimmermann's ~\cite[section 4.5]{Zimmermann22b}discussion of Grassmannians in the context of model order reduction, and Benner, Gugercin, and Willcox's ~\cite{Benner15}discussion in the context of Krylov subspace methods.
\new{%
See also Soodhalter, de Sturler, and Kilmer's~\cite{Soodhalter20} recent survey on \emph{Krylov subspace recycling}.
}

\new{%
Consider the parametrized differential equation on $[0, 1]$,
\begin{align}\label{eq:differential_equation_to_be_discretized}
	\derivative[2]{y}{t} + x_1 \derivative{y}{t} + x_2 y = t (1 - t),\quad
	y(0) = y(1) = 0,
\end{align}
where $x_1$ and $x_2$ are parameters.
This can model, for example, heat transfer in an uninsulated rod.
We can discretize the equation by dividing $[0, 1]$ into $n$ equal parts.
Denoting $\Delta = 1 / n$, $Y = (y(i \Delta))_{i = 0}^{n}$, and $B = ((i \Delta) (1 - i \Delta))_{i = 0}^{n}$ and using a finite difference scheme to estimate $\derivative*[2]{y}{t}$ and $\derivative*{y}{t}$, we have
\begin{align}
	\frac{Y_{i - 1} - 2 Y_i + Y_{i + 1}}{\Delta^{2}} + x_1 \frac{-Y_{i - 1} + Y_{i + 1}}{2 \Delta} + x_2 Y_i ={}& B_i,\\
	A(x_1, x_2) Y ={}& B,
\end{align}
where
\begin{align}
	A(x_1, x_2) = 
	\frac{1}{\Delta^{2}}
	\begin{bmatrix}
		-2	& 1\\
		1 	& -2 	 	& 1\\
			& \ddots 	& \ddots 	& \ddots\\
			&			& 1 		& -2 	 	& 1\\
			&			&			& 1		& -2
	\end{bmatrix}
	+ \frac{x_1}{2 \Delta}
	\begin{bmatrix}
		0	& 1\\
		-1 	& 0 	 	& 1\\
			& \ddots 	& \ddots 	& \ddots\\
			&			& -1 		& 0 	 	& 1\\
			&			&			& -1		& 0
	\end{bmatrix}
	+ x_2
	\begin{bmatrix}
		1	& \\
			& \ddots 	&\\
			&			& 1
	\end{bmatrix}.
\end{align}
}

\new{%
Krylov subspace methods solve high-dimensional linear equations like these by looking for a solution in a \emph{Krylov subspace} $\operatorname{span} \set{v, A v, \dots, A^{k - 1} v}$.
Typically, $A$ is also multiplied by some \emph{preconditioner} $P^{-1}$, so that the Krylov subspace is
\begin{align}
	\mathcal{K}_k(A, v) = \operatorname{span} \set{v, P^{-1} A v, \dots, (P^{-1} A)^{k - 1} v}.
\end{align}
The tighter that the eigenvalues of $P^{-1} A$ are clustered, the faster the error in the Krylov subspace converges.
In our numerical experiments, we use a generic algebraic multigrid preconditioner from Julia package \texttt{Preconditioners.jl}, but problem-specific preconditioners are available for most classic differential equations.
}

\new{%
If we fix $k$ and $v \in \reals^{n}$, then 
\begin{align}\label{eq:map_into_grassmannian}
	f(x_1, x_2) = \mathcal{K}_k(A(x_1, x_2), v)
\end{align}
is a map from $\reals^{2}$ into $\mathrm{Gr}(n, k)$.
Rather than computing the Krylov subspace from its definition for each $(x_1, x_2)$, we want to be able to quickly compute an approximate Krylov subspace.
}

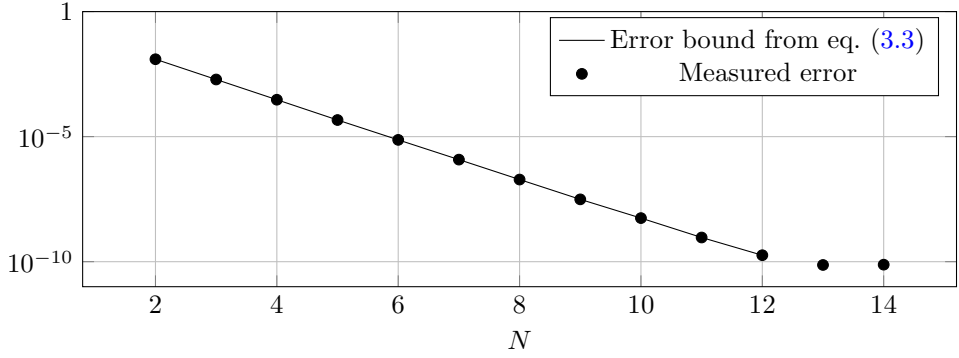
\begin{figure}[t]
	\centering
	\begin{tikzpicture}
	\begin{semilogyaxis}[
	xlabel=$N$,
	ytick={1e0,1e-5,1e-10},
	yticklabels={\num{e0},\num{e-5},\num{e-10}},
	xtick={2, 4, 6, 8, 10, 12, 14, 16},
	xticklabels={2, 4, 6, 8, 10, 12, 14},
	grid=major,
	width=1.0\textwidth,
	height=0.4\textwidth,
	ymax=1e0,
	ymin=1e-11,
	]
	\pgfplotstableread[row sep=newline,col sep=comma]{
		Ns,es,bs
		2,0.012449186576274142,0.012451547618608956
		3,0.0019284211881600395,0.0019292386213432555
		4,0.0002994124599646736,0.0002996394072707707
		5,4.619579918720073e-5,4.624860236753842e-5
		6,7.422432240689437e-6,7.430347040709181e-6
		7,1.205633769848871e-6,1.2072007813166879e-6
		8,1.9242457284465447e-7,1.9270505791624848e-7
		9,3.105477331742593e-8,3.1102151639431605e-8
		10,5.5320504831485664e-9,5.5391514151364385e-9
		11,9.271930801918417e-10,9.273093056242986e-10
		12,1.8160030123734654e-10,1.8074223897984137e-10
		13,7.395544549256469e-11,nan
		14,7.605369791824881e-11,nan
	}\mydata
	\addplot[no marks] table [x=Ns, y=bs] {\mydata};
	\addplot[only marks, mark=*] table [x=Ns, y=es] {\mydata};
	\legend{Error bound from \cref{eq:manifold_error_bound1}\\Measured error\\}
	\end{semilogyaxis}
	\end{tikzpicture}
	\caption{Approximation error for the map in \cref{eq:map_into_grassmannian} into the Grassmannian $\mathrm{Gr}(n, k)$, compared against what is predicted by \cref{thm:manifold_error_bound}, for different numbers of Chebyshev nodes $N$.}
	\label{fig:krylov_subspace}
\end{figure}

\new{%
In this example, we approximate $f$ on $R = [1, 2]^{2}$ with $n = 1000$, $k = 5$, and $v = B$.
These parameters are chosen so that $A Y = B$ can be solved with forward error on the order of $\num{e-7}$ using $\mathcal{K}_k(A, v)$, and so that the original differential equation, \cref{eq:differential_equation_to_be_discretized}, can be solved with error on the order of $\num{e-5}$.
}

We now proceed as in \cref{alg:manifold_approximation_scheme}.
Choose $p$ and let $g = {\operatorname{log}_{p}} \circ f$.
Then consider the polynomial $\mathhat{g}$ that interpolates $g$ in $(N + 1)^{2}$ Chebyshev nodes on $R$.
Finally, let $\mathhat{f} = {\operatorname{exp}_{p}} \circ {\mathhat{g}}$.

Since both $f$ and the manifold logarithm are analytic, and because of \cref{prop:multivariate_chebyshev_error_bound_for_analytic_functions}, we expect the error $\norm{g - \mathhat{g}}_p$ to decay at least exponentially in $N$.
Indeed, the numerical experiment presented in \cref{fig:krylov_subspace} supports this guess.
In this experiment, we restrict $f$ to a discrete subset $R' \subset R$ of $\num{1000}$ points, sampled independently from a uniform distribution on $R$.
We are thus able to use $\epsilon = \max_{x \in R'} \norm{g(x) - \mathhat{g}(x)}_p$ in \cref{eq:manifold_error_bound1} to derive a bound for $d(f(x), \mathhat{f}(x))$.
\Cref{fig:krylov_subspace} shows that this error bound is confirmed by the experiment.
After $N = 12$, rounding causes the errors to plateau and the error bound stops being meaningful.
The full code to produce \cref{fig:krylov_subspace} is available in the \texttt{ManiFactor.jl} package~\cite[\texttt{Example3.jl}]{manifactor}

\subsection{An example from dynamic low-rank approximation}%
\label{sub:An example from dynamic low-rank approximation}

Expanding on a synthetic example by Ceruti and Christian~\cite{Ceruti22}, we consider
\begin{align}
	A(x_1, x_2, x_3) = \exp{x_1} \exp{x_2 W_1} D \transpose{\left( \exp{x_3 W_2} \right)},
\end{align}
where $D$ is an $n \times n$ diagonal matrix with diagonal entries $1$, $2^{-1}$, \dots, $2^{1 - n}$, and $W_1$ and $W_2$ are skew-symmetric matrices, with entries sampled independently from a uniform distribution on $[-1, 1]$, and then normalized.
Recall that when $W$ is skew-symmetric, $\exp{W}$ is orthogonal.
The best rank-$1$ approximation to $A$ is
\begin{align}\label{eq:map_into_segre}
	f(x) = \exp{x_1} \exp{x_2 W_1} e_1 \transpose{\left( \exp{x_3 W_2} e_1 \right)},
\end{align}
where $e_1$ is the column vector $(1, 0, \dots, 0)$.
Thus $f$ is a map from $\reals^{3}$ into the \emph{Segre manifold} of rank-$1$ matrices $\operatorname{Seg}(\reals^{n} \times \reals^{n})$.
See \cref{sub:Segre manifold} for a discussion of its geometry.
We approximate $f$ on $R = [-1, 1]^{3}$ with $n = 100$.

Similarly to \cref{sub:An example from Krylov subspaces}, we proceed as \cref{alg:manifold_approximation_scheme}, while restricting to a discrete subset $R' \subset R$ of $\num{1000}$ points sampled independently from a uniform distribution on $R$.
Again, because of \cref{prop:multivariate_chebyshev_error_bound_for_analytic_functions}, we expect the error to decay exponentially in the number of Chebyshev nodes $N = N_1 = N_2 = N_3$.
By \cref{eq:segre_sectional_curvature}, we have a lower bound $H = -\exp{2}$ for the sectional curvature.
Thus \cref{thm:manifold_error_bound} implies that the approximation error is bounded by \cref{eq:manifold_error_bound2}.

The numerical experiment presented in \cref{fig:closest_rank_one_matrix} again confirms the error bound.
After $N = 13$, rounding causes the errors to plateau and the error bound stops being meaningful.
The full code to product \cref{fig:closest_rank_one_matrix} is available in the \texttt{ManiFactor.jl} package~\cite[\texttt{Example5.jl}]{manifactor}.

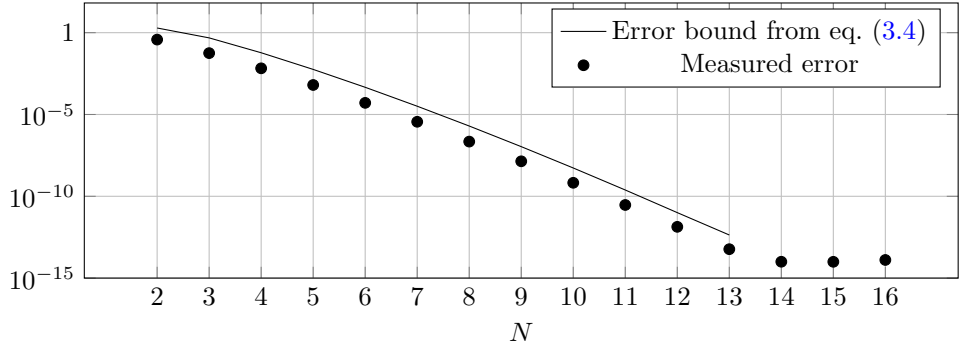
\begin{figure}[t] 
	\centering
	\begin{tikzpicture}
	\begin{semilogyaxis}[
	xlabel=$N$,
	ytick      ={1e0,1e-5,1e-10,1e-15},
	yticklabels={\num{e0},\num{e-5},\num{e-10},\num{e-15}},
	xtick      ={2, 3, 4, 5, 6, 7, 8, 9, 10, 11, 12, 13, 14, 15, 16},
	xticklabels={2, 3, 4, 5, 6, 7, 8, 9, 10, 11, 12, 13, 14, 15, 16},
	grid=major,
	width=1.0\textwidth,
	height=0.4\textwidth,
	ymin=1e-15,
	]
	\pgfplotstableread[row sep=newline,col sep=comma]{
		Ns,es,bs
		2,0.37820712940694434,1.9358418368781574
		3,0.05632353442849433,0.48088925322957443
		4,0.006633818069585923,0.05934192983083497
		5,0.0006358810042859098,0.005692457287231779
		6,5.135181047233317e-5,0.0004597093484987388
		7,3.583409012953188e-6,3.203810896051839e-5
		8,2.2019144507723354e-7,1.972220114736185e-6
		9,1.3744444175911857e-8,1.0824906422111038e-7
		10,6.682068832633296e-10,5.372621912426975e-9
		11,2.942243904266294e-11,2.4239386948759377e-10
		12,1.3580222079831663e-12,1.0029074699474928e-11
		13,5.843822938840181e-14,4.30536288636737e-13
		14,9.982228716336735e-15,nan
		15,9.893303102138232e-15,nan
		16,1.2821773727292417e-14,nan
	}\mydata;
	\addplot[no marks] table [x=Ns, y=bs] {\mydata};
	\addplot[only marks,mark=*] table [x=Ns, y=es] {\mydata};
	\legend{Error bound from \cref{eq:manifold_error_bound2}, Measured error};
	\end{semilogyaxis}
	\end{tikzpicture}
	\caption{Approximation error for the map in \cref{eq:map_into_segre} into the Segre manifold $\operatorname{Seg}(\reals^{n} \times \reals^{n})$, compared against what is predicted by \cref{thm:manifold_error_bound}, for different numbers of Chebyshev nodes $N$ in each variable.}
	\label{fig:closest_rank_one_matrix}
\end{figure}

\subsection{Effects of using a retraction}%
\label{sub:Effects of retraction}

\begin{figure}[t]
	\centering
	\begin{tikzpicture}
	\begin{semilogyaxis}[
	xlabel=$N$,
	ytick={1e0,1e-5,1e-10},
	yticklabels={\num{e0},\num{e-5},\num{e-10}},
	xtick={2, 4, 6, 8, 10, 12, 14, 16},
	xticklabels={2, 4, 6, 8, 10, 12, 14},
	grid=major,
	width=1.0\textwidth,
	height=0.4\textwidth,
	ymax=1e0,
	ymin=1e-11,
	]
	\pgfplotstableread[row sep=newline,col sep=comma]{
		Ns,es0,es1,es2
		2,0.012449186576274142,0.012549935050241481,0.012549935050241478
		3,0.0019284211881600395,0.002004418407215595,0.0020044184072155963
		4,0.0002994124599646736,0.0003196102439274483,0.00031961024392745854
		5,4.619579918720073e-5,5.061164559015156e-5,5.061164559016016e-5
		6,7.422432240689437e-6,8.289640844933448e-6,8.289640844949548e-6
		7,1.205633769848871e-6,1.3601517663312395e-6,1.3601517663321636e-6
		8,1.9242457284465447e-7,2.1649011674537205e-7,2.1649011674744278e-7
		9,3.105477331742593e-8,3.441244440645916e-8,3.4412444421722167e-8
		10,5.5320504831485664e-9,5.9995096430036865e-9,5.999509638587159e-9
		11,9.271930801918417e-10,9.682403440670558e-10,9.682403435056248e-10
		12,1.8160030123734654e-10,1.7746517623065015e-10,1.77465190148819e-10
		13,7.395544549256469e-11,7.284215590821134e-11,7.284233262901796e-11
		14,7.605369791824881e-11,7.597657383584195e-11,7.597659016176408e-11
	}\mydata
	\addplot[only marks, mark=*] table [x=Ns, y=es0] {\mydata};
	\addplot[only marks, mark=triangle, mark size=4] table [x=Ns, y=es1] {\mydata};
	\addplot[only marks, mark=square, mark size=3] table [x=Ns, y=es2] {\mydata};
	\legend{Error using exponential map\\Error using QR retraction\\Error using polar retraction\\}
	\end{semilogyaxis}
	\end{tikzpicture}
	\caption{Approximation error for the map in \cref{eq:map_into_grassmannian}, using different retraction methods.}
	\label{fig:krylov_subspace_reloaded}
\end{figure}
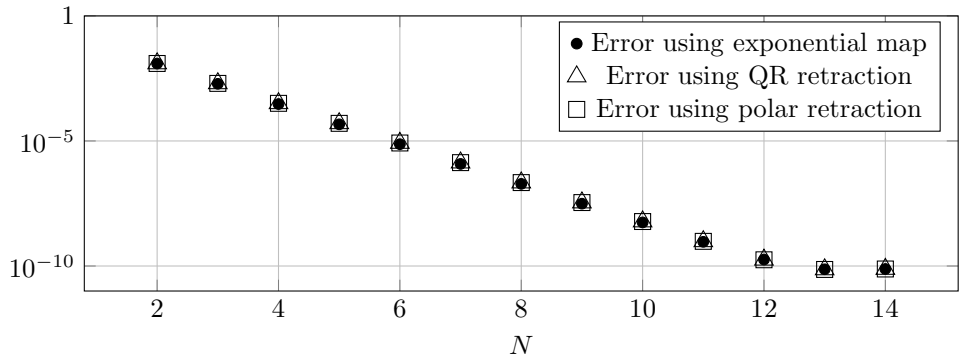

\new{%
In \cref{sub:Retractions}, we discussed how a retraction and inverse retraction may be used instead of exponential and logarithmic maps.
For manifolds where the exponential or logarithmic maps are not known, we have no choice but to retract.
However, it can also be interesting to do so even when the geodesics are known, since it may be faster or more numerically stable.
}

\new{%
\texttt{Manifolds.jl} implements two retractions on the Grassmannian manifold, the \emph{QR retraction} and the \emph{polar retraction}.
They are described in detail by Absil et al.~\cite[section 4.1]{Absil08}.
We are now interested in how replacing the exponential map affects the Riemannian error.
Recall the setup from \cref{sub:An example from Krylov subspaces}.
In \cref{fig:krylov_subspace_reloaded}, we recreate the experiment from \cref{fig:krylov_subspace} and compare with using the retractions.
There is essentially no difference in the errors.
}

\new{%
We are also interested in how the time to evaluate $\mathhat{f}$ is affected.
Let $N = 10$ and let $\mathhat{G}$ be a ST-HOSVD of $G$ with core dimensions $(r_1, r_2, r_3) = (5, 5, 5)$.
Using \texttt{btime} method from \texttt{BenchmarkTools.jl}, we observe
\begin{itemize}
	\item \SI{2.45e-4}{\second} to evaluate $\mathhat{f}$ with exponential map,
	\item \SI{1.57e-4}{\second} to evaluate $\mathhat{f}$ with QR retraction,
	\item \SI{1.55e-4}{\second} to evaluate $\mathhat{f}$ with polar retraction.
\end{itemize}
Retracting in this case thus brings a significant speedup.
}

\new{%
The full code to product \cref{fig:krylov_subspace_reloaded} and the timings are available in the \texttt{ManiFactor.jl} package~\cite[\texttt{Example4.jl}]{manifactor}.
}

\section{Conclusions}
\label{sec:conclusions}

In conclusion, we find that the approximability of maps into a Riemannian manifold is related to the curvature of the manifold.
Moreover, for many manifolds that are relevant for numerical analysis, especially matrix manifolds, explicit curvature bounds allow us to derive explicit error bounds.

From here, several areas of further research are visible.
First, one could derive the geodesics and sectional curvature for more manifolds.
For example spaces of tensors with fixed rank or multilinear rank.
Second, one could develop more efficient and more optimal ways of choosing the point $p$ from where the manifold is linearized.
The current way, geodesic interpolation of sample evaluations, is a heuristic to estimate the Karcher mean of those evaluations, which in turn is a heuristic to minimize $\sigma$ in \cref{thm:manifold_error_bound}.
Third, it could be possible to consider more than one tangent space.
Given several approximations of $f$, based around several points, what is a good estimate of $f(x)$, and what are good ways of choosing such points?

\appendix

\section{Manifolds with lower bounded sectional curvature}
\label{sec:Manifolds with lower bounded sectional curvature}

In this section, we list Riemannian manifolds with explicit lower bounds for their sectional curvature $K$ and whose manifold exponential and logarithm are known.
For these manifolds, \cref{thm:manifold_error_bound} gives an explicit bound for the error on the manifold in terms of the error on the tangent space.
While sectional curvature is always bounded locally, for many manifolds that are relevant for applications we can find explicit expressions for those local bounds.
In many cases, the sectional curvature is even bounded globally, so that we can choose the same $H$ for every chart in \cref{thm:manifold_error_bound}.

\subsection{Lie groups}%
\label{sub:Lie groups}

Given an inner product on its Lie algebra, one can extend it to a unique Riemannian metric on the Lie group by demanding invariance under left (or right) translation.
Because of this invariance, the set of sectional curvatures at a point is independent of that point.
Therefore, a lower bound for the sectional curvature at one point will be a global lower bound.
There is an expression for the sectional curvature of a Lie group in terms of its Lie bracket~\cite[Proposition 3.18 (3)]{Cheeger08}.
We only note that, for matrix groups whose Lie algebra has the Euclidean inner product, it follows from this expression that%
\footnote{Using $\norm{[u, v]} \leq \sqrt{2} \norm{u} \norm{v}$, which also implies $\norm{\operatorname{ad}_u}_{\mathrm{op}} \leq \sqrt{2} \norm{u}$.}%
\begin{align}
	-\frac{11}{2} \leq K.
\end{align}

For example, $\mathrm{GL}(n)$ is the group of invertible $n \times n$ matrices.
Its metric is $\innerproduct{u}{v}_A = \operatorname{tr} \mleft( \transpose{\left( \inverse{A} u \right)} \inverse{A} v \mright)$,
where $u$, $v \in T_A \mathrm{GL}(n) = \reals^{n \times n}$.
Locally, all (finite-dimensional) Lie groups are Lie subgroups of some $\mathrm{GL}(n)$.
Andruchow, Larotonda, Recht, and Varela~\cite{Andruchow14} derive its geodesics.

The compact Lie groups and their products with vector spaces are precisely the Lie groups that admit a \emph{bi-invariant} metric, i.e., a metric that is invariant under both left and right translation.
This is also equivalent to the manifold exponential agreeing with the Lie group exponential.
The sectional curvature is positive for such manifolds~\cite[Corollary 3.19 and Proposition 3.34]{Cheeger08}.

The classification of compact Lie groups is a classic result.
The real compact Lie groups are~\cite[Theorem 2.17]{Arvanitogeorgos03}
\begin{itemize}
	\item tori $S^{1} \times \dots \times S^{1}$,
	\item orthogonal groups $\mathrm{O}(n)$ and special orthogonal groups $\mathrm{SO}(n)$,
	\item unitary groups $\mathrm{U}(n)$ and special unitary groups $\mathrm{SU}(n)$,
	\item spin groups $\mathrm{Spin}(n)$,
	\item symplectic groups $\mathrm{Sp}(n)$,
	\item compact exceptional Lie groups $G_2$, $F_4$, $E_6$, and $E_8$,
	\item products, disjoint unions, and finite covers of the above.
\end{itemize}

\subsection{Homogeneous spaces}%
\label{sub:Homogeneous manifolds}

A \emph{Riemannian homogeneous space} is a Riemannian manifold equipped with some transitive action that the metric is invariant under.
Similarly to Lie groups, the curvature is then the same at every point.
Any Riemannian homogeneous space $X$ is of the form $G / H$ where $G$ is the Lie group that acts on $X$ and $H$ is a Lie subgroup.
Moreover, the curvature of $X$ is always greater than or equal to the curvature of $G$.

The \emph{Stiefel manifold} $\mathrm{St}(n, k) = \mathrm{O}(n) \mathbin{/} \mathrm{O}(n - k)$ is the space of orthogonal $n \times k$ matrices.
Zimmermann~\cite{Zimmermann17} considers efficient geodesic computations on the Stiefel manifold.

The \emph{Grassmannian manifold} $\mathrm{Gr}(n, k) = \mathrm{O}(n) \mathbin{/} (\mathrm{O}(n - k) \times \mathrm{O}(k))$ is the space of $k$-dimensional linear subspaces of $\reals^{n}$.
Bendokat, Zimmermann, and Absil~\cite{Bendokat23} consider efficient geodesic computations on the Grassmannian manifold.

For the Stiefel and Grassmannian manifolds, their metrics as Riemannian homogeneous spaces are called \emph{canonical metrics} to distinguish them from their metrics as embedded submanifolds of $\reals^{n \times n}$.
With the canonical metric, they both have nonnegative curvature.

\new{%
Let $0 = d_0 < d_1 < \dots < d_k = n$ be an increasing integer sequence.
The \emph{flag manifold} $\mathcal{F}_{(d_0, \dots, d_{k})} = \mathrm{O}(n) \mathbin{/} (\mathrm{O}(d_1 - d_0) \times \dots \times \mathrm{O}(d_k - d_{k - 1}))$ is the space of increasing sequences of linear subspaces in $\reals^{n}$.
It has nonnegative curvature.
}

$\sigma_B \from A \mapsto B A \transpose{B}$ defines a transitive action of $\mathrm{GL}(n)$ on the space $\mathcal{P}_n$ of \emph{symmetric positive definite $n \times n$ matrices}.
The \emph{affine invariant metric} on $\mathcal{P}_n$ is thus derived from the quotient $\mathcal{P}_n = \mathrm{GL}(n) \mathbin{/} \mathrm{O}(n)$.
See Bhatia~\cite[Theorem 6.1.6]{Bhatia07} for an expression for its geodesics.
It has nonpositive curvature.

Similarly to $\mathcal{P}_n$, $\sigma_B$ defines a transitive action of $\mathrm{GL}^{+}(n)$, the group of matrices with positive determinant, on the space $\mathrm{S}_+(n, k)$ of \emph{rank-$k$ positive semidefinite $n \times n$ matrices}.
We have that
\begin{align}
	\mathrm{S}_+(n, k) = \mathrm{GL}^{+}(n) \mathbin{/}
	\begin{bmatrix}
		\mathrm{SO}(k) & \reals^{k \times (n - k)}\\
		& \mathrm{GL}^{+}(n - k)
	\end{bmatrix}.
\end{align}
Vandereycken, Absil, and Vandewalle~\cite{Vandereycken12} use a right-invariant metric on $\mathrm{GL}^{+}(n)$ to derive geodesics on $\mathrm{S}_+(n, k)$.

We also mention another metric that can be put on the space of rank-$k$ positive semidefinite matrices.
Let $\reals^{n \times k}_{*}$ be the space of full rank $n \times k$ matrices.
Any rank-$k$ positive semidefinite matrix can be written as $Y \transpose{Y}$, $Y \in \reals^{n \times k}_{*}$.
Such a representation is unique up to right multiplication by an orthogonal matrix.
Hence, we define $\mathcal{S}_{+}(n, k) = \reals^{n \times k}_{*} \mathbin{/} \mathrm{O}(k)$, where $\reals^{n \times k}_{*}$ is equipped with the Euclidean metric.
Although this manifold is not a Riemannian homogeneous space, a metric may still be induced by demanding that the quotient map is a Riemannian submersion.
Massart and Absil~\cite{Massart20} summarize its geometry.
It has nonnegative curvature.

\subsection{Segre manifold}%
\label{sub:Segre manifold}

The \emph{Segre manifold}, $\mathrm{Seg}(\reals^{N_1} \times \dots \times \reals^{N_m})$, is the space of rank-$1$ tensors in the product space $\reals^{N_1} \otimes \dots \otimes \reals^{N_m}$ \cite{Harris1992}.
It is immersed by
\begin{align}
	\operatorname{Seg} \from \reals^{+} \times S^{N_1 - 1} \times \dots \times S^{N_m - 1} &\to \reals^{N_1} \otimes \dots \otimes \reals^{N_m},\\
	(\lambda, x_1, \dots, x_m) &\mapsto \lambda x_1 \otimes \dots \otimes x_m.
\end{align}
Swijsen, van der Veken, and Vannieuwenhoven~\cite{Swijsen2021} derive the exponential map for this manifold in the metric induced by the immersion.
Swijsen~\cite[theorem 6.2.1]{Swijsen22} derive the manifold logarithm.
\new{%
Both the exponential and logarithmic map are shown to be efficient to evaluate.
}

Moreover, Swijsen~\cite[Corollary 6.2.4]{Swijsen22} shows that the sectional curvature satisfies
\begin{align}
	-\frac{1}{\lambda^2} \leq K \leq 0.%
	\label{eq:segre_sectional_curvature}
\end{align}
Note that in this case there is only a local lower bound, namely if the image of $f$ fits in a chart with $\lambda$ lower bounded by $\lambda_*$, then we may choose $H = -1 / \lambda_*^{2}$.

\bibliography{bibliography}


\begin{thebibliography}{96}
\ifx \bisbn   \undefined \def \bisbn  #1{ISBN #1}\fi
\ifx \binits  \undefined \def \binits#1{#1}\fi
\ifx \bauthor  \undefined \def \bauthor#1{#1}\fi
\ifx \batitle  \undefined \def \batitle#1{#1}\fi
\ifx \bjtitle  \undefined \def \bjtitle#1{#1}\fi
\ifx \bvolume  \undefined \def \bvolume#1{\textbf{#1}}\fi
\ifx \byear  \undefined \def \byear#1{#1}\fi
\ifx \bissue  \undefined \def \bissue#1{#1}\fi
\ifx \bfpage  \undefined \def \bfpage#1{#1}\fi
\ifx \blpage  \undefined \def \blpage #1{#1}\fi
\ifx \burl  \undefined \def \burl#1{\textsf{#1}}\fi
\ifx \doiurl  \undefined \def \doiurl#1{\url{https://doi.org/#1}}\fi
\ifx \betal  \undefined \def \betal{\textit{et al.}}\fi
\ifx \binstitute  \undefined \def \binstitute#1{#1}\fi
\ifx \binstitutionaled  \undefined \def \binstitutionaled#1{#1}\fi
\ifx \bctitle  \undefined \def \bctitle#1{#1}\fi
\ifx \beditor  \undefined \def \beditor#1{#1}\fi
\ifx \bpublisher  \undefined \def \bpublisher#1{#1}\fi
\ifx \bbtitle  \undefined \def \bbtitle#1{#1}\fi
\ifx \bedition  \undefined \def \bedition#1{#1}\fi
\ifx \bseriesno  \undefined \def \bseriesno#1{#1}\fi
\ifx \blocation  \undefined \def \blocation#1{#1}\fi
\ifx \bsertitle  \undefined \def \bsertitle#1{#1}\fi
\ifx \bsnm \undefined \def \bsnm#1{#1}\fi
\ifx \bsuffix \undefined \def \bsuffix#1{#1}\fi
\ifx \bparticle \undefined \def \bparticle#1{#1}\fi
\ifx \barticle \undefined \def \barticle#1{#1}\fi
\bibcommenthead
\ifx \bconfdate \undefined \def \bconfdate #1{#1}\fi
\ifx \botherref \undefined \def \botherref #1{#1}\fi
\ifx \url \undefined \def \url#1{\textsf{#1}}\fi
\ifx \bchapter \undefined \def \bchapter#1{#1}\fi
\ifx \bbook \undefined \def \bbook#1{#1}\fi
\ifx \bcomment \undefined \def \bcomment#1{#1}\fi
\ifx \oauthor \undefined \def \oauthor#1{#1}\fi
\ifx \citeauthoryear \undefined \def \citeauthoryear#1{#1}\fi
\ifx \endbibitem  \undefined \def \endbibitem {}\fi
\ifx \bconflocation  \undefined \def \bconflocation#1{#1}\fi
\ifx \arxivurl  \undefined \def \arxivurl#1{\textsf{#1}}\fi
\csname PreBibitemsHook\endcsname

\bibitem[\protect\citeauthoryear{Christensen and
  Christensen}{2004}]{Christensen04}
\begin{bbook}
\bauthor{\bsnm{Christensen}, \binits{O.}},
\bauthor{\bsnm{Christensen}, \binits{K.L.}}:
\bbtitle{Approximation Theory: from {T}aylor Polynomials to Wavelets}.
\bsertitle{Applied and Numerical Harmonic Analysis}.
\bpublisher{Birkh{\"a}user},
\blocation{Boston, MA}
(\byear{2004})
\end{bbook}
\endbibitem

\bibitem[\protect\citeauthoryear{Trefethen}{2019}]{Trefethen13}
\begin{bbook}
\bauthor{\bsnm{Trefethen}, \binits{L.N.}}:
\bbtitle{Approximation Theory and Approximation Practice, Extended Edition}.
\bpublisher{Society for Industrial and Applied Mathematics},
\blocation{Philadelphia, PA}
(\byear{2019}).
\doiurl{10.1137/1.9781611975949} .
\burl{https://epubs.siam.org/doi/abs/10.1137/1.9781611975949}
\end{bbook}
\endbibitem

\bibitem[\protect\citeauthoryear{Temlyakov}{2018}]{Temlyakov18}
\begin{bbook}
\bauthor{\bsnm{Temlyakov}, \binits{V.}}:
\bbtitle{Multivariate Approximation}.
\bsertitle{Cambridge Monographs on Applied and Computational Mathematics}.
\bpublisher{Cambridge University Press},
\blocation{Cambridge, England}
(\byear{2018}).
\doiurl{10.1017/9781108689687}
\end{bbook}
\endbibitem

\bibitem[\protect\citeauthoryear{Wendland}{2004}]{Wendland04}
\begin{bbook}
\bauthor{\bsnm{Wendland}, \binits{H.}}:
\bbtitle{Moving least squares}.
\bsertitle{Cambridge Monographs on Applied and Computational Mathematics},
pp. \bfpage{35}--\blpage{45}.
\bpublisher{Cambridge University Press},
\blocation{Cambridge, England}
(\byear{2004})
\end{bbook}
\endbibitem

\bibitem[\protect\citeauthoryear{Beylkin and Mohlenkamp}{2002}]{BM2002}
\begin{barticle}
\bauthor{\bsnm{Beylkin}, \binits{G.}},
\bauthor{\bsnm{Mohlenkamp}, \binits{M.J.}}:
\batitle{Numerical operator calculus in higher dimensions}.
\bjtitle{Proceedings of the National Academy of Sciences of the United States
  of America}
\bvolume{99}(\bissue{16}),
\bfpage{10246}--\blpage{10251}
(\byear{2002})
\end{barticle}
\endbibitem

\bibitem[\protect\citeauthoryear{Beylkin and Mohlenkamp}{2005}]{BM2005}
\begin{barticle}
\bauthor{\bsnm{Beylkin}, \binits{G.}},
\bauthor{\bsnm{Mohlenkamp}, \binits{M.J.}}:
\batitle{Algorithms for numerical analysis in high dimensions}.
\bjtitle{SIAM Journal on Scientific Computing}
\bvolume{26}(\bissue{6}),
\bfpage{2133}--\blpage{2159}
(\byear{2005})
\end{barticle}
\endbibitem

\bibitem[\protect\citeauthoryear{Bigoni et~al.}{2016}]{BEM2016}
\begin{barticle}
\bauthor{\bsnm{Bigoni}, \binits{D.}},
\bauthor{\bsnm{Engsig-Karup}, \binits{A.P.}},
\bauthor{\bsnm{Marzouk}, \binits{Y.M.}}:
\batitle{Spectral tensor-train decomposition}.
\bjtitle{SIAM Journal on Scientific Computing}
\bvolume{38}(\bissue{4}),
\bfpage{2405}--\blpage{2439}
(\byear{2016})
\doiurl{10.1137/15m1036919}
\end{barticle}
\endbibitem

\bibitem[\protect\citeauthoryear{Hashemi and Trefethen}{2017}]{HT2017}
\begin{barticle}
\bauthor{\bsnm{Hashemi}, \binits{B.}},
\bauthor{\bsnm{Trefethen}, \binits{L.N.}}:
\batitle{Chebfun in three dimensions}.
\bjtitle{SIAM Journal on Scientific Computing}
\bvolume{39},
\bfpage{341}--\blpage{363}
(\byear{2017})
\doiurl{10.1137/16m1083803}
\end{barticle}
\endbibitem

\bibitem[\protect\citeauthoryear{Hashemi and Nakatsukasa}{2018}]{HN2018}
\begin{barticle}
\bauthor{\bsnm{Hashemi}, \binits{B.}},
\bauthor{\bsnm{Nakatsukasa}, \binits{Y.}}:
\batitle{On the spectral problem for trivariate functions}.
\bjtitle{BIT Numerical Mathematics}
\bvolume{58},
\bfpage{981}--\blpage{1008}
(\byear{2018})
\end{barticle}
\endbibitem

\bibitem[\protect\citeauthoryear{Gorodetsky et~al.}{2019}]{GKM2019}
\begin{barticle}
\bauthor{\bsnm{Gorodetsky}, \binits{A.}},
\bauthor{\bsnm{Karaman}, \binits{S.}},
\bauthor{\bsnm{Marzouk}, \binits{Y.}}:
\batitle{A continuous analogue of the tensor-train decomposition}.
\bjtitle{Computer Methods in Applied Mechanics and Engineering}
\bvolume{347},
\bfpage{59}--\blpage{84}
(\byear{2019})
\doiurl{10.1016/j.cma.2018.12.015}
\end{barticle}
\endbibitem

\bibitem[\protect\citeauthoryear{Dolgov et~al.}{2021}]{Dolgov2021}
\begin{barticle}
\bauthor{\bsnm{Dolgov}, \binits{S.}},
\bauthor{\bsnm{Kressner}, \binits{D.}},
\bauthor{\bsnm{Str\"{o}ssner}, \binits{C.}}:
\batitle{Functional {T}ucker approximation using {C}hebyshev interpolation}.
\bjtitle{SIAM Journal on Scientific Computing}
\bvolume{43}(\bissue{3}),
\bfpage{2190}--\blpage{2210}
(\byear{2021})
\doiurl{10.1137/20M1356944}
{\href{https://arxiv.org/abs/https://doi.org/10.1137/20M1356944}{{https://doi.org/10.1137/20M1356944}}}
\end{barticle}
\endbibitem

\bibitem[\protect\citeauthoryear{Strössner et~al.}{2022}]{Strossner2022}
\begin{botherref}
\oauthor{\bsnm{Strössner}, \binits{C.}},
\oauthor{\bsnm{Sun}, \binits{B.}},
\oauthor{\bsnm{Kressner}, \binits{D.}}:
Approximation in the extended functional tensor train format
(2022).
\url{https://arxiv.org/abs/2211.11338}
\end{botherref}
\endbibitem

\bibitem[\protect\citeauthoryear{Str\"{o}ssner and
  Kressner}{2023}]{Strossner2023}
\begin{barticle}
\bauthor{\bsnm{Str\"{o}ssner}, \binits{C.}},
\bauthor{\bsnm{Kressner}, \binits{D.}}:
\batitle{Low-rank tensor approximations for solving multimarginal optimal
  transport problems}.
\bjtitle{SIAM Journal on Imaging Sciences}
\bvolume{16}(\bissue{1}),
\bfpage{169}--\blpage{191}
(\byear{2023})
\doiurl{10.1137/22M1478355}
{\href{https://arxiv.org/abs/https://doi.org/10.1137/22M1478355}{{https://doi.org/10.1137/22M1478355}}}
\end{barticle}
\endbibitem

\bibitem[\protect\citeauthoryear{Zimmermann}{2020}]{Zimmermann20}
\begin{barticle}
\bauthor{\bsnm{Zimmermann}, \binits{R.}}:
\batitle{{H}ermite interpolation and data processing errors on {R}iemannian
  matrix manifolds}.
\bjtitle{SIAM Journal on Scientific Computing}
\bvolume{42}(\bissue{5}),
\bfpage{2593}--\blpage{2619}
(\byear{2020})
\doiurl{10.1137/19M1282878}
{\href{https://arxiv.org/abs/https://doi.org/10.1137/19M1282878}{{https://doi.org/10.1137/19M1282878}}}
\end{barticle}
\endbibitem

\bibitem[\protect\citeauthoryear{Zimmermann and Bergmann}{2022}]{Zimmermann22}
\begin{botherref}
\oauthor{\bsnm{Zimmermann}, \binits{R.}},
\oauthor{\bsnm{Bergmann}, \binits{R.}}:
Multivariate {H}ermite interpolation of manifold-valued data
(2022).
\url{https://arxiv.org/abs/2212.07281}
\end{botherref}
\endbibitem

\bibitem[\protect\citeauthoryear{Chakraborty and Vemuri}{2019}]{Chakraborty19}
\begin{barticle}
\bauthor{\bsnm{Chakraborty}, \binits{R.}},
\bauthor{\bsnm{Vemuri}, \binits{B.C.}}:
\batitle{{Statistics on the {S}tiefel manifold: Theory and applications}}.
\bjtitle{The Annals of Statistics}
\bvolume{47}(\bissue{1}),
\bfpage{415}--\blpage{438}
(\byear{2019})
\doiurl{10.1214/18-AOS1692}
\end{barticle}
\endbibitem

\bibitem[\protect\citeauthoryear{Absil et~al.}{2008}]{Absil08}
\begin{bbook}
\bauthor{\bsnm{Absil}, \binits{P.-A.}},
\bauthor{\bsnm{Mahony}, \binits{R.}},
\bauthor{\bsnm{Sepulchre}, \binits{R.}}:
\bbtitle{Optimization Algorithms on Matrix Manifolds}
vol. \bseriesno{78}.
\bpublisher{Princeton University Press},
\blocation{Princeton, New Jersey}
(\byear{2008}).
\doiurl{10.1515/9781400830244}
\end{bbook}
\endbibitem

\bibitem[\protect\citeauthoryear{Benner et~al.}{2015}]{Benner15}
\begin{barticle}
\bauthor{\bsnm{Benner}, \binits{P.}},
\bauthor{\bsnm{Gugercin}, \binits{S.}},
\bauthor{\bsnm{Willcox}, \binits{K.}}:
\batitle{A survey of projection-based model reduction methods for parametric
  dynamical systems}.
\bjtitle{SIAM Review}
\bvolume{57}(\bissue{4}),
\bfpage{483}--\blpage{531}
(\byear{2015})
\doiurl{10.1137/130932715}
{\href{https://arxiv.org/abs/https://doi.org/10.1137/130932715}{{https://doi.org/10.1137/130932715}}}
\end{barticle}
\endbibitem

\bibitem[\protect\citeauthoryear{Cheng et~al.}{2016}]{Cheng16}
\begin{bbook}
\bauthor{\bsnm{Cheng}, \binits{G.}},
\bauthor{\bsnm{Ho}, \binits{J.}},
\bauthor{\bsnm{Salehian}, \binits{H.}},
\bauthor{\bsnm{Vemuri}, \binits{B.C.}}:
In: \beditor{\bsnm{Turaga}, \binits{P.K.}},
\beditor{\bsnm{Srivastava}, \binits{A.}} (eds.)
\bbtitle{Recursive Computation of the Fr{\'e}chet Mean on Non-positively Curved
  {R}iemannian Manifolds with Applications},
pp. \bfpage{21}--\blpage{43}.
\bpublisher{Springer},
\blocation{Cham}
(\byear{2016}).
\doiurl{10.1007/978-3-319-22957-7_2} .
\burl{https://doi.org/10.1007/978-3-319-22957-7_2}
\end{bbook}
\endbibitem

\bibitem[\protect\citeauthoryear{Thanwerdas and Pennec}{2021}]{Thanwerdas21}
\begin{bchapter}
\bauthor{\bsnm{Thanwerdas}, \binits{Y.}},
\bauthor{\bsnm{Pennec}, \binits{X.}}:
\bctitle{Geodesics and curvature of the quotient-affine metrics on full-rank
  correlation matrices}.
In: \beditor{\bsnm{Nielsen}, \binits{F.}},
\beditor{\bsnm{Barbaresco}, \binits{F.}} (eds.)
\bbtitle{Geometric Science of Information},
pp. \bfpage{93}--\blpage{102}.
\bpublisher{Springer},
\blocation{Cham}
(\byear{2021})
\end{bchapter}
\endbibitem

\bibitem[\protect\citeauthoryear{Thanwerdas and Pennec}{2023}]{Thanwerdas23}
\begin{barticle}
\bauthor{\bsnm{Thanwerdas}, \binits{Y.}},
\bauthor{\bsnm{Pennec}, \binits{X.}}:
\batitle{O(n)-invariant {R}iemannian metrics on {SPD} matrices}.
\bjtitle{Linear Algebra and its Applications}
\bvolume{661},
\bfpage{163}--\blpage{201}
(\byear{2023})
\doiurl{10.1016/j.laa.2022.12.009}
\end{barticle}
\endbibitem

\bibitem[\protect\citeauthoryear{Pennec}{2020}]{Pennec20a}
\begin{bchapter}
\bauthor{\bsnm{Pennec}, \binits{X.}}:
\bctitle{3 - manifold-valued image processing with {SPD} matrices}.
In: \beditor{\bsnm{Pennec}, \binits{X.}},
\beditor{\bsnm{Sommer}, \binits{S.}},
\beditor{\bsnm{Fletcher}, \binits{T.}} (eds.)
\bbtitle{{R}iemannian Geometric Statistics in Medical Image Analysis},
pp. \bfpage{75}--\blpage{134}.
\bpublisher{Academic Press},
\blocation{Cambridge, MA}
(\byear{2020}).
\doiurl{10.1016/B978-0-12-814725-2.00010-8} .
\burl{https://www.sciencedirect.com/science/article/pii/B9780128147252000108}
\end{bchapter}
\endbibitem

\bibitem[\protect\citeauthoryear{Moakher}{2006}]{Moakher06}
\begin{barticle}
\bauthor{\bsnm{Moakher}, \binits{M.}}:
\batitle{On the averaging of symmetric positive-definite tensors}.
\bjtitle{Journal of Elasticity}
\bvolume{82},
\bfpage{273}--\blpage{296}
(\byear{2006})
\doiurl{10.1007/s10659-005-9035-z}
\end{barticle}
\endbibitem

\bibitem[\protect\citeauthoryear{Pennec et~al.}{2006}]{Pennec20b}
\begin{barticle}
\bauthor{\bsnm{Pennec}, \binits{X.}},
\bauthor{\bsnm{Fillard}, \binits{P.}},
\bauthor{\bsnm{Ayache}, \binits{N.}}:
\batitle{A {R}iemannian framework for tensor computing}.
\bjtitle{International Journal of Computer Vision}
\bvolume{66},
\bfpage{41}--\blpage{66}
(\byear{2006})
\doiurl{10.1007/s11263-005-3222-z}
\end{barticle}
\endbibitem

\bibitem[\protect\citeauthoryear{Oseledets and
  Tyrtyshnikov}{2010}]{Oseledets2010}
\begin{barticle}
\bauthor{\bsnm{Oseledets}, \binits{I.}},
\bauthor{\bsnm{Tyrtyshnikov}, \binits{E.}}:
\batitle{{T}{T}-cross approximation for multidimensional arrays}.
\bjtitle{Linear Algebra and its Applications}
\bvolume{432}(\bissue{1}),
\bfpage{70}--\blpage{88}
(\byear{2010})
\doiurl{10.1016/j.laa.2009.07.024}
\end{barticle}
\endbibitem

\bibitem[\protect\citeauthoryear{Oseledets}{2011}]{Oseledets2011}
\begin{barticle}
\bauthor{\bsnm{Oseledets}, \binits{I.V.}}:
\batitle{Tensor-train decomposition}.
\bjtitle{SIAM Journal on Scientific Computing}
\bvolume{33}(\bissue{5}),
\bfpage{2295}--\blpage{2317}
(\byear{2011})
\doiurl{10.1137/090752286}
{\href{https://arxiv.org/abs/https://doi.org/10.1137/090752286}{{https://doi.org/10.1137/090752286}}}
\end{barticle}
\endbibitem

\bibitem[\protect\citeauthoryear{Swijsen et~al.}{2022}]{Swijsen2021}
\begin{botherref}
\oauthor{\bsnm{Swijsen}, \binits{L.}},
\oauthor{\bsnm{{van der Veken}}, \binits{J.}},
\oauthor{\bsnm{Vannieuwenhoven}, \binits{N.}}:
Tensor completion using geodesics on {S}egre manifolds.
Numerical Linear Algebra with Applications
\textbf{29}
(2022)
\doiurl{10.1002/nla.2446}
\end{botherref}
\endbibitem

\bibitem[\protect\citeauthoryear{Vandereycken et~al.}{2012}]{Vandereycken12}
\begin{barticle}
\bauthor{\bsnm{Vandereycken}, \binits{B.}},
\bauthor{\bsnm{Absil}, \binits{P.-A.}},
\bauthor{\bsnm{Vandewalle}, \binits{S.}}:
\batitle{{A {R}iemannian geometry with complete geodesics for the set of
  positive semidefinite matrices of fixed rank}}.
\bjtitle{IMA Journal of Numerical Analysis}
\bvolume{33}(\bissue{2}),
\bfpage{481}--\blpage{514}
(\byear{2012})
\doiurl{10.1093/imanum/drs006}
{\href{https://arxiv.org/abs/https://academic.oup.com/imajna/article-pdf/33/2/481/1845277/drs006.pdf}{{https://academic.oup.com/imajna/article-pdf/33/2/481/1845277/drs006.pdf}}}
\end{barticle}
\endbibitem

\bibitem[\protect\citeauthoryear{Koch and Lubich}{2007}]{Koch07}
\begin{barticle}
\bauthor{\bsnm{Koch}, \binits{O.}},
\bauthor{\bsnm{Lubich}, \binits{C.}}:
\batitle{Dynamical low‐rank approximation}.
\bjtitle{SIAM Journal on Matrix Analysis and Applications}
\bvolume{29}(\bissue{2}),
\bfpage{434}--\blpage{454}
(\byear{2007})
\doiurl{10.1137/050639703}
{\href{https://arxiv.org/abs/https://doi.org/10.1137/050639703}{{https://doi.org/10.1137/050639703}}}
\end{barticle}
\endbibitem

\bibitem[\protect\citeauthoryear{Soodhalter et~al.}{2020}]{Soodhalter20}
\begin{barticle}
\bauthor{\bsnm{Soodhalter}, \binits{K.M.}},
\bauthor{\bsnm{Sturler}, \binits{E.}},
\bauthor{\bsnm{Kilmer}, \binits{M.E.}}:
\batitle{A survey of subspace recycling iterative methods}.
\bjtitle{GAMM-Mitteilungen}
\bvolume{43}(\bissue{4}),
\bfpage{202000016}
(\byear{2020})
\doiurl{10.1002/gamm.202000016}
{\href{https://arxiv.org/abs/https://onlinelibrary.wiley.com/doi/pdf/10.1002/gamm.202000016}{{https://onlinelibrary.wiley.com/doi/pdf/10.1002/gamm.202000016}}}
\end{barticle}
\endbibitem

\bibitem[\protect\citeauthoryear{Zimmermann}{2022}]{Zimmermann22b}
\begin{botherref}
\oauthor{\bsnm{Zimmermann}, \binits{R.}}:
Manifold interpolation and model reduction
(2022).
\url{https://arxiv.org/abs/1902.06502}
\end{botherref}
\endbibitem

\bibitem[\protect\citeauthoryear{Yang}{1995}]{Yang95}
\begin{barticle}
\bauthor{\bsnm{Yang}, \binits{B.}}:
\batitle{Projection approximation subspace tracking}.
\bjtitle{IEEE Transactions on Signal Processing}
\bvolume{43}(\bissue{1}),
\bfpage{95}--\blpage{107}
(\byear{1995})
\doiurl{10.1109/78.365290}
\end{barticle}
\endbibitem

\bibitem[\protect\citeauthoryear{Mankovich et~al.}{2024}]{Manikovich24}
\begin{bchapter}
\bauthor{\bsnm{Mankovich}, \binits{N.}},
\bauthor{\bsnm{Camps-Valls}, \binits{G.}},
\bauthor{\bsnm{Birdal}, \binits{T.}}:
\bctitle{{ Fun with Flags: Robust Principal Directions via Flag Manifolds }}.
In: \bbtitle{2024 IEEE/CVF Conference on Computer Vision and Pattern
  Recognition (CVPR)},
pp. \bfpage{330}--\blpage{340}.
\bpublisher{IEEE Computer Society},
\blocation{Los Alamitos, CA, USA}
(\byear{2024}).
\doiurl{10.1109/CVPR52733.2024.00039} .
\burl{https://doi.ieeecomputersociety.org/10.1109/CVPR52733.2024.00039}
\end{bchapter}
\endbibitem

\bibitem[\protect\citeauthoryear{Buet and Pennec}{2023}]{Buet23}
\begin{botherref}
\oauthor{\bsnm{Buet}, \binits{B.}},
\oauthor{\bsnm{Pennec}, \binits{X.}}:
Flagfolds
(2023).
\url{https://arxiv.org/abs/2305.10583}
\end{botherref}
\endbibitem

\bibitem[\protect\citeauthoryear{De~Weer et~al.}{2022}]{DVLM2022}
\begin{barticle}
\bauthor{\bsnm{De~Weer}, \binits{T.}},
\bauthor{\bsnm{Vannieuwenhoven}, \binits{N.}},
\bauthor{\bsnm{Lammens}, \binits{N.}},
\bauthor{\bsnm{Meerbergen}, \binits{K.}}:
\batitle{The parametrized superelement approach for lattice joint modelling and
  simulation}.
\bjtitle{Computational Mechanics}
\bvolume{70}(\bissue{2}),
\bfpage{451}--\blpage{475}
(\byear{2022})
\doiurl{10.1007/s00466-022-02176-9}
\end{barticle}
\endbibitem

\bibitem[\protect\citeauthoryear{Lee}{2018}]{Lee18}
\begin{bbook}
\bauthor{\bsnm{Lee}, \binits{J.M.}}:
\bbtitle{Introduction to {R}iemannian Manifolds.}
\bsertitle{Graduate texts in mathematics: 176}.
\bpublisher{Springer},
\blocation{New York, NY}
(\byear{2018})
\end{bbook}
\endbibitem

\bibitem[\protect\citeauthoryear{Boumal}{2023}]{Boumal2023}
\begin{bbook}
\bauthor{\bsnm{Boumal}, \binits{N.}}:
\bbtitle{{An Introduction to Optimization on Smooth Manifolds}}.
\bpublisher{Cambridge University Press},
\blocation{Cambridge, England}
(\byear{2023}).
\doiurl{10.1017/9781009166164}
\end{bbook}
\endbibitem

\bibitem[\protect\citeauthoryear{Dyn and Sharon}{2017a}]{Dyn17a}
\begin{barticle}
\bauthor{\bsnm{Dyn}, \binits{N.}},
\bauthor{\bsnm{Sharon}, \binits{N.}}:
\batitle{Manifold-valued subdivision schemes based on geodesic inductive
  averaging}.
\bjtitle{Journal of Computational and Applied Mathematics}
\bvolume{311},
\bfpage{54}--\blpage{67}
(\byear{2017})
\doiurl{10.1016/j.cam.2016.07.008}
\end{barticle}
\endbibitem

\bibitem[\protect\citeauthoryear{Dyn and Sharon}{2017b}]{Dyn17b}
\begin{botherref}
\oauthor{\bsnm{Dyn}, \binits{N.}},
\oauthor{\bsnm{Sharon}, \binits{N.}}:
A global approach to the refinement of manifold data.
Mathematics of Computation
\textbf{86}
(2017)
\doiurl{10.1090/mcom/3087}
\end{botherref}
\endbibitem

\bibitem[\protect\citeauthoryear{Hardering and Wirth}{2021}]{Hardering21}
\begin{barticle}
\bauthor{\bsnm{Hardering}, \binits{H.}},
\bauthor{\bsnm{Wirth}, \binits{B.}}:
\batitle{Quartic {$L^p$}-convergence of cubic {R}iemannian splines}.
\bjtitle{IMA Journal of Numerical Analysis}
\bvolume{42}(\bissue{4}),
\bfpage{3360}--\blpage{3385}
(\byear{2021})
\doiurl{10.1093/imanum/drab077}
{\href{https://arxiv.org/abs/https://academic.oup.com/imajna/article-pdf/42/4/3360/46323898/drab077.pdf}{{https://academic.oup.com/imajna/article-pdf/42/4/3360/46323898/drab077.pdf}}}
\end{barticle}
\endbibitem

\bibitem[\protect\citeauthoryear{Heeren et~al.}{2019}]{Heeren19}
\begin{barticle}
\bauthor{\bsnm{Heeren}, \binits{B.}},
\bauthor{\bsnm{Rumpf}, \binits{M.}},
\bauthor{\bsnm{Wirth}, \binits{B.}}:
\batitle{Variational time discretization of {R}iemannian splines}.
\bjtitle{IMA Journal of Numerical Analysis}
\bvolume{39}(\bissue{1}),
\bfpage{61}--\blpage{104}
(\byear{2019})
\doiurl{10.1093/imanum/drx077}
{\href{https://arxiv.org/abs/https://academic.oup.com/imajna/article-pdf/39/1/61/27580118/drx077.pdf}{{https://academic.oup.com/imajna/article-pdf/39/1/61/27580118/drx077.pdf}}}
\end{barticle}
\endbibitem

\bibitem[\protect\citeauthoryear{Zhang and Noakes}{2019}]{Zhang19}
\begin{barticle}
\bauthor{\bsnm{Zhang}, \binits{E.}},
\bauthor{\bsnm{Noakes}, \binits{L.}}:
\batitle{The cubic de {C}asteljau construction and {R}iemannian cubics}.
\bjtitle{Computer Aided Geometric Design}
\bvolume{75},
\bfpage{101789}
(\byear{2019})
\doiurl{10.1016/j.cagd.2019.101789}
\end{barticle}
\endbibitem

\bibitem[\protect\citeauthoryear{Bergmann and Gousenbourger}{2018}]{Bergmann18}
\begin{botherref}
\oauthor{\bsnm{Bergmann}, \binits{R.}},
\oauthor{\bsnm{Gousenbourger}, \binits{P.-Y.}}:
A variational model for data fitting on manifolds by minimizing the
  acceleration of a {B}ézier curve.
Frontiers in Applied Mathematics and Statistics
\textbf{4}
(2018)
\doiurl{10.3389/fams.2018.00059}
\end{botherref}
\endbibitem

\bibitem[\protect\citeauthoryear{Gousenbourger et~al.}{2019}]{Gousenbourger19}
\begin{barticle}
\bauthor{\bsnm{Gousenbourger}, \binits{P.-Y.}},
\bauthor{\bsnm{Massart}, \binits{E.}},
\bauthor{\bsnm{Absil}, \binits{P.-A.}}:
\batitle{Data fitting on manifolds with composite {B}ézier-like curves and
  blended cubic splines}.
\bjtitle{Journal of Mathematical Imaging and Vision}
\bvolume{61},
\bfpage{645}--\blpage{671}
(\byear{2019})
\doiurl{10.1007/s10851-018-0865-2}
\end{barticle}
\endbibitem

\bibitem[\protect\citeauthoryear{Sharon et~al.}{2023}]{Sharon23}
\begin{barticle}
\bauthor{\bsnm{Sharon}, \binits{N.}},
\bauthor{\bsnm{Cohen}, \binits{R.S.}},
\bauthor{\bsnm{Wendland}, \binits{H.}}:
\batitle{On multiscale quasi-interpolation of scattered scalar- and
  manifold-valued functions}.
\bjtitle{SIAM Journal on Scientific Computing}
\bvolume{45}(\bissue{5}),
\bfpage{2458}--\blpage{2482}
(\byear{2023})
\doiurl{10.1137/22M1528306}
{\href{https://arxiv.org/abs/https://doi.org/10.1137/22M1528306}{{https://doi.org/10.1137/22M1528306}}}
\end{barticle}
\endbibitem

\bibitem[\protect\citeauthoryear{Petersen and M{\"u}ller}{2019}]{Petersen19}
\begin{barticle}
\bauthor{\bsnm{Petersen}, \binits{A.}},
\bauthor{\bsnm{M{\"u}ller}, \binits{H.-G.}}:
\batitle{{Fréchet regression for random objects with Euclidean predictors}}.
\bjtitle{The Annals of Statistics}
\bvolume{47}(\bissue{2}),
\bfpage{691}--\blpage{719}
(\byear{2019})
\doiurl{10.1214/17-AOS1624}
\end{barticle}
\endbibitem

\bibitem[\protect\citeauthoryear{Grohs et~al.}{2017}]{Grohs16}
\begin{barticle}
\bauthor{\bsnm{Grohs}, \binits{P.}},
\bauthor{\bsnm{Sprecher}, \binits{M.}},
\bauthor{\bsnm{Yu}, \binits{T.}}:
\batitle{Scattered manifold-valued data approximation}.
\bjtitle{Numer. Math.}
\bvolume{135}(\bissue{4}),
\bfpage{987}--\blpage{1010}
(\byear{2017})
\doiurl{10.1007/s00211-016-0823-0}
\end{barticle}
\endbibitem

\bibitem[\protect\citeauthoryear{Cornea et~al.}{2016}]{Cornea16}
\begin{barticle}
\bauthor{\bsnm{Cornea}, \binits{E.}},
\bauthor{\bsnm{Zhu}, \binits{H.}},
\bauthor{\bsnm{Kim}, \binits{P.}},
\bauthor{\bsnm{Ibrahim}, \binits{J.G.}}:
\batitle{Regression models on {R}iemannian symmetric spaces}.
\bjtitle{Journal of the Royal Statistical Society Series B: Statistical
  Methodology}
\bvolume{79}(\bissue{2}),
\bfpage{463}--\blpage{482}
(\byear{2016})
\doiurl{10.1111/rssb.12169}
{\href{https://arxiv.org/abs/https://academic.oup.com/jrsssb/article-pdf/79/2/463/49232126/rssb12169-sup-0001-supinfo.pdf}{{https://academic.oup.com/jrsssb/article-pdf/79/2/463/49232126/rssb12169-sup-0001-supinfo.pdf}}}
\end{barticle}
\endbibitem

\bibitem[\protect\citeauthoryear{Hinkle et~al.}{2014}]{Hinkle14}
\begin{barticle}
\bauthor{\bsnm{Hinkle}, \binits{J.}},
\bauthor{\bsnm{Fletcher}, \binits{P.T.}},
\bauthor{\bsnm{Joshi}, \binits{S.}}:
\batitle{Intrinsic polynomials for regression on {R}iemannian manifolds}.
\bjtitle{Journal of Mathematical Imaging and Vision}
\bvolume{50},
\bfpage{32}--\blpage{52}
(\byear{2014})
\doiurl{10.1007/s10851-013-0489-5}
\end{barticle}
\endbibitem

\bibitem[\protect\citeauthoryear{Thomas~Fletcher}{2013}]{Fletcher13}
\begin{barticle}
\bauthor{\bsnm{Thomas~Fletcher}, \binits{P.}}:
\batitle{Geodesic regression and the theory of least squares on {R}iemannian
  manifolds}.
\bjtitle{International Journal of Computer Vision}
\bvolume{105},
\bfpage{171}--\blpage{185}
(\byear{2013})
\doiurl{10.1007/s11263-012-0591-y}
\end{barticle}
\endbibitem

\bibitem[\protect\citeauthoryear{Séguin}{2024}]{Seguin24}
\begin{botherref}
\oauthor{\bsnm{Séguin}, \binits{K.D.} \bsuffix{Axel}}:
{H}ermite interpolation with retractions on manifolds.
BIT Numerical Mathematics
\textbf{64}
(2024)
\doiurl{10.1007/s10543-024-01023-y}
\end{botherref}
\endbibitem

\bibitem[\protect\citeauthoryear{{Ben-Zion Vardi} et~al.}{2024}]{Vardi24}
\begin{barticle}
\bauthor{\bsnm{{Ben-Zion Vardi}}, \binits{H.}},
\bauthor{\bsnm{Dyn}, \binits{N.}},
\bauthor{\bsnm{Sharon}, \binits{N.}}:
\batitle{{H}ermite subdivision schemes for manifold-valued {H}ermite data}.
\bjtitle{Computer Aided Geometric Design}
\bvolume{111},
\bfpage{102342}
(\byear{2024})
\doiurl{10.1016/j.cagd.2024.102342}
\end{barticle}
\endbibitem

\bibitem[\protect\citeauthoryear{Moosm\"{u}ller}{2016}]{Moosmuller16}
\begin{barticle}
\bauthor{\bsnm{Moosm\"{u}ller}, \binits{C.}}:
\batitle{${C}^1$ analysis of {H}ermite subdivision schemes on manifolds}.
\bjtitle{SIAM Journal on Numerical Analysis}
\bvolume{54}(\bissue{5}),
\bfpage{3003}--\blpage{3031}
(\byear{2016})
\doiurl{10.1137/15M1033459}
{\href{https://arxiv.org/abs/https://doi.org/10.1137/15M1033459}{{https://doi.org/10.1137/15M1033459}}}
\end{barticle}
\endbibitem

\bibitem[\protect\citeauthoryear{Moosm\"uller}{2017}]{Moosmuller17}
\begin{barticle}
\bauthor{\bsnm{Moosm\"uller}, \binits{C.}}:
\batitle{{H}ermite subdivision on manifolds via parallel transport}.
\bjtitle{Advances in Computational Mathematics}
\bvolume{43},
\bfpage{1059}--\blpage{1074}
(\byear{2017})
\doiurl{10.1007/s10444-017-9516-1}
\end{barticle}
\endbibitem

\bibitem[\protect\citeauthoryear{Gebhardt et~al.}{2023}]{Gebhardt23}
\begin{barticle}
\bauthor{\bsnm{Gebhardt}, \binits{C.G.}},
\bauthor{\bsnm{Schubert}, \binits{J.}},
\bauthor{\bsnm{Steinbach}, \binits{M.C.}}:
\batitle{Long-time principal geodesic analysis in director-based dynamics of
  hybrid mechanical systems}.
\bjtitle{Communications in Nonlinear Science and Numerical Simulation}
\bvolume{122},
\bfpage{107240}
(\byear{2023})
\doiurl{10.1016/j.cnsns.2023.107240}
\end{barticle}
\endbibitem

\bibitem[\protect\citeauthoryear{Curry et~al.}{2019}]{Curry19}
\begin{barticle}
\bauthor{\bsnm{Curry}, \binits{C.}},
\bauthor{\bsnm{Marsland}, \binits{S.}},
\bauthor{\bsnm{McLachlan}, \binits{R.I.}}:
\batitle{Principal symmetric space analysis}.
\bjtitle{Journal of Computational Dynamics}
\bvolume{6}(\bissue{2}),
\bfpage{251}--\blpage{276}
(\byear{2019})
\doiurl{10.3934/jcd.2019013}
\end{barticle}
\endbibitem

\bibitem[\protect\citeauthoryear{Lazar and Lin}{2017}]{Lazar17}
\begin{barticle}
\bauthor{\bsnm{Lazar}, \binits{D.}},
\bauthor{\bsnm{Lin}, \binits{L.}}:
\batitle{Scale and curvature effects in principal geodesic analysis}.
\bjtitle{Journal of Multivariate Analysis}
\bvolume{153},
\bfpage{64}--\blpage{82}
(\byear{2017})
\doiurl{10.1016/j.jmva.2016.09.009}
\end{barticle}
\endbibitem

\bibitem[\protect\citeauthoryear{Chakraborty et~al.}{2016}]{Chakraborty16}
\begin{bchapter}
\bauthor{\bsnm{Chakraborty}, \binits{R.}},
\bauthor{\bsnm{Seo}, \binits{D.}},
\bauthor{\bsnm{Vemuri}, \binits{B.C.}}:
\bctitle{An efficient exact-{P}{G}{A} algorithm for constant curvature
  manifolds}.
In: \bbtitle{Proceedings of the IEEE Conference on Computer Vision and Pattern
  Recognition (CVPR)}
(\byear{2016})
\end{bchapter}
\endbibitem

\bibitem[\protect\citeauthoryear{Diepeveen et~al.}{2023}]{Diepeveen23}
\begin{botherref}
\oauthor{\bsnm{Diepeveen}, \binits{W.}},
\oauthor{\bsnm{Chew}, \binits{J.}},
\oauthor{\bsnm{Needell}, \binits{D.}}:
Curvature corrected tangent space-based approximation of manifold-valued data
(2023).
\url{https://arxiv.org/abs/2306.00507}
\end{botherref}
\endbibitem

\bibitem[\protect\citeauthoryear{Diepeveen}{2024}]{Diepeveen24}
\begin{botherref}
\oauthor{\bsnm{Diepeveen}, \binits{W.}}:
Pulling back symmetric {R}iemannian geometry for data analysis
(2024).
\url{https://arxiv.org/abs/2403.06612}
\end{botherref}
\endbibitem

\bibitem[\protect\citeauthoryear{Jacobsson}{2024}]{manifactor}
\begin{botherref}
\oauthor{\bsnm{Jacobsson}, \binits{S.}}:
Mani{F}actor.jl: Approximating maps into manifolds.
\url{https://gitlab.kuleuven.be/numa/software/ManiFactor}
(2024)
\end{botherref}
\endbibitem

\bibitem[\protect\citeauthoryear{Lee}{2013}]{Lee13}
\begin{bbook}
\bauthor{\bsnm{Lee}, \binits{J.M.}}:
\bbtitle{Introduction to Smooth Manifolds}.
\bsertitle{Graduate texts in mathematics: 218}.
\bpublisher{Springer},
\blocation{New York, NY}
(\byear{2013})
\end{bbook}
\endbibitem

\bibitem[\protect\citeauthoryear{Cheeger and Ebin}{2008}]{Cheeger08}
\begin{bbook}
\bauthor{\bsnm{Cheeger}, \binits{J.}},
\bauthor{\bsnm{Ebin}, \binits{D.G.}}:
\bbtitle{Comparison Theorems in {R}iemannian Geometry}.
\bpublisher{AMS Chelsea Publishing},
\blocation{New York, NY}
(\byear{2008})
\end{bbook}
\endbibitem

\bibitem[\protect\citeauthoryear{Reid and Szendroi}{2005}]{Reid05}
\begin{bbook}
\bauthor{\bsnm{Reid}, \binits{M.}},
\bauthor{\bsnm{Szendroi}, \binits{B.}}:
\bbtitle{Geometry and Topology}.
\bpublisher{Cambridge University Press},
\blocation{Cambridge, England}
(\byear{2005}).
\doiurl{10.1017/CBO9780511807510}
\end{bbook}
\endbibitem

\bibitem[\protect\citeauthoryear{Absil and Oseledets}{2015}]{Absil15}
\begin{barticle}
\bauthor{\bsnm{Absil}, \binits{P.-A.}},
\bauthor{\bsnm{Oseledets}, \binits{I.V.}}:
\batitle{Low-rank retractions: a survey and new results}.
\bjtitle{Computational Optimization and Applications}
\bvolume{62},
\bfpage{5}--\blpage{29}
(\byear{2015})
\doiurl{10.1007/s10589-014-9714-4}
\end{barticle}
\endbibitem

\bibitem[\protect\citeauthoryear{Rice}{1966}]{Rice66}
\begin{barticle}
\bauthor{\bsnm{Rice}, \binits{J.R.}}:
\batitle{A theory of condition}.
\bjtitle{SIAM Journal on Numerical Analysis}
\bvolume{3}(\bissue{2}),
\bfpage{287}--\blpage{310}
(\byear{1966})
\doiurl{10.1137/0703023}
{\href{https://arxiv.org/abs/https://doi.org/10.1137/0703023}{{https://doi.org/10.1137/0703023}}}
\end{barticle}
\endbibitem

\bibitem[\protect\citeauthoryear{Karcher}{1977}]{Karcher1977}
\begin{barticle}
\bauthor{\bsnm{Karcher}, \binits{H.}}:
\batitle{{R}iemannian center of mass and mollifier smoothing}.
\bjtitle{Communications on Pure and Applied Mathematics}
\bvolume{30}(\bissue{5}),
\bfpage{509}--\blpage{541}
(\byear{1977})
\doiurl{10.1002/cpa.3160300502}
\end{barticle}
\endbibitem

\bibitem[\protect\citeauthoryear{Petersen}{2016}]{Petersen16}
\begin{bbook}
\bauthor{\bsnm{Petersen}, \binits{P.}}:
\bbtitle{{R}iemannian Geometry}.
\bpublisher{Springer},
\blocation{New York, NY}
(\byear{2016}).
\doiurl{10.1007/978-3-319-26654-1}
\end{bbook}
\endbibitem

\bibitem[\protect\citeauthoryear{Arnaudon et~al.}{2013}]{Arnaudon13}
\begin{bbook}
\bauthor{\bsnm{Arnaudon}, \binits{M.}},
\bauthor{\bsnm{Barbaresco}, \binits{F.}},
\bauthor{\bsnm{Yang}, \binits{L.}}:
In: \beditor{\bsnm{Nielsen}, \binits{F.}},
\beditor{\bsnm{Bhatia}, \binits{R.}} (eds.)
\bbtitle{Medians and Means in {R}iemannian Geometry: Existence, Uniqueness and
  Computation},
pp. \bfpage{169}--\blpage{197}.
\bpublisher{Springer},
\blocation{Berlin, Heidelberg}
(\byear{2013}).
\doiurl{10.1007/978-3-642-30232-9_8} .
\burl{https://doi.org/10.1007/978-3-642-30232-9_8}
\end{bbook}
\endbibitem

\bibitem[\protect\citeauthoryear{Nielsen and Bhatia}{2013}]{Nielsen13}
\begin{bbook}
\bauthor{\bsnm{Nielsen}, \binits{F.}},
\bauthor{\bsnm{Bhatia}, \binits{R.}}:
\bbtitle{Matrix Information Geometry}.
\bpublisher{Springer},
\blocation{Berlin, Heidelberg}
(\byear{2013}).
\doiurl{10.1007/978-3-642-30232-9}
\end{bbook}
\endbibitem

\bibitem[\protect\citeauthoryear{Chakraborty and Vemuri}{2020}]{Chakraborty20}
\begin{bchapter}
\bauthor{\bsnm{Chakraborty}, \binits{R.}},
\bauthor{\bsnm{Vemuri}, \binits{B.C.}}:
\bctitle{Efficient recursive estimation of the {R}iemannian barycenter on the
  hypersphere and the special orthogonal group with applications}.
In: \beditor{\bsnm{Pennec}, \binits{X.}},
\beditor{\bsnm{Sommer}, \binits{S.}},
\beditor{\bsnm{Fletcher}, \binits{T.}} (eds.)
\bbtitle{{R}iemannian Geometric Statistics in Medical Image Analysis},
pp. \bfpage{273}--\blpage{297}.
\bpublisher{Academic Press},
\blocation{Cambridge, MA}
(\byear{2020}).
\doiurl{10.1016/B978-0-12-814725-2.00015-7} .
\burl{https://www.sciencedirect.com/science/article/pii/B9780128147252000157}
\end{bchapter}
\endbibitem

\bibitem[\protect\citeauthoryear{Chakraborty and Vemuri}{2015}]{Chakraborty15}
\begin{bchapter}
\bauthor{\bsnm{Chakraborty}, \binits{R.}},
\bauthor{\bsnm{Vemuri}, \binits{B.C.}}:
\bctitle{Recursive {F}réchet mean computation on the {G}rassmannian and its
  applications to computer vision}.
In: \bbtitle{2015 IEEE International Conference on Computer Vision (ICCV)},
pp. \bfpage{4229}--\blpage{4237}
(\byear{2015}).
\doiurl{10.1109/ICCV.2015.481}
\end{bchapter}
\endbibitem

\bibitem[\protect\citeauthoryear{Schultz}{1969}]{Schultz1969}
\begin{barticle}
\bauthor{\bsnm{Schultz}, \binits{M.H.}}:
\batitle{{$L^\infty$}-multivariate approximation theory}.
\bjtitle{SIAM Journal on Numerical Analysis}
\bvolume{6}(\bissue{2}),
\bfpage{161}--\blpage{183}
(\byear{1969})
\end{barticle}
\endbibitem

\bibitem[\protect\citeauthoryear{Mason}{1980}]{Mason1980}
\begin{barticle}
\bauthor{\bsnm{Mason}, \binits{J.C.}}:
\batitle{Near-best multivariate approximation by {F}ourier series, {C}hebyshev
  series and {C}hebyshev interpolation}.
\bjtitle{Journal of Approximation Theory}
\bvolume{28}(\bissue{4}),
\bfpage{349}--\blpage{358}
(\byear{1980})
\doiurl{10.1016/0021-9045(80)90069-6}
\end{barticle}
\endbibitem

\bibitem[\protect\citeauthoryear{Shi and Townsend}{2021}]{ST2021}
\begin{barticle}
\bauthor{\bsnm{Shi}, \binits{T.}},
\bauthor{\bsnm{Townsend}, \binits{A.}}:
\batitle{On the compressibility of tensors}.
\bjtitle{SIAM Journal on Matrix Analysis and Applications}
\bvolume{42}(\bissue{1}),
\bfpage{275}--\blpage{298}
(\byear{2021})
\doiurl{10.1137/20m1316639}
\end{barticle}
\endbibitem

\bibitem[\protect\citeauthoryear{Breiding et~al.}{2023}]{BGMV2023}
\begin{barticle}
\bauthor{\bsnm{Breiding}, \binits{P.}},
\bauthor{\bsnm{Gesmundo}, \binits{F.}},
\bauthor{\bsnm{Micha\l{}ek}, \binits{M.}},
\bauthor{\bsnm{Vannieuwenhoven}, \binits{N.}}:
\batitle{Algebraic compressed sensing}.
\bjtitle{Applied and Computational Harmonic Analysis}
\bvolume{65},
\bfpage{374}--\blpage{406}
(\byear{2023})
\doiurl{10.1016/j.acha.2023.03.006}
\end{barticle}
\endbibitem

\bibitem[\protect\citeauthoryear{Rong et~al.}{2021}]{RWX2021}
\begin{barticle}
\bauthor{\bsnm{Rong}, \binits{Y.}},
\bauthor{\bsnm{Wang}, \binits{Y.}},
\bauthor{\bsnm{Xu}, \binits{Z.}}:
\batitle{Almost everywhere injectivity conditions for the matrix recovery
  problem}.
\bjtitle{Applied and Computational Harmonic Analysis}
\bvolume{50},
\bfpage{386}--\blpage{400}
(\byear{2021})
\doiurl{10.1016/j.acha.2019.09.002}
\end{barticle}
\endbibitem

\bibitem[\protect\citeauthoryear{Ballani et~al.}{2013}]{BGK2013}
\begin{barticle}
\bauthor{\bsnm{Ballani}, \binits{J.}},
\bauthor{\bsnm{Grasedyck}, \binits{L.}},
\bauthor{\bsnm{Kluge}, \binits{M.}}:
\batitle{Black box approximation of tensors in hierarchical {T}ucker format}.
\bjtitle{Linear Algebra and its Applications}
\bvolume{438}(\bissue{2}),
\bfpage{639}--\blpage{657}
(\byear{2013})
\end{barticle}
\endbibitem

\bibitem[\protect\citeauthoryear{Grasedyck et~al.}{2015}]{GKK2015}
\begin{barticle}
\bauthor{\bsnm{Grasedyck}, \binits{L.}},
\bauthor{\bsnm{Kluge}, \binits{M.}},
\bauthor{\bsnm{Kr\"amer}, \binits{S.}}:
\batitle{Variants of alternating least squares tensor completion in the tensor
  train format}.
\bjtitle{SIAM Journal on Scientific Computing}
\bvolume{37}(\bissue{5}),
\bfpage{2424}--\blpage{2450}
(\byear{2015})
\end{barticle}
\endbibitem

\bibitem[\protect\citeauthoryear{Steinlechner}{2016}]{Steinlechner2016}
\begin{barticle}
\bauthor{\bsnm{Steinlechner}, \binits{M.}}:
\batitle{{R}iemannian optimization for high-dimensional tensor completion}.
\bjtitle{SIAM Journal on Scientific Computing}
\bvolume{38}(\bissue{5}),
\bfpage{461}--\blpage{484}
(\byear{2016})
\end{barticle}
\endbibitem

\bibitem[\protect\citeauthoryear{Kressner et~al.}{2014}]{KSV2014}
\begin{barticle}
\bauthor{\bsnm{Kressner}, \binits{D.}},
\bauthor{\bsnm{Steinlechner}, \binits{M.}},
\bauthor{\bsnm{Vandereycken}, \binits{B.}}:
\batitle{Low-rank tensor completion by {R}iemannian optimization}.
\bjtitle{BIT Numerical Mathematics}
\bvolume{54}(\bissue{2}),
\bfpage{447}--\blpage{468}
(\byear{2014})
\end{barticle}
\endbibitem

\bibitem[\protect\citeauthoryear{Signoretto et~al.}{2013}]{STDdLS2013}
\begin{barticle}
\bauthor{\bsnm{Signoretto}, \binits{M.}},
\bauthor{\bsnm{{Tran Dinh}}, \binits{Q.}},
\bauthor{\bsnm{{De Lathauwer}}, \binits{L.}},
\bauthor{\bsnm{Suykens}, \binits{J.A.K.}}:
\batitle{Learning with tensors: a framework based on convex optimization and
  spectral regularization}.
\bjtitle{Machine Learning}
\bvolume{94}(\bissue{3}),
\bfpage{303}--\blpage{351}
(\byear{2013})
\doiurl{10.1007/s10994-013-5366-3}
\end{barticle}
\endbibitem

\bibitem[\protect\citeauthoryear{{De Lathauwer} et~al.}{2000}]{Lathauwer2000}
\begin{barticle}
\bauthor{\bsnm{{De Lathauwer}}, \binits{L.}},
\bauthor{\bsnm{{De Moor}}, \binits{B.}},
\bauthor{\bsnm{Vandewalle}, \binits{J.}}:
\batitle{A multilinear singular value decomposition}.
\bjtitle{SIAM Journal on Matrix Analysis and Applications}
\bvolume{21},
\bfpage{1253}--\blpage{1278}
(\byear{2000})
\end{barticle}
\endbibitem

\bibitem[\protect\citeauthoryear{Vannieuwenhoven
  et~al.}{2012}]{Vannieuwenhoven2012}
\begin{barticle}
\bauthor{\bsnm{Vannieuwenhoven}, \binits{N.}},
\bauthor{\bsnm{Vandebril}, \binits{R.}},
\bauthor{\bsnm{Meerbergen}, \binits{K.}}:
\batitle{A new truncation strategy for the higher-order singular value
  decomposition}.
\bjtitle{SIAM Journal on Scientific Computing}
\bvolume{34}(\bissue{2}),
\bfpage{1027}--\blpage{1052}
(\byear{2012})
\doiurl{10.1137/110836067}
{\href{https://arxiv.org/abs/https://doi.org/10.1137/110836067}{{https://doi.org/10.1137/110836067}}}
\end{barticle}
\endbibitem

\bibitem[\protect\citeauthoryear{Axen et~al.}{2023}]{Axen2023}
\begin{botherref}
\oauthor{\bsnm{Axen}, \binits{S.D.}},
\oauthor{\bsnm{Baran}, \binits{M.}},
\oauthor{\bsnm{Bergmann}, \binits{R.}},
\oauthor{\bsnm{Rzecki}, \binits{K.}}:
Manifolds.jl: An Extensible {J}ulia Framework for Data Analysis on Manifolds.
Association for Computing Machinery (ACM)
(2023).
\doiurl{10.1145/3618296}
\end{botherref}
\endbibitem

\bibitem[\protect\citeauthoryear{Olver and Townsend}{2014}]{Olver14}
\begin{bchapter}
\bauthor{\bsnm{Olver}, \binits{S.}},
\bauthor{\bsnm{Townsend}, \binits{A.}}:
\bctitle{A practical framework for infinite-dimensional linear algebra}.
In: \bbtitle{Proceedings of the 1st Workshop for High Performance Technical
  Computing in Dynamic Languages -- HPTCDL `14}.
\bpublisher{{IEEE}},
\blocation{Los Angeles}
(\byear{2014})
\end{bchapter}
\endbibitem

\bibitem[\protect\citeauthoryear{Periša}{2023}]{Perisa23}
\begin{botherref}
\oauthor{\bsnm{Periša}, \binits{L.}}:
TensorToolbox.jl: Julia package for tensors.
\url{https://github.com/lanaperisa/TensorToolbox.jl}
(2023)
\end{botherref}
\endbibitem

\bibitem[\protect\citeauthoryear{Ceruti and Christian}{2022}]{Ceruti22}
\begin{botherref}
\oauthor{\bsnm{Ceruti}, \binits{G.}},
\oauthor{\bsnm{Christian}, \binits{L.}}:
An unconventional robust integrator for dynamical low-rank approximation.
BIT Numerical Mathematics
\textbf{62}(1)
(2022)
\doiurl{10.1007/s10543-021-00873-0}
\end{botherref}
\endbibitem

\bibitem[\protect\citeauthoryear{Andruchow et~al.}{2014}]{Andruchow14}
\begin{barticle}
\bauthor{\bsnm{Andruchow}, \binits{E.}},
\bauthor{\bsnm{Larotonda}, \binits{G.}},
\bauthor{\bsnm{Recht}, \binits{L.}},
\bauthor{\bsnm{Varela}, \binits{A.}}:
\batitle{The left invariant metric in the general linear group}.
\bjtitle{Journal of Geometry and Physics}
\bvolume{86},
\bfpage{241}--\blpage{257}
(\byear{2014})
\doiurl{10.1016/j.geomphys.2014.08.009}
\end{barticle}
\endbibitem

\bibitem[\protect\citeauthoryear{Arvanitogeorgos}{2003}]{Arvanitogeorgos03}
\begin{bbook}
\bauthor{\bsnm{Arvanitogeorgos}, \binits{A.}}:
\bbtitle{An Introduction to {L}ie Groups and the Geometry of Homogeneous
  Spaces}.
\bpublisher{American Mathematical Society},
\blocation{Providence, Rhode Island}
(\byear{2003})
\end{bbook}
\endbibitem

\bibitem[\protect\citeauthoryear{Zimmermann}{2017}]{Zimmermann17}
\begin{barticle}
\bauthor{\bsnm{Zimmermann}, \binits{R.}}:
\batitle{A matrix-algebraic algorithm for the {R}iemannian logarithm on the
  {S}tiefel manifold under the canonical metric}.
\bjtitle{SIAM Journal on Matrix Analysis and Applications}
\bvolume{38}(\bissue{2}),
\bfpage{322}--\blpage{342}
(\byear{2017})
\doiurl{10.1137/16M1074485}
{\href{https://arxiv.org/abs/https://doi.org/10.1137/16M1074485}{{https://doi.org/10.1137/16M1074485}}}
\end{barticle}
\endbibitem

\bibitem[\protect\citeauthoryear{Bendokat et~al.}{2023}]{Bendokat23}
\begin{botherref}
\oauthor{\bsnm{Bendokat}, \binits{T.}},
\oauthor{\bsnm{Zimmermann}, \binits{R.}},
\oauthor{\bsnm{Absil}, \binits{P.-A.}}:
A {G}rassmann Manifold Handbook: Basic Geometry and Computational Aspects
(2023).
\url{https://arxiv.org/abs/2011.13699}
\end{botherref}
\endbibitem

\bibitem[\protect\citeauthoryear{Bhatia}{2007}]{Bhatia07}
\begin{bbook}
\bauthor{\bsnm{Bhatia}, \binits{R.}}:
\bbtitle{Positive Definite Matrices}.
\bpublisher{Princeton University Press},
\blocation{Princeton, New Jersey}
(\byear{2007}).
\burl{http://www.jstor.org/stable/j.ctt7rxv2}
Accessed 2024-07-03
\end{bbook}
\endbibitem

\bibitem[\protect\citeauthoryear{Massart and Absil}{2020}]{Massart20}
\begin{barticle}
\bauthor{\bsnm{Massart}, \binits{E.}},
\bauthor{\bsnm{Absil}, \binits{P.-A.}}:
\batitle{Quotient geometry with simple geodesics for the manifold of fixed-rank
  positive-semidefinite matrices}.
\bjtitle{SIAM Journal on Matrix Analysis and Applications}
\bvolume{41}(\bissue{1}),
\bfpage{171}--\blpage{198}
(\byear{2020})
\doiurl{10.1137/18M1231389}
{\href{https://arxiv.org/abs/https://doi.org/10.1137/18M1231389}{{https://doi.org/10.1137/18M1231389}}}
\end{barticle}
\endbibitem

\bibitem[\protect\citeauthoryear{Harris}{1992}]{Harris1992}
\begin{bbook}
\bauthor{\bsnm{Harris}, \binits{J.}}:
\bbtitle{Algebraic Geometry, A First Course}.
\bsertitle{Graduate Text in Mathematics},
vol. \bseriesno{133},
p. \bfpage{328}.
\bpublisher{Springer},
\blocation{New York, NY}
(\byear{1992})
\end{bbook}
\endbibitem

\bibitem[\protect\citeauthoryear{Swijsen}{2022}]{Swijsen22}
\begin{botherref}
\oauthor{\bsnm{Swijsen}, \binits{L.}}:
Tensor decompositions and {R}iemannian optimization.
PhD thesis,
KU Leuven
(2022)
\end{botherref}
\endbibitem

\end{thebibliography}

\end{document}